\documentclass[oneside,article]{memoir}
\usepackage{amsmath}         
\usepackage{amsthm,thmtools,thm-restate} 
\usepackage{enumitem}        

\usepackage{tikz}
  \usetikzlibrary{matrix,arrows,positioning}
  
\usepackage{tikz-cd}

\usepackage{float} 
\usepackage{caption, subcaption} 

\usepackage[nobysame,alphabetic,initials]{amsrefs} 
\usepackage[hidelinks]{hyperref}                   
\usepackage[capitalize,nosort]{cleveref}           

\usepackage{amssymb}
\usepackage[mathscr]{euscript}
\usepackage{relsize}
\usepackage{stmaryrd}
\usepackage{stackengine}

\usepackage{indentfirst} 
\usepackage{multicol}

\usepackage[inline,nomargin]{fixme} 
  \fxsetup{targetlayout=color}

\usepackage{wrapfig}
\usepackage{float} 

\isopage[12]

\setlrmargins{*}{*}{1}
\checkandfixthelayout

\counterwithout{section}{chapter}
\usepackage{appendix}

  \newcommand{\adj}{\dashv}

  \newcommand{\Ho}{\operatorname{Ho}}

  \newcommand{\op}{{\mathord\mathrm{op}}}

  \newcommand{\push}{\cup}

  \newcommand{\colim}{\operatorname*{colim}}

  \newcommand{\slice}{\mathbin\downarrow}

  \newcommand{\Ex}{\operatorname{Ex}}

  \newcommand{\nerve}{\operatorname{N}}

  \newcommand{\Sd}{\operatorname{Sd}}

  \newcommand{\bd}{\partial}

  \newcommand{\id}[1][]{\operatorname{id}_{#1}}
  
  \newcommand{\simp}[1]{\mathord\Delta^{#1}}

  \newcommand{\horn}[2]{\Lambda^{#1}_{#2}}

  \newcommand{\uvar}{\mathord{\relbar}}

  \renewcommand{\hat}{\widehat}

  \newcommand{\ncat}[1]{\mathsf{#1}}

  \newcommand{\sSet}{\ncat{sSet}}

  \newcommand{\from}{\colon}

  \newcommand{\ito}{\hookrightarrow}

  \declaretheorem[style=definition,within=section]{definition}
  \declaretheorem[style=definition,numberlike=definition]{example}
  \declaretheorem[style=definition,numberlike=definition]{remark}

  \declaretheorem[style=plain,numberlike=definition]{corollary}
  \declaretheorem[style=plain,numberlike=definition]{lemma}
  \declaretheorem[style=plain,numberlike=definition]{proposition}
  \declaretheorem[style=plain,numberlike=definition]{theorem}

  \declaretheorem[style=plain,numbered=no,name=Theorem]{theorem*}
  \declaretheorem[style=plain,numbered=no,name=Conjecture]{conjecture*}

  \Crefname{corollary}{Corollary}{Corollaries}
  \Crefname{definition}{Definition}{Definitions}
  \Crefname{lemma}{Lemma}{Lemmas}
  \Crefname{proposition}{Proposition}{Propositions}
  \Crefname{remark}{Remark}{Remarks}
  \Crefname{theorem}{Theorem}{Theorems}
  \Crefname{conjecture}{Conjecture}{Conjectures}

  \newlist{axioms}{enumerate}{1}
  \Crefname{axiomsi}{}{}

  \newenvironment{tikzeq*}
  {
    \begingroup
    \begin{equation*}
    \begin{tikzpicture}[baseline=(current bounding box.center)]
  }
  {
    \end{tikzpicture}
    \end{equation*}
    \endgroup
    \ignorespacesafterend
  }

  \tikzset
  {
    diagram/.style=
    {
      matrix of math nodes,
      column sep={4.3em,between origins},
      row sep={4em,between origins},
      text height=1.5ex,
      text depth=.25ex
    },
    over/.style={preaction={draw=white,-,line width=6pt}},
    every to/.style={font=\footnotesize},
    inj/.style={right hook->},
    surj/.style={-{Latex[open]}},
    cof/.style={>->},
    fib/.style={->>},
  }

  \DeclareFontFamily{U}{mathx}{\hyphenchar\font45}

  \DeclareFontShape{U}{mathx}{m}{n}{
    <5> <6> <7> <8> <9> <10>
    <10.95> <12> <14.4> <17.28> <20.74> <24.88>
    mathx10}{}

  \DeclareSymbolFont{mathx}{U}{mathx}{m}{n}

  \DeclareFontFamily{U}{mathb}{\hyphenchar\font45}

  \DeclareFontShape{U}{mathb}{m}{n}{
    <5> <6> <7> <8> <9> <10>
    <10.95> <12> <14.4> <17.28> <20.74> <24.88>
    mathb10}{}

  \DeclareSymbolFont{mathb}{U}{mathb}{m}{n}

  \DeclareMathAccent{\widebar}{0}{mathx}{"73}

  \DeclareMathSymbol{\Rsh}{\mathrel}{mathb}{"E9}

  \DeclareFontFamily{U}{MnSymbolA}{}

  \DeclareFontShape{U}{MnSymbolA}{m}{n}{
    <-6> MnSymbolA5
    <6-7> MnSymbolA6
    <7-8> MnSymbolA7
    <8-9> MnSymbolA8
    <9-10> MnSymbolA9
    <10-12> MnSymbolA10
    <12-> MnSymbolA12}{}

  \DeclareSymbolFont{MnSyA}{U}{MnSymbolA}{m}{n}

  \DeclareMathSymbol{\twoheaddownarrow}{\mathrel}{MnSyA}{27}

  \newcommand{\MSC}[1]{%
    \let\thempfn\relax
    \footnotetext[0]{2020 Mathematics Subject Classification: #1.}
  }

\tikzstyle{vertex}=[circle, fill, minimum size=4pt, inner sep=0pt]

\newcommand{\Cat}{\mathsf{Cat}}

\newcommand{\fcat}[2]{{#2}^{#1}} 

\newcommand{\adjunct}[4]{#1 \from #3 \rightleftarrows #4 : \! #2} 

\newcommand{\face}[2]{\partial^{#1}_{#2}} 
\newcommand{\degen}[2]{\sigma^{#1}_{#2}} 

\newcommand{\boxcat}{\mathord{\square}} 

\newcommand{\Kan}{\mathsf{Kan}}
\newcommand{\cKan}{\Kan_{\boxcat}}
\newcommand{\sKan}{\Kan_{\Delta}}

\newcommand{\Deltaaug}{\Delta_{\mathsf{aug}}} 
\newcommand{\sSetaug}{\sSet_{\mathsf{aug}}} 
\newcommand{\join}{\ast} 

\newcommand{\cCat}{\Cat_{\boxcat}} 
\newcommand{\cohnerve}[1][\boxcat]{{\nerve}_{#1}} 

\newcommand{\restr}[2]{{#1}|_{#2}} 

\newcommand{\mc}[1]{\mathcal{#1}} 

\newcommand{\maxm}[1]{{#1}^{\#}} 
\newcommand{\minm}[1]{{#1}^{\flat}} 
\newcommand{\natm}[1]{{#1}^{\natural}} 

\newcommand{\mcore}{\operatorname{Core_+}} 

\newcommand{\sd}{\operatorname{sd}} 

\newcommand{\msd}{\operatorname{sd}_+}
\newcommand{\mSd}{\operatorname{Sd}_+}
\newcommand{\msSet}{\sSet_+}
\newcommand{\mEx}{\operatorname{Ex}_+}
\newcommand{\msdop}{\msd^\op}
\newcommand{\mSdop}{\mSd^\op}
\newcommand{\mExop}{\mEx^\op}

\newcommand{\msub}{\mathrm{max}_!}
\newcommand{\msup}{\mathrm{max}^*}
\newcommand{\msupb}{\overline{\msup}}
\newcommand{\minsub}{\mathrm{min}_!}
\newcommand{\minsup}{\mathrm{min}^*}
\newcommand{\minsupb}{\overline{\minsup}}

\newcommand{\LI}[2]{{}^* \! L^{#1}_{#2}}
\newcommand{\dfLI}[1][k]{\LI{n}{#1}}
\newcommand{\LJ}[2]{L^{#1}_{#2}}
\newcommand{\dfLJ}[1][k]{\LJ{n}{#1}}
\newcommand{\RI}[2]{{}^* \! R^{#1}_{#2}}
\newcommand{\dfRI}[1][k]{\RI{n}{#1}}
\newcommand{\RJ}[2]{R^{#1}_{#2}}
\newcommand{\dfRJ}[1][k]{\RJ{n}{#1}}

\newcommand{\lpathloop}[2][W]{#2 \slice #1}
\newcommand{\rpathloop}[2][W]{#1 \slice #2}
\newcommand{\lcone}{C^L}
\newcommand{\rcone}{C^R}
\newcommand{\LF}[1]{\mathrm{LF}_{#1}}
\newcommand{\RF}[1]{\mathrm{RF}_{#1}}
\newcommand{\Sp}{\mc{S}} 
\newcommand{\ULF}{\Sp_*} 
\newcommand{\map}{\operatorname{map}}
\newcommand{\eqsd}{\operatorname{\hat{\mathrm{sd}}}} 
\newcommand{\eqSd}{\operatorname{\hat{\mathrm{Sd}}}} 

\thanksmarkseries{arabic}
\author{
  Daniel Carranza \thanks{Johns Hopkins University} \and 
  Krzysztof Kapulkin \thanks{University of Western Ontario} \and
  Zachery Lindsey \thanks{Northwestern University}
}
\title{Calculus of Fractions for Quasicategories} 
\date{\today}

\begin{document}

  \maketitle

  \begin{abstract}
    We describe a generalization of Gabriel and Zisman's Calculus of Fractions to quasicategories, showing that the two essentially coincide for the nerve of a category.
    We then prove that the marked Ex-functor can be used to compute the localization of a marked quasicategory satisfying our condition and that the appropriate (co)completeness properties of the quasicategory carry over to its localization.
    \MSC{18N55, 18N60, 18N40, 55U35, 18N50}
  \end{abstract}

\tableofcontents*

\section*{Introduction} \label{sec:intro}

The localization of a category at a class of maps (usually thought of as some sort of `weak equivalences') refers to a universal or free way of inverting these maps in the category.
Computing it is in general hard, as the morphisms in the localization are classes of finite zigzags, e.g., $\cdot \leftarrow \cdot \rightarrow \cdot \leftarrow \cdots \rightarrow \cdot$ where the `backwards' maps are weak equivalences.
It is perhaps not surprising that even if the category in question is locally small, its localization need not be.
Likewise, it is unreasonable to expect the localization to possess any interesting categorical structure, like limits or colimits, and indeed very few examples of localizations do, even when the category being localized is itself complete and cocomplete.

Introduced in \cite{gabriel-zisman}, \emph{calculus of fractions} is a set of conditions on a marked category, i.e., a pair $(\mc{C}, W)$ consisting of a category and a class of maps, which simplifies this task.
Shall $(\mc{C}, W)$ satisfy (or admit) calculus of fractions, the set of maps between any two objects can be described using zigzags of length 2 instead of general zigzags.
In addition, several (co)completeness properties can be carried over from $\mc{C}$ to its localization.
As such, calculus of fractions is a powerful tool for computing localization, see, e.g., \cite{margolis:spectra}.
We should note here that there are two sets of conditions one could impose: calculus of \emph{left} fractions and calculus of \emph{right} fractions.
Throughout this introduction, we speak of the former, but all statements have dual versions, given in the paper, involving the latter.

In this paper, we present a generalization of the calculus of fractions to the theory of $(\infty, 1)$-categories, taken here to be quasicategories.
Given a marked quasicategory $(\mc{C}, W)$, we can express our conditions as a certain lifting property against a family of poset inclusions of a half-cube into a whole cube.
This resembles the familiar condition from the usual calculus of fractions \cite[Def.~2.2]{gabriel-zisman} which asks that every (co)span be completed to a square.

As a test case, we consider quasicategories arising as nerves of categories and show that the nerve of a marked category $(\mc{C}, W)$ satisfies (our quasicategorical) calculus of fractions exactly when $(\mc{C}, W)$ satisfies \emph{proper} calculus of fractions.
The properness condition is a mild strengthening of the definition found in \cite{gabriel-zisman} and is satisfied by most interesting examples.
In particular, it is implied by the 2-out-of-3 property.

As indicated above, the key result proven in \cite[Ch.~1]{gabriel-zisman} is that if $(\mc{C}, W)$ satisfies calculus of left fractions, then the localization of $\mc{C}$ at $W$ is given by the \emph{category of left fractions} $W^{-1}\mc{C}$, i.e., its objects are zigzags $\cdot \to \cdot \overset{\sim}{\leftarrow} \cdot$ of length $2$.
Our replacement of it is the \emph{marked Ex-functor}.
Recall that the Ex-functor \cite{kan:css} is a functor on simplicial sets given by $(\Ex X)_n = \sSet(\sd[n], X)$, where $\sd[n]$ is the poset of non-empty subsets of the linear order $[n] = \{ 0 \leq 1 \leq \ldots \leq n \}$ ordered by inclusion.
We `upgrade' it to a functor $\mEx$ from marked simplicial sets to simplicial sets by equipping $\sd[n]$ with a marking: a $1$-simplex $A_0 \subseteq A_1$ is marked if and only if $\max A_0 = \max A_1$.
We note that the marked $\sd$ functor was previously considered by several authors, including: Barwick and Kan \cite{barwick-kan:relative} in their work on the model structure of the category of categories with weak equivalences, Szumi{\l}o \cite{szumilo:frames} in his proof that the homotopy theories of cofibration categories and cocomplete quasicategories are equivalent; and Lazarev, Sylvan, Tanaka \cite{lazarev-sylvan-tanaka:sectors} in their work on the $\infty$-category of stabilized Liouville sectors.

With this definition, we can state our first main theorem:

\begin{theorem*}[cf.~\cref{mEx-qcat,mEx-computes-localization}]
  If a marked quasicategory $(\mc{C}, W)$ satisfies calculus of left fractions, then $\mEx(\mc{C}, W)$ is a quasicategory and the canonical map $\max^* \from \mc{C} \to \mEx(\mc{C}, W)$, induced by $\max \from \sd[n] \to [n]$, is the localization of $\mc{C}$ at $W$.
\end{theorem*}

We note here that our calculus of fractions condition is satisfied by a wide range of marked quasicategories.
For example, all reflective localizations (referred to simply as localizations in \cite[Def.~5.2.7.2]{lurie:htt}) satisfy calculus of fractions.

As an immediate corollary, we deduce that for a quasicategory $\mc{C}$, the simplicial set $\Ex \mc{C}$ is a Kan complex if and only if $\mc{C}^\sharp$ admits calculus of left fractions.
This in particular generalizes the work of Meier and Ozornova \cite{meier-ozornova:partial-model}, who studied when $\Ex$ of the nerve of a category is a Kan complex, to the setting of quasicategories.

Moreover, from the theorem above, we then deduce that if $(\mc{C}, W)$ satisfies calculus of left fractions, then $W$ is saturated (i.e.~it is precisely the class of maps inverted by the localization functor) if and only if it satisfies the 2-out-of-6 property (\cref{saturated_iff_2-outta-6}).
This mirrors the analogous result for $1$-categories, proven in \cite{kashiwara-schapira}.

We then generalize another theorem of Gabriel and Zisman's, namely that if a marked category $(\mc{C}, W)$ satisfies calculus of left fractions, then the canonical functor from $\mc{C}$ to its localization preserves finite colimits, and moreover if $\mc{C}$ is cocomplete, then so is its localization.
In \cite{gabriel-zisman}, this is proven by establishing that the hom-sets in the category of fractions are filtered colimits of those in $\mc{C}$.
The same line of argument applies in our setting, although the proofs are less straightforward:

\begin{theorem*}[cf.~\cref{mapping-space-is-colim,lpathloop-filtered,localization-has-limits}]
Suppose that a marked quasicategory $(\mc{C}, W)$ satisfies calculus of left fractions and $W$ is closed under 2-out-of-3.
  \begin{enumerate}
    \item The mapping space in $\mEx(\mc{C}, W)$ from $x$ to $y$ is a filtered colimit of mapping spaces from $x$ to $y'$, indexed by marked $1$-simplices $y \overset{\sim}{\rightarrow} y'$.
    \item If $\mc{C}$ is finitely cocomplete, then so is $\mEx(\mc{C}, W)$.
    Moreover, $\max^* \from \mc{C} \to \mEx(\mc{C}, W)$ preserves finite colimits.
  \end{enumerate}
\end{theorem*}
Using the theorem, we can immediately prove the following three facts further characterizing the condition of calculus of fractions and relating the 1-categorical and quasicategorical localizations:
\begin{itemize}
  \item If $(\mc{C}. W)$ is (the nerve of) a $1$-category satisfying calculus of left fractions, then its localization is again (the nerve of) a $1$-category (\cref{1-cat-clf-localization-is-1-cat}).
  \item If $\mc{C}$ has finite colimits and $W$ is saturated, then $(\mc{C}, W)$ satisfies calculus of left fractions if and only if its localization has finite colimits and the localization functor preserves them (\cref{clf-iff-fin-lim-preserve}).
  \item If $(\mc{C}, W)$ satisfies calculus of left and right fractions and $\mc{C}$ is a stable quasicategory, then its localization is again a stable quasicategory (\cref{stable_clf_crf}).
  In the context of abelian (instead of stable) $1$-categories, this is due to \cite{gabriel-zisman}.
\end{itemize}

\textbf{Relation to the work of D.-C.~Cisinski.}
Ideas similar to those of \cref{sec:mapping-space} are considered in \cite[\S 7.2]{cisinski:higher-categories}.
One of the main results of \cref{sec:mapping-space} is that if $(\mc{C}, W)$ satisfies calculus of (right) fractions then the mapping spaces in the localization may be computed as a filtered colimit of mapping spaces in $\mc{C}$, indexed by the ``marked slice'' over the domain (as indicated above).
For us, this is a not-so-easy consequence of the definition of the marked Ex-functor.
In \cite[Thm.~7.2.8]{cisinski:higher-categories}, an analogous colimit formula (which is no longer filtered in general) is proven for any sufficiently well-behaved model of the marked slice.
Moreover in \cite[Rem.~7.2.10]{cisinski:higher-categories}, a model for this colimit is given whose 0-simplices are spans/right fractions, just as in the mapping spaces of our marked Ex (though the two do not coincide in general).
In the language of \cite[Def.~7.2.6]{cisinski:higher-categories}, a well-behaved model for the marked slice at an object $x$ is called a \emph{calculus of right fractions} at $x$.

\textbf{Comments on style.}
As mentioned before, calculus of fractions appears in two forms: calculus of left fractions and calculus of right fractions.
For the readers' convenience, we give most statements in both forms; the only exceptions being technical lemmas that are only needed for a proof of another theorem and that we do not expect to be of independent interest.

In addition, we assume that the reader is familiar with the basics of simplicial homotopy theory \cite{goerss-jardine} and the main ideas of quasicategory theory \cite{cisinski:higher-categories,joyal:theory-of-qcats,lurie:htt}.
While we have included some background material on these in \cref{sec:quasicategories}, it is rather brief.
We try to give precise references, notably to \cite{cisinski:higher-categories,joyal:qcat-kan,joyal:theory-of-qcats,lurie:htt} for all facts used in the paper, but we realize that this does not replace a comprehensive introduction.

\textbf{Organization of the paper.}
In the first three sections, we collect the necessary background on calculus of fractions (\cref{sec:classical-clf}), following \cite{gabriel-zisman}, quasicategories (\cref{sec:quasicategories}), and marked simplicial sets (\cref{sec:marked-sSet}).
As indicated above, throughout the paper, we assume familiarity with foundational notions and results of simplicial homotopy theory, so this review is rather brief and intended mostly to establish the necessary notation.
We then define calculus of fractions for quasicategories (\cref{sec:calculus-def}), compare it to the classical version (\cref{sec:clf-infty-vs-classical}), and provide a wide range of examples (\cref{sec:clf-examples}).
We then move towards describing the localizations of marked quasicategories admitting a calculus of fractions.
In \cref{sec:mEx-is-qcat}, we define the marked Ex-functor and show that it takes marked quasicategories satisfying calculus of fractions to quasicategories.
Before showing that it computes the localization of an arbitrary marked quasicategory in \cref{sec:localization-general}, we first prove (\cref{sec:localization-at-equivs}) this result when the quasicategory is marked at equivalences.
We then turn our attention to properties of the resulting localization.
In \cref{sec:mapping-space}, we give a formula for the mapping spaces in the localization as a filtered colimit of those in the quasicategory itself; and in \cref{sec:limits}, we prove results about existence of finite (co)limits in the localization as well as preservation thereof by the localization functor.

\textbf{Acknowledgements.}
This paper grew out of the third author's Ph.D.~thesis written in 2018 at Indiana University under the supervision of Mike Mandell.
The three of us are deeply grateful to Mike for suggesting this research direction to us.
We would also like to thank Andrew Blumberg and Mike Mandell for writing \cite{blumberg-mandell:string-topology}, which provided the initial motivation and an application of some of the results found in this paper.

This material is based upon work supported by the National Science Foundation under Grant No.~DMS-1928930 while the first two authors participated in a program hosted by the Mathematical Sciences Research Institute in Berkeley, California, during the 2022--23 academic year.

\chapter*{Preliminaries}
\cftaddtitleline{toc}{chapter}{Preliminaries}{}

\section{Classical calculus of fractions} \label{sec:classical-clf}

In this section, we recall the classical results on calculus of fractions following \cite[Ch.~1]{gabriel-zisman}.
The main results of interest are \cref{classical-fractions-compute-localization,classical-fractions-colims}.
These results both have analogues in the quasicategorical case and our exposition in this section is intended to reflect the structure of the paper going forward.
In particular, the analogue of \cref{classical-fractions-compute-localization} is proven in \cref{sec:localization-general}, and the analogue of \cref{classical-fractions-colims} is proven in \cref{sec:mapping-space,sec:limits}.

We first define marked categories, which package the data of a category with the collection of morphisms one wishes to localize at.
\begin{definition} 
	A \emph{marked category} is a pair $(\mc{C}, W)$ where $\mc{C}$ is a category and $W$ is a collection of morphisms in $\mc{C}$.
\end{definition}
In a marked category $(\mc{C}, W)$, we write either
$\begin{tikzcd}[sep = 1.3em, cramped]
	x \ar[r, "\sim"{xshift=-0.3ex}] & y
\end{tikzcd}$ or $\begin{tikzcd}[sep = 1.3em, cramped]
	x \ar[r, "\sim"'{xshift=-0.3ex}] & y
\end{tikzcd}$
to indicate that the morphism $\begin{tikzcd}[sep = 1.2em, cramped] x \ar[r] & y \end{tikzcd}$ is in $W$.
For a category $\mc{C}$, we write $\natm{\mc{C}}$ for the pair $\mc{C}$ marked at all isomorphisms of $\mc{C}$.
Given marked categories $(\mc{C}, W), (\mc{C}', W')$, we write $\Cat_{+}((\mc{C}, W), (\mc{C}', W'))$ for the full subcategory of the functor category from $\mc{C}$ to $\mc{D}$ consisting of functors which sends morphisms in $W$ to morphisms in $W'$.

We recall the definition of the localization of $\mc{C}$ at a class $W$.
\begin{definition} \label{def:ordinary-localization}
	Let $(\mc{C}, W)$ be a marked category.
	The \emph{localization} of $\mc{C}$ at $W$ is a functor
	\[ \gamma \from \mc{C} \to \mc{C}[W^{-1}] \]
	into a category $\mc{C}[W^{-1}]$ such that
	\begin{itemize}
		\item the functor $\gamma$ sends morphisms in $W$ to isomorphisms in $\mc{C}[W^{-1}]$; and
		\item for any category $\mc{D}$, the pre-composition functor
		\[ \gamma^* \from \Cat(\mc{C}[W^{-1}], \mc{D}) \to \Cat_+((\mc{C}, W), \natm{\mc{D}}) \]
		is an equivalence of categories.
	\end{itemize}
\end{definition}

In \cite{gabriel-zisman}, a construction is given for the localization of a pair $(\mc{C}, W)$ which satisfies \emph{calculus of fractions}.
This condition allows one to construct a workable model of the localization in which morphisms between objects become (co)spans of morphisms rather than arbitrary zigzags.
\begin{definition} \label{def:classical-clf}
	A marked category $(\mc{C}, W)$ satisfies \emph{calculus of left fractions} (or \emph{CLF}) if the following conditions hold:
	\begin{enumerate}
		\item $W$ is closed under composition and contains identities;
		\item any span $(f \from X \to Y, w \from X \to X')$ with $w \in W$ can be completed to a commutative square
		\[ \begin{tikzcd}
			X \ar[r, "w", "\sim"'] \ar[d, "f"'] & X' \ar[d, dotted] \\
			Y \ar[r, "\sim"', dotted, "w'"] & Y'
		\end{tikzcd} \]
		with $w' \in W$;
		\item given parallel morphisms $f, g \from X \to Y$, if there exists $w \from X' \to X$ in $W$ such that $fw = gw$ then there exists $v \from Y \to Y'$ in $W$ such that $vf = vg$ (as in the diagram)
		\[ \begin{tikzcd}
			X' \ar[r, "w", "\sim"'] & X \ar[r, yshift=0.7ex, "f"] \ar[r, yshift=-0.7ex, "g"'] & Y \ar[r, "v", "\sim"', dotted] & Y'
		\end{tikzcd} \]
	\end{enumerate}
\end{definition}
We also introduce a slight strengthening of calculus of fractions.
\begin{definition} \label{def:proper-clf}
	A marked category $(\mc{C}, W)$ satisfies \emph{proper CLF} if it satisfes conditions (1) and (3) of \cref{def:classical-clf}, as well as the following strengthening of condition (2):
	\begin{enumerate}
		\item[2'.] Any span $(f \from X \to Y, w \from X \to X')$ with $w \in W$ can be completed to a commutative square
		\[ \begin{tikzcd}
			X \ar[r, "w", "\sim"'] \ar[d, "f"'] & X' \ar[d, dotted, "f'"] \\
			Y \ar[r, "\sim"', dotted, "w'"] & Y'
		\end{tikzcd} \]
		with $w' \in W$.
		Moreover, if $f$ is in $W$ then so is $f'$.
	\end{enumerate}
\end{definition}
\begin{definition} \label{def:classical-crf}
	A marked category $(\mc{C}, W)$ satisfies \emph{calculus of right fractions} (or \emph{CRF}) if $(\mc{C}^\op, W)$ satisfies CLF.
	It satisfies \emph{proper CRF} if $(\mc{C}^\op, W)$ satisfies proper CLF.
\end{definition}
In \cref{sec:calculus-def}, we will introduce the notion of calculus of fractions for quasicategories.
When viewing an ordinary category as a quasicategory, this notion will agree with proper CLF (respectively, proper CRF) but not CLF (respectively, CRF).
\begin{remark}
	Most examples which satisfy CLF (or CRF) also satisfy proper CLF (respectively, proper CRF).
	For instance, if $W$ is closed under 2-out-of-3 then the two are equivalent.
	The diagram
	
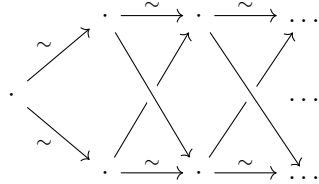
\[ \begin{tikzcd}
		{} & \cdot \ar[r, "\sim"] & \cdot \ar[r, "\sim"] & \dots \\
		\cdot \ar[ur, "\sim"] \ar[dr, "\sim"'] & {} & {} & \dots \\
		{} & \cdot \ar[r, "\sim"] \ar[uur] & \cdot \ar[r, "\sim"] \ar[uur] \ar[from=uul, crossing over] & \dots \ar[from=uul, crossing over]
	\end{tikzcd} \]
	depicts an example of a marked category which satisfies CLF but not proper CLF.
\end{remark}

We describe the model of the localization constructed in \cite{gabriel-zisman} for a pair $(\mc{C}, W)$ satisfying CLF.

Define a category $\mc{C} W^{-1}$ whose objects are those of $\mc{C}$.
The set of morphisms from $x$ to $y$ in $\mc{C} W^{-1}$ is the set of cospans
\[ \begin{tikzcd}
	x \ar[r, "f"] & y' & y \ar[l, "w", "\sim"', swap]
\end{tikzcd} \]
where $w$ is a weak equivalence, subject to the equivalence relation defined by $(f, w) \sim (g, v)$ if there exists a commutative diagram
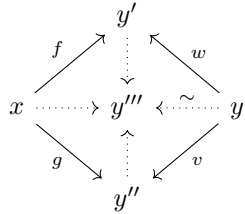
\[ \begin{tikzcd}
	{} & y' \ar[d, dotted] & {} \\
	x \ar[ur, "f"] \ar[dr, "g"'] \ar[r, dotted] & y''' & y \ar[ul, "w"'] \ar[dl, "v"] \ar[l, dotted, "\sim"'] \\
	{} & y'' \ar[u, dotted]  & {}
\end{tikzcd} \]
In particular, conditions (2) and (3) of CLF (\cref{def:classical-clf}) are required to define composition of morphisms.
With this, there is an evident functor $\mc{C} \to \mc{C}W^{-1}$ which is identity on objects and which sends a morphism $f$ to the equivalence class of the co-span $(f, \id)$.

Dually, if $(\mc{C}, W)$ satisfies CRF then there is an analogous construction $W^{-1} \mc{C}$ whose morphisms are given by spans.

\begin{theorem}[{\cite[Prop.~2.4]{gabriel-zisman}}] \label{classical-fractions-compute-localization}
	Let $(\mc{C}, W)$ be a marked category.
	\begin{enumerate}
		\item If $(\mc{C}, W)$ satisfies calculus of left fractions then the functor $\mc{C} \to \mc{C} W^{-1}$ is the localization of $\mc{C}$ at $W$.
		\item If $(\mc{C}, W)$ satisfies calculus of right fractions then the functor $\mc{C} \to W^{-1} \mc{C}$ is the localization of $\mc{C}$ at $W$. \qed
	\end{enumerate}
\end{theorem}

Moreover, under certain conditions on $(\mc{C}, W)$, the category $\mc{C} W^{-1}$ is locally small and admits finite colimits (dually, the category $W^{-1} \mc{C}$ is locally small and admits finite limits).
In the following statement, we write $\lpathloop{x}$ (dually $\rpathloop{x}$) for the full subcategory of the slice category $x \slice \mc{C}$ (dually $\mc{C} \slice x$) comprised of weak equivalences $x \xrightarrow{\sim} x'$ (dually $x' \xrightarrow{\sim} x$).
\begin{theorem}[{\cite[Prop.~3.1 \& Cor.~3.2]{gabriel-zisman}}] \label{classical-fractions-colims}
	Let $x, y$ be objects in a locally small marked category $(\mc{C}, W)$.
	\begin{enumerate}
		\item Suppose $(\mc{C}, W)$ satisfies CLF and that $\lpathloop{y}$ is finally small.
		Then,
		\begin{itemize}
			\item the set of maps in $\mc{C} W^{-1}$ from $x$ to $y$ is computed by the filtered colimit
			\[ \mc{C}W^{-1}(x, y) \cong \colim\limits_{y' \in \lpathloop{y}} \mc{C}(x, y'); \]
			\item the functor $\mc{C} \to \mc{C} W^{-1}$ preserves finite colimits; and
			\item if $\mc{C}$ admits all finite colimits then so does $\mc{C} W^{-1}$.
		\end{itemize}
		\item Suppose $(\mc{C}, W)$ satisfies CRF and that $\rpathloop{x}$ is finally small.
		Then,
		\begin{itemize}
			\item the set of maps in $W^{-1} \mc{C}$ from $x$ to $y$ is computed by the filtered colimit
			\[ \mc{C}W^{-1}(x, y) \cong \colim\limits_{x' \in (\rpathloop{x})^\op} \mc{C}(x', y); \]
			\item the functor $\mc{C} \to W^{-1} \mc{C}$ preserves finite limits; and
			\item if $\mc{C}$ admits all finite limits then so does $ W^{-1} \mc{C}$. \qed	
		\end{itemize}
	\end{enumerate}
\end{theorem}

\section{Quasicategories} \label{sec:quasicategories}

In this section, we review aspects of quasicategory theory such as: detecting equivalences, the Joyal model structure, joins, slices, mapping spaces, and (co)limits.
Throughout, we will assume familiarity with the basics of simplicial homotopy theory, including the Kan--Quillen model structure.

We write $\sSet$ for the category of simplicial sets.
Arbitrary simplicial sets will typically be denoted with capital letters e.g.~$X, Y, \dots$.
We follow standard notation, including writing $[n]$ for the linear order $\{ 0 \leq 1 \leq \dots \leq n \}$ and $\Delta$ for the category of finite non-empty linear orders.
We write
\[ \face{}{i} \from [n-1] \to [n], \quad \degen{}{i} \from [n] \to [n-1] \]
for the face and degeneracy poset maps, respectively, omitting the data of the domain and codomain from our notation.
We write $\simp{n}$ for the representable $n$-simplex, $\bd \simp{n}$ for its boundary, and $\horn{n}{k}$ for the $(n, k)$-horn.
The action of simplicial operators will be written on the right, e.g.~$x\degen{}{0}$ denotes the 0-degeneracy of an $n$-simplex $x \in X_n$.

\begin{definition}
	A \emph{quasicategory} is a simplicial set $\mc{C}$ with the right lifting property against the set
	\[ \{ \horn{n}{k} \ito \simp{n} \mid n \geq 2, \ 0 < k < n \}. \]
\end{definition}
\begin{example}
	If $\mc{C}$ is a category then its nerve $N \mc{C} \in \sSet$
	is a quasicategory.
\end{example}
We adopt the convention of referring to 0-simplices in a quasicategory as \emph{objects} and 1-simplices in a quasicategory as \emph{morphisms}.

Let $E^1$ denote the nerve of the contractible category with two objects $0, 1$.
\begin{definition}
	A 1-simplex $f$ in a simplicial set $X$ is an \emph{$E^1$-equivalence} if, as a map $f \from \simp{1} \to X$, it admits a lift to a map $E^1 \to X$.
	\[ \begin{tikzcd}
		\simp{1} \ar[r, "f"] \ar[d, hook] & X \\
		E^1 \ar[ur, dotted]
	\end{tikzcd} \]
\end{definition}
\begin{example}
	If $\mc{C}$ is a category then an $E^1$-equivalence in its nerve is exactly an isomorphism in $\mc{C}$.
\end{example}
Let $J$ denote the simplicial set
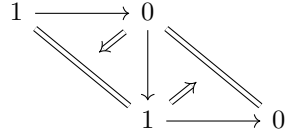
\[ \begin{tikzcd}[sep = large]
	1 \ar[r] \ar[rd, equal, ""{name=L}] & 0 \ar[d] \ar[rd, equal, ""'{name=R}] \ar[to=L, Rightarrow, shorten=0.5ex] & {} \\
	{} & 1 \ar[r] \ar[to=R, Rightarrow, shorten=0.5ex] & 0
\end{tikzcd} \]
Unless otherwise specified, any map $\simp{1} \ito J$ will denote the inclusion of the 1-simplex from 0 to 1.
\begin{definition}
	A 1-simplex $f$ in a simplicial set $X$ is a \emph{$J$-equivalence} if, as a map $f \from \simp{1} \to X$, it admits a lifts to a map $J \to X$.
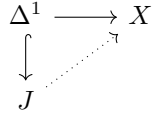
	\[ \begin{tikzcd}
		\simp{1} \ar[r] \ar[d, hook] & X \\
		J \ar[ur, dotted]
	\end{tikzcd} \]
\end{definition}

We recall that $E^1$-equivalences and $J$-equivalence coincide in a quasicategory.
\begin{proposition} \label{e1-equiv-j-equiv-eq}
	If $X$ is a quasicategory then a 1-simplex in $X$ is a $J$-equivalence if and only if it is an $E^1$-equivalence. \qed
\end{proposition}
In light of \cref{e1-equiv-j-equiv-eq}, we will refer to both $E^1$-equivalences and $J$-equivalences in a quasicategory as simply \emph{equivalences} (or \emph{invertible}).

The simplicial sets $E^1$ and $J$ give well-behaved interval objects for defining a notion of homotopy.
\begin{definition} \leavevmode
	\begin{itemize}
		\item For simplicial maps $f, g \from X \to Y$, an \emph{$E^1$-homotopy} (or \emph{natural equivalence}) from $f$ to $g$ is a map $\alpha \from X \times E^1 \to Y$ such that $\restr{\alpha}{X \times \{ 0 \}} = g$ and $\restr{\alpha}{X \times \{ 1 \}} = g$.
		\item A map $f \from X \to Y$ is an \emph{$E^1$-homotopy equivalence} if there exists $g \from Y \to X$ along with $E^1$-homotopies $gf \sim \id[X]$ and $fg \sim \id[Y]$.
		\item For simplicial maps $f, g \from X \to Y$, a \emph{$J$-homotopy} from $f$ to $g$ is a map $\alpha \from X \times J \to Y$ such that $\restr{\alpha}{X \times \{ 0 \}} = f$ and $\restr{\alpha}{X \times \{ 1 \}} = g$.
		\item A map $f \from X \to Y$ is a \emph{$J$-homotopy equivalence} if there exists $g \from Y \to X$ along with $J$-homotopies $gf \sim \id[X]$ and $fg \sim \id[Y]$.
	\end{itemize}
\end{definition}
If $Y$ is a quasicategory then any map $\alpha \from X \times \simp{1} \to Y$ whose restriction $\restr{\alpha}{\{ x \} \times \simp{1}} \from \{ x \} \times \simp{1} \to Y$ along every 0-simplex $x \in X_0$ is invertible lifts to a natural equivalence.

Recall that quasicategories are the fibrant objects in the \emph{Joyal model structure} on simplicial sets.
\begin{theorem}[{\cite[Thm.~6.12]{joyal:theory-of-qcats}}] \label{joyal-model-str}
	The category of simplicial sets $\sSet$ admits a model structure where
	\begin{itemize}
		\item cofibrations are monomorphisms;
		\item fibrant objects are quasi-categories; fibrations between fibrant objects are inner isofibrations;
		\item weak equivalences are weak categorical equivalences. \qed
	\end{itemize}
\end{theorem}
In particular, both $E^1$-homotopy equivalences and $J$-homotopy equivalences are weak categorical equivalences.

We also recall the definition of left fibrations, which we use in \cref{sec:limits} to compute colimits of $\infty$-groupoids.
\begin{definition}
	A \emph{left fibration} is a simplicial map $f \from X \to Y$ with the right lifting property against the set
	\[ \{ \horn{n}{k} \ito \simp{n} \mid n \geq 1, \ 0 \leq k < n \}. \]
\end{definition}
In particular, a left fibration between quasicategories is a fibration (\cite[Prop.~4.10]{joyal:theory-of-qcats}).

In \cref{sec:limits}, we also make use of the join and slice constructions, which we now define.
Let $\Deltaaug$ denote the category obtained by adjoining an initial object $[-1]$ to $\Delta$.
Define the \emph{ordinal sum} functor $\Deltaaug \times \Deltaaug \to \Deltaaug$ by
\[ ([m], [n]) \mapsto [m + n + 1]. \]
By Day convolution, this gives a monoidal product on the presheaf category, which we denote by
\[ \uvar \join \uvar \from \sSetaug \times \sSetaug \to \sSetaug . \]
The inclusion $\Delta \subseteq \Deltaaug$ induces a forgetful functor on presheaf categories $\sSetaug \to \sSet$ (which we do not introduce notation for).
This functor admits a right adjoint $X \mapsto \hat{X}$ defined by
\[ \hat{X}_n := \begin{cases}
	\{ \ast \} & n = -1 \\
	X_n & \text{otherwise.}
\end{cases} \]
\begin{definition}
	The \emph{join} of simplicial sets is the functor $\uvar \join \uvar \sSet \times \sSet \to \sSet$ defined by
	\[ X \join Y := \hat{X} \join \hat{Y}. \]
\end{definition}
\begin{proposition}[{\cite[Prop.~3.2]{joyal:theory-of-qcats}}]
	For simplicial sets $X, Y$, the $n$-simplices of $X \join Y$ are given by the formula
	\[ (X \join Y)_n := X_n \sqcup Y_n \sqcup \coprod\limits_{i + j + 1 = n} X_i \times Y_j . \]
\end{proposition}
If $X$ and $Y$ are quasicategories then $X \join Y$ is again a quasicategory (\cite[Cor.~3.23]{joyal:theory-of-qcats}).

The simplicial set $X \join Y$ admits two natural inclusions
\[ i_0 \from X \ito (X \join Y) \quad i_1 \from Y \ito (X \join Y). \]
Thus, we may view the functors $\uvar \join X$ and $X \join \uvar$ as taking values in the slice under $X$, i.e.~
\[ \uvar \join X, X \join \uvar \from \sSet \to X \slice \sSet . \]
In this sense, both functors admit right adjoints (\cite[Prop.~3.12]{joyal:theory-of-qcats}).
These are the \emph{slice over} and \emph{slice under} functors, respectively.
Explicitly, these functors send an object $(Y, f \from X \to Y)$ to the simplicial sets $f \slice Y$ and $Y \slice f$ defined by
\[ \begin{array}{@{(}r@{ \ \slice  \ }l@{ )_n \ := \ }l}
	f & Y & X \slice \sSet \left( (X \join \simp{n}, i_0), (Y, f) \right) \\
	Y & f & X \slice \sSet \left( (\simp{n} \join X, i_1), (Y, f) \right).
\end{array} \]
There are natural projection maps
\[ f \slice Y \to Y \quad Y \slice f \to Y \]
defined by restriction along the inclusion $\simp{n} \to \simp{n} \join X$.

If $Y$ is a quasicategory then both the slice over and the slice under $Y$ are quasicategories (\cite[Cor.~3.20]{joyal:theory-of-qcats}).
As such, we occasionally refer to them as \emph{slice quasicategories}.

Recall that \emph{mapping spaces} in a quasicategory may be constructed in multiple equivalent ways, one of which is as a certain fiber of the slice under the domain (or over the codomain).
\begin{definition}
	Let $x, y$ be objects in a quasicategory $\mc{C}$.
	\begin{enumerate}
		\item The \emph{left mapping space} $\map_L(\mc{C}, x, y)$ from $x$ to $y$ in $\mc{C}$ is the pullback
		\[ \begin{tikzcd}
			\map_L(\mc{C}, x, y) \ar[r] \ar[d] \ar[rd, phantom, "\lrcorner" very near start] & x \slice \mc{C} \ar[d] \\
			\simp{0} \ar[r, "y"] & \mc{C}
		\end{tikzcd} \]
		\item The \emph{right mapping space} $\map_R(\mc{C},x, y)$ from $x$ to $y$ in $\mc{C}$ is the pullback 
		\[ \begin{tikzcd}
			\map_R(\mc{C}, x, y) \ar[r] \ar[d] \ar[rd, phantom, "\lrcorner" very near start] & \mc{C} \slice y \ar[d] \\
			\simp{0} \ar[r, "x"] & \mc{C}
		\end{tikzcd} \]
	\end{enumerate}
\end{definition}
\begin{proposition}[{\cite[Prop.~1.2.2.3 \& Cor.~4.2.1.8]{lurie:htt}}]
	The left and right mapping spaces are homotopy-equivalent Kan complexes. \qed
\end{proposition}

Using slice quasicategories, one may define the notion of cones over (and under) a diagram, as well as limits and colimits in a quasicategory.
\begin{definition}
	Let $\mc{C}$ be a quasicategory and $f \from K \to \mc{C}$ be a simplicial map.
	\begin{enumerate}
		\item An \emph{initial object} in $\mc{C}$ is an object $x \in \mc{C}$ such that any map $f \from \bd \simp{n} \to \mc{C}$ satisfying $\restr{f}{\{ 0 \}} = x$ can be lifted to a map $\simp{n} \to \mc{C}$.
		\item A \emph{cone} under $f$ is a lift of $f$ to a map $K \join \simp{0} \to \mc{C}$, i.e.~an object of $f \slice \mc{C}$.
		\item The \emph{colimit cone} of $f$ is the initial object of $f \slice \mc{C}$.
	\end{enumerate}
\end{definition}
These definitions dualize in a straightforward way.
\begin{definition}
	Let $\mc{C}$ be a quasicategory and $f \from K \to \mc{C}$ be a simplicial map. 
	\begin{enumerate}
		\item A \emph{terminal object} in $\mc{C}$ is an object $x \in \mc{C}$ such that any map $f \from \bd \simp{n} \to \mc{C}$ satisfying $\restr{f}{\{ n+1 \}} = x$ can be lifted to a map $\simp{n} \to \mc{C}$.
		\item A \emph{cone} over $f$ is a lift of $f$ to a map $\simp{0} \join K \to \mc{C}$, i.e.~an object of $\mc{C} \slice f$.
		\item The \emph{limit cone} of $f$ is the terminal object of $\mc{C} \slice f$.
	\end{enumerate}
\end{definition}
Note that the mapping space between any two initial (or terminal) objects is contractible \cite[Prop.~4.4]{joyal:qcat-kan}, hence any two colimits (or limits) of $f$ are equivalent.

\section{Marked simplicial sets and localization} \label{sec:marked-sSet}

Our final section of preliminaries describes the analogues of marked categories and localization for quasicategories.
Mainly, we are interested in the category of marked simplicial sets and the model structure thereon.
We begin with the definition of a marked simplicial set.
\begin{definition}
	A \emph{marked simplicial set} is a pair $(X, W)$ where $X$ is a simplicial set and $W \subseteq X_1$ is a subset of 1-simplices which contains all degenerate 1-simplices.
\end{definition}
We write $\msSet$ for the category of marked simplicial sets and simplicial maps which preserve marked 1-simplices.
This category is Cartesian closed, and we write $(-)^{(X, W)} \from \msSet \to \msSet$ for the left adjoint to the functor $(X, W) \times - \from \msSet \to \msSet$.
Moreover, if $(X', W')$ is a marked quasicategory then $(X', W')^{(X, W)}$ is a quasicategory for any $(X, W) \in \msSet$ \cite[Rem.~3.1.3.1]{lurie:htt}.

We make use of the following functors between simplicial sets and marked simplicial sets.
\begin{proposition}
	There is an adjoint quadruple 
	\[ (\minm{(\uvar)} \adj U \adj \maxm{(\uvar)} \adj \mcore) \]
	between marked simplicial sets and simplicial sets where
	\begin{itemize}
		\item $\minm{(-)} \from \sSet \to \msSet$ is the \emph{minimal marking} functor which sends $X$ to the pair $(X, \degen{}{0} X_0)$, where $\degen{}{0} X_0$ denotes the collection of degenerate 1-simplices in $X$;
		\item $U \from \msSet \to \sSet$ is the forgetful functor which sends a marked simplicial set $(X, W)$ to its underlying simplicial set $X$.
		\item $\maxm{(-)} \from \sSet \to \msSet$ is the \emph{maximal marking} functor which sends $X$ to the pair $(X, X_1)$; and
		\item $\mcore \from \msSet \to \sSet$ is the \emph{marked core} which sends $(X, W)$ to the maximal simplicial subset of $X$ whose 1-simplices are all marked. \qed
	\end{itemize}
\end{proposition}
Additionally, we have the \emph{natural marking} functor 
\[ \natm{(\uvar)} \from \sSet \to \msSet \]
which sends $X$ to the pair $(X, W)$ where $W$ is the set of $J$-equivalences in $X$.
The natural marking is functorial as simplicial maps preserve $J$-equivalences.
\begin{remark}
	The minimal marking functor $\minm{(-)} \from \sSet \to \msSet$ admits a further left adjoint which sends a marked simplicial set $(X, W)$ to the quotient $X / W$.
	That is, the pushout
	\[ \begin{tikzcd}
		\coprod\limits_{w \in W} \simp{1} \ar[r, "w"] \ar[d] \ar[rd, phantom, "\ulcorner" very near end] & X \ar[d] \\
		\coprod\limits_{w \in W} \simp{0} \ar[r] & X / W
	\end{tikzcd} \]
	However, we will not make use of this functor.
\end{remark}

The category $\msSet$ is a reflective subcategory of a presheaf category.
The representable presheaves are the minimally marked simplices $\minm{(\simp{0})}, \minm{(\simp{1})}, \minm{(\simp{2})}, \dots$ along with the maximally marked 1-simplex $\maxm{(\simp{1})}$.
In particular, every marked simplicial set is a canonical colimit of representable presheaves.

We now discuss generalizations of closure under composition and closure under 2-out-of-3 to the quasicategorical setting.
\begin{definition}
	Let $(X, W)$ be a marked simplicial set.
	\begin{enumerate}
		\item The set $W$ is \emph{weakly closed under composition} if every map $\maxm{(\horn{2}{1})} \to (X, W)$ lifts to a map $\maxm{(\simp{2})} \to (X, W)$.
		\item The set $W$ is \emph{strongly closed under composition} if, given any 2-simplex $u \in X_2$ whose $\face{}{0}$- and $\face{}{2}$-faces are marked, we have that its $\face{}{1}$-face is also marked.
		\item The set $W$ is \emph{closed under 2-out-of-3} if, given a 2-simplex $u \in X_2$ such that any two faces of $u$ are marked, the third face is also marked.
	\end{enumerate}
\end{definition}
\begin{remark}
	We can rephrase the above properties in terms of properties of the marked core of $(X, W)$.
	Namely, the set $W$ is
	\begin{enumerate}
		\item weakly closed under composition if and only if the marked core $\mcore(X, W)$ has the right lifting property with respect to the horn inclusion $\horn{2}{1} \ito \simp{2}$;
		\item strongly closed under composition if and only if the marked core inclusion ${\mcore(X, W)} \ito X$ is an inner fibration;
		\item closed under 2-out-of-3 if and only if the marked core inclusion $\mcore(X, W) \ito X$ is a Kan fibration.
	\end{enumerate}
	In particular, closure under 2-out-of-3 implies strong closure under composition.

	If $X$ is a quasicategory then condition (1) is equivalent to $\mcore(X, W)$ being a quasicategory.
	Thus, if $W$ is strongly closed under composition then it is weakly closed under composition.
	This is not true in general, with the maximally marked horn $\maxm{(\horn{2}{1})}$ providing a counter-example.
\end{remark}

Note the join of simplicial sets naturally ascends to a monoidal product on marked simplicial sets as follows: 
given marked simplicial sets $(X, W)$ and $(X', W')$, we write $(X, W) \join (X', W')$ for the marking on $X \join X'$ given by the images of $W$ and $W'$ under the join inclusions, respectively.
The induced endofunctors
\[ - \join (X, W) , (X, W) \join - \from \msSet \to \msSet \]
have right adjoint defined on an object $(X', W', f\from (X, W) \to (X', W'))$ by
\[ \begin{array}{@{(}r@{ \ \slice  \ }l@{, W')_n \ := \ }l}
	f & X' & (X, W) \slice \msSet \left( (\simp{n} \join X, i_2), (Y, f) \right) \\
	X' & f & (X, W) \slice \msSet \left( (X \join \simp{n}, i_1), (Y, f) \right)
\end{array} \]
where a 1-simplex $(X, W) \join \simp{1} \to (X', W')$ (or $\simp{1} \join (X, W) \to (X', W')$) is marked if it ascends to a map $(X, W) \join \maxm{(\simp{1})} \to (X', W')$ (or $\maxm{(\simp{1})} \join (X, W) \to (X', W')$).
That is, the marked simplices are created by the pre-image of $W$ under the projection map $(f \slice X', W') \to (X', W')$.

We now define localization of quasicategories.
\begin{definition} \label{def:infty-localization}
	Let $(\mc{C}, W)$ be a marked quasicategory.
	The \emph{localization} of $\mc{C}$ at $W$ is a marked map
	\[ \gamma \from (\mc{C}, W) \to \natm{\mc{C}[W^{-1}]} \]
	from $\mc{C}$ to some quasicategory $\mc{C}[W^{-1}]$ with the property that, for any quasicategory $\mc{D}$, the induced map
	\[ \gamma^* \from \sSet(\mc{C}[W^{-1}], \mc{D}) \to \msSet((\mc{C}, W), \natm{\mc{D}}) \]
	is an equivalence of quasicategories.
\end{definition}
\begin{remark}
	We distinguish two ``edge-cases'' of localization.
	First, when $(\mc{C}, W) = \natm{\mc{C}}$, the localization is an equvalence in the Joyal model structure.
	Second, when $(\mc{C}, W) = \maxm{\mc{C}}$, the localization is a fibrant replacement of $\mc{C}$ in the Kan--Quillen model structure.
	This follows from the proof of \cite[Prop.~7.1.3]{cisinski:higher-categories}, which shows that the localization fits into (and can be computed as) a homotopy pushout
	\[ \begin{tikzcd}
		\mcore(\mc{C}, W) \ar[r, hook] \ar[d] & \mc{C} \ar[d, "\gamma"] \\
		\Ex^\infty \left( \mcore(\mc{C}, W) \right) \ar[r] & \mc{C}[W^{-1}]
	\end{tikzcd} \]
	of quasicategories.
	This also shows the localization always exists and is unique up to categorical equivalence.
\end{remark}
For an ordinary marked category $(\mc{C}, W)$, it is generally not the case that the ``$\infty$-localization'' of $(N\mc{C}, W)$ (in the sense of \cref{def:infty-localization}) is categorically-equivalent to the nerve of its ``1-localization'' (in the sense of \cref{def:ordinary-localization}).
For instance, every $\infty$-groupoid arises as the localization of a maximally-marked poset, but not every $\infty$-groupoid is weakly equivalent (in the Joyal model structure) to the nerve of a category.
However, when passing to homotopy categories, the $\infty$-localization of $(N\mc{C}, W)$ becomes the 1-localization of $(\Ho N \mc{C}, W) = (\mc{C}, W)$ (cf.~\cite[Rem.~7.1.6]{cisinski:higher-categories}).

The corresponding notion of homotopy for marked simplicial sets is given by \emph{marked homotopies}.
\begin{definition} \leavevmode
	\begin{enumerate}
		\item For marked maps $f, g \from (X, W) \to (X', W')$, a \emph{marked homotopy} from $f$ to $g$ is a marked map $\alpha \from (X, W) \times \maxm{(\simp{1})} \to (X', W')$ such that $\restr{\alpha}{X \times \{ 0 \}} = f$ and $\restr{\alpha}{X \times \{ 1 \}} = g$.
		\item A map $f \from (X, W) \to (X', W')$ is a \emph{marked homotopy equivalence} if there exists $g \from (X', W') \to (X, W)$ along with marked homotopies $gf \sim \id[(X, W)]$ and $fg \sim \id[(X', W')]$.
	\end{enumerate}
\end{definition}

The following result shows that marked homotopies between preorders are preserved by the join.
\begin{proposition} \label{preorder-marked-htpy-join}
	For any marked simplicial set $(X, W)$, the functors
	\[ - \join (X, W), (X, W) \join - \from \msSet \to \msSet \]
	preserve marked homotopies between preorders (and hence marked homotopy equivalences between preorders).
\end{proposition}
\begin{proof}
	We prove the result for $- \join X$, as the proof for $X \join -$ is analogous.

	For $m \geq 0$ and a preorder $P$, consider the map 
	\[ \alpha \from (P \join \minm{(\simp{m})}) \times \maxm{(\simp{1})} \to (P \times \maxm{(\simp{1})}) \join \minm{(\simp{m})} \] 
	defined, as a preorder map, by
	\[ \alpha (x, t) := \begin{cases}
		(x, t) & x \in P \\
		x & x \in \simp{m}.
	\end{cases} \]
	It is straightforward to verify this map preserves markings and induces a natural transformation
	\[ \begin{tikzcd}
		 {} & P \slice \msSet \ar[d, Rightarrow, "\alpha", shorten=1.5ex] \ar[rd, "{- \times \id[\maxm{(\simp{1})}]}"] & {} \\
		\msSet \ar[ur, "{P \join -}"] \ar[rr, "{(P \times \maxm{(\simp{1})}) \join -}"'] & {} & (P \times \maxm{(\simp{1})}) \slice \msSet
	\end{tikzcd} \]
	via extension by colimits.
	With this, we have that for any marked homotopy $H \from P \times \maxm{(\simp{1})} \to Q$ from $f$ to $g$ and $(X, W) \in \msSet$, the composite
	\[ (P \join (X, W)) \times \maxm{(\simp{1})} \xrightarrow{\alpha_{(X, W)}} (P \times \maxm{(\simp{1})}) \join (X, W) \xrightarrow{H \join (X, W)} Q \join (X, W) \]
	is a marked homotopy from $f \join (X, w)$ to $g \join (X, W)$.
\end{proof}

Recall from \cite{lurie:htt} that $\msSet$ admits a model structure (the model structure for Cartesian fibrations over $\simp{0}$).
\begin{theorem}[{\cite[Props.~3.1.3.7 \& 3.1.4.1]{lurie:htt}}]
	The category $\msSet$ admits a model structure where
	\begin{itemize}
		\item cofibrations are monomorphisms;
		\item fibrant objects are quasicategories marked at equivalences. \qed
	\end{itemize}
\end{theorem}
We also note that marked homotopy equivalences are weak equivalences in this model structure.
The key property of this model structure which we wish to use is the following.
\begin{proposition}[{\cite[Prop.~1.1.3]{hinich:dwyer-kan-localization-revisited}}] \label{msSet-weq-is-localization}
	A marked map of the form $(\mc{C}, W) \to \natm{\mc{D}}$ between quasicategories is a weak equivalence if and only if it is the localization of $\mc{C}$ at $W$. \qed
\end{proposition}

This model structure is Quillen equivalent to the Joyal model structure via the minimal marking adjunction.
\begin{theorem}[{\cite[Prop.~3.1.5.3]{lurie:htt}}]
	The adjunction
	\[ \adjunct{\minm{(\uvar)}}{U}{\sSet}{\msSet} \]
	is a Quillen equivalence between the Joyal model structure on $\sSet$ and the model structure on $\msSet$. \qed
\end{theorem}

\chapter*{$(\infty, 1)$-calculus of fractions}
\cftaddtitleline{toc}{chapter}{$(\infty, 1)$-calculus of fractions}{}

\section{$(\infty, 1)$-calculus of fractions --- definition} \label{sec:calculus-def}

In this section, we introduce the definition of calculus of fractions for quasicategories.
Namely, the marked quasicategories satisfying calculus of (left or right) fractions will be those that satisfy a certain right lifting property.

To describe this right lifing property, we first give an auxiliary construction.
\begin{definition} \label{def:lnk} \leavevmode
	\begin{enumerate}
		\item For $n \geq 0$ and $0 \leq k \leq n$, define a subset $\dfLI$ of the powerset $\mathcal{P}[n]$, viewed as a simplicial set, by
		\[ \dfLI := \{ A \subseteq [n] \mid k \in A \}. \]
		A 1-simplex $A_0 \subseteq A_1$ in $\dfLI$ is marked if $\max(A_0) = \max(A_1)$.
		\item For $n \geq 0$ and $0 \leq k \leq n$, let $\dfRI$ be the opposite simplicial set of $\dfLI$, where a 1-simplex $A_0 \supseteq A_1$ is marked if $\min A_0 = \min A_1$.
		\item For $n \geq 0$ and $0 \leq k \leq n$, let $\dfLJ$ be the maximal simplicial subset of $\dfLI$ which omits the 0-simplex $[n]$.
		\item For $n \geq 0$ and $0 \leq k \leq n$, let $\dfRJ$ be the maximal simplicial subset of $\dfRI$ which omits the 0-simplex $[n]$.
	\end{enumerate}
\end{definition}
\begin{example}
	We depict the pairs $\LJ{2}{1} \subseteq \LI{2}{1}$, $\LJ{2}{2} \subseteq \LI{2}{2}$, and $\LJ{3}{1} \subseteq \LI{3}{1}$, respectively.
	\begin{table}[H]
		\centering
		\begin{subtable}{0.49\textwidth}
			\begin{tabular}{c@{ \ }c@{ \ }c}
				$\begin{tikzcd}
					\{ 1 \} \ar[r, "\sim"] \ar[d] & \{ 0, 1 \} \\
					\{ 1, 2 \}
				\end{tikzcd}$ & $\subseteq$ & $\begin{tikzcd}
					\{ 1 \} \ar[r, "\sim"] \ar[d] & \{ 0, 1 \} \ar[d] \\
					\{ 1, 2 \} \ar[r, "\sim"] & \{ 0, 1, 2 \}
				\end{tikzcd}$
			\end{tabular}
			\caption*{The inclusion $\LJ{2}{1} \subseteq \LI{2}{1}$.}
		\end{subtable} \hfill \begin{subtable}{0.49\textwidth}
			\begin{tabular}{c@{ \ }c@{ \ }c}
				$\begin{tikzcd}
					\{ 2 \} \ar[r, "\sim"] \ar[d, "\sim"'] & \{ 0, 2 \} \\
					\{ 1, 2 \}
				\end{tikzcd}$ & $\subseteq$ & $\begin{tikzcd}
					\{ 2 \} \ar[r, "\sim"] \ar[d, "\sim"'] & \{ 0, 2 \} \ar[d, "\sim"] \\
					\{ 1, 2 \} \ar[r, "\sim"] & \{ 0, 1, 2 \}
				\end{tikzcd}$
			\end{tabular}
			\caption*{The inclusion $\LJ{2}{2} \subseteq \LI{2}{2}$.}
		\end{subtable}
	\end{table}
	\begin{table}[H]
		\centering
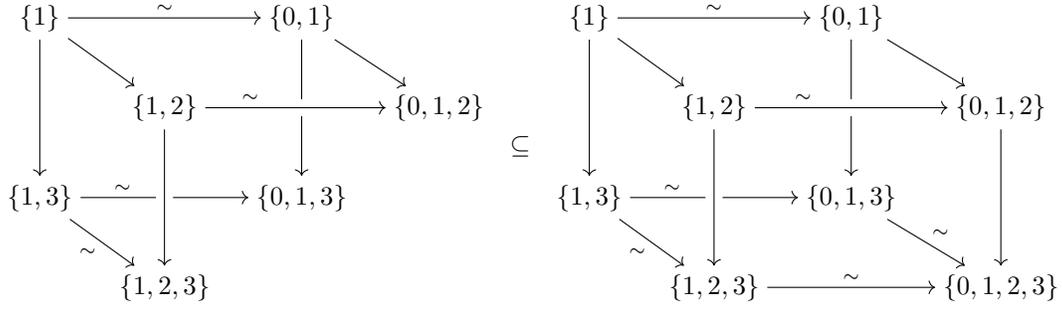
		\[ \begin{array}{c c c}
			\begin{tikzcd}[cramped, column sep = small]
				\{ 1 \} \ar[rr, "\sim"] \ar[dd] \ar[rd] & {} & \{ 0, 1 \} \ar[dd] \ar[rd] & {} \\
				{} & \{ 1,2 \} \ar[rr, crossing over, "\sim" near start]  & {} & \{ 0, 1, 2 \} \\
				\{ 1,3 \} \ar[rr, "\sim" near start] \ar[rd, "\sim"'] & {} & \{0,1, 3 \} & {} \\
				{} & \{ 1, 2, 3 \} \ar[from=uu, crossing over] & {} & {}
			\end{tikzcd} & \subseteq & \begin{tikzcd}[cramped, column sep = small]
				\{ 1 \} \ar[rr, "\sim"] \ar[dd] \ar[rd] & {} & \{ 0, 1 \} \ar[dd] \ar[rd] & {} \\
				{} & \{ 1,2 \} \ar[rr, crossing over, "\sim" near start]  & {} & \{ 0, 1, 2 \} \ar[dd] \\
				\{ 1,3 \} \ar[rr, "\sim" near start] \ar[rd, "\sim"'] & {} & \{0,1, 3 \} \ar[rd, "\sim"] & {} \\
				{} & \{ 1, 2, 3 \} \ar[from=uu, crossing over] \ar[rr, "\sim"] & {} & \{ 0, 1, 2, 3 \}
			\end{tikzcd}
		\end{array} \]
		\caption*{The inclusion $\LJ{3}{1} \subseteq \LI{3}{1}$.}
	\end{table}
	We also depict the pairs $\RJ{2}{0} \subseteq \RI{2}{0}$, $\RJ{2}{1} \subseteq \RI{2}{1}$, and $\RJ{3}{2} \subseteq \RI{3}{2}$, respectively.
	\begin{table}[H]
		\centering
		\begin{subtable}[b]{0.49\textwidth}
			\[ \begin{array}{c@{ \ }c@{ \ }c}
				\begin{tikzcd}
					{} & \{ 0, 1 \} \ar[d, "\sim"] \\
					\{ 0, 2 \} \ar[r, "\sim"] & \{ 0 \}
				\end{tikzcd} & \subseteq & \begin{tikzcd}
					\{ 0, 1, 2 \} \ar[r, "\sim"] \ar[d, "\sim"'] & \{ 0, 1 \} \ar[d, "\sim"] \\
					\{ 0, 2 \} \ar[r, "\sim"] & \{ 0 \}
				\end{tikzcd}
			\end{array} \]
			\caption*{The inclusion $\RJ{2}{0} \subseteq \RI{2}{0}$}
		\end{subtable} \hfill \begin{subtable}[b]{0.49\textwidth}
			\[ \begin{array}{c@{ \ }c@{ \ }c}
				\begin{tikzcd}
					{} & \{ 0, 1 \} \ar[d] \\
					\{ 1, 2 \} \ar[r, "\sim"] & \{ 1 \}
				\end{tikzcd} & \subseteq & \begin{tikzcd}
					\{ 0, 1, 2 \} \ar[r, "\sim"] \ar[d] & \{ 0, 1 \} \ar[d] \\
					\{ 1, 2 \} \ar[r, "\sim"] & \{ 1 \}
				\end{tikzcd}
			\end{array} \]
			\caption*{The inclusion $\RJ{2}{1} \subseteq \RI{2}{1}$.}
		\end{subtable}
	\end{table}
	\begin{table}[H]
		\centering
		\[ \begin{array}{c c c}
			\begin{tikzcd}[cramped, column sep = small]
				{} & {} & \{ 1, 2, 3 \} \ar[dd, "\sim" near end] \ar[rd] & {} \\
				{} & \{ 0, 2, 3 \} \ar[rr, crossing over]  & {} & \{ 2, 3 \} \ar[dd, "\sim"] \\
				\{ 0, 1, 2 \} \ar[rr] \ar[rd, "\sim"'] & {} & \{ 1, 2 \} \ar[rd] & {} \\
				{} & \{ 0, 2 \} \ar[from=uu, crossing over, "\sim", near end] \ar[rr] & {} & \{ 2 \}
			\end{tikzcd} & \subseteq & \begin{tikzcd}[cramped, column sep = small]
				\{ 0, 1, 2, 3 \} \ar[rr] \ar[dd, "\sim"'] \ar[rd, "\sim"] & {} & \{ 1, 2, 3 \} \ar[dd, "\sim" near end] \ar[rd] & {} \\
				{} & \{ 0, 2, 3 \} \ar[rr, crossing over]  & {} & \{ 2, 3 \} \ar[dd, "\sim"] \\
				\{ 0, 1, 2 \} \ar[rr] \ar[rd, "\sim"'] & {} & \{ 1, 2 \} \ar[rd] & {} \\
				{} & \{ 0, 2 \} \ar[from=uu, crossing over, "\sim", near end] \ar[rr] & {} & \{ 2 \}
			\end{tikzcd}
		\end{array} \]
		\caption*{The inclusion $\RJ{3}{2} \subseteq \RI{3}{2}$.}
	\end{table}
\end{example}
The isomorphism $F \from [n]^\op \to [n]$ defined by $F(i) = n - i$ induces a map 
\[ F_* \from (\dfLI)^\op \to \dfRI[n-k] \] 
which sends a subset to its image under $F$.
\begin{proposition} \label{lnk-rnk-iso}
	The map $F_* \from (\dfLI)^\op \to \dfRI[n-k]$ is an isomorphism.
	Moreover, $F_*$ restricts to an isomorphism
	\[ \restr{F_*}{\dfLJ} \from (\dfLJ)^\op \xrightarrow{\cong} \dfRJ[n-k]. \]
	\qed
\end{proposition}
\begin{definition} \label{def:infty-clf}
	Let $(\mc{C}, W)$ be a marked quasicategory.
	\begin{enumerate}
		\item We say $(\mc{C}, W)$ satisfies \emph{calculus of left fractions} (or CLF) if $W$ is weakly closed under composition and $(\mc{C}, W)$ has the right lifting property against the set
		\[ \{ \dfLJ \ito \dfLI \mid n \geq 2, \ 0 < k \leq n \}. \]
		That is, for $n \geq 2$ and $0 < k \leq n$, any map $\dfLJ \to (\mc{C}, W)$ lifts to a map $\dfLI \to (\mc{C}, W)$.
		\[ \begin{tikzcd}
			\dfLJ \ar[r] \ar[d, hook] & (\mc{C}, W) \\
			\dfLI \ar[ur, dotted]
		\end{tikzcd} \]
		\item We say $(\mc{C}, W)$ satisfies \emph{calclulus of right fractions} (or CRF) if $W$ is weakly closed under composition and $(\mc{C}, W)$ has the right lifting property against the set
		\[ \{ \dfRJ \ito \dfRI \mid n \geq 2, \ 0 \leq k < n \}. \]
	\end{enumerate}
\end{definition}
\begin{remark}
	It follows from \cref{lnk-rnk-iso} that $(\mc{C}, W)$ satisfies CLF if and only if $(\mc{C}^\op, W)$ satisfies CRF. 
\end{remark}

A priori, the quasicategorical definition of calculus of fractions may not agree with the classical definition when viewing an ordinary category as a quasicategory.
We will show in \cref{1-cat-crf} that the quasicategorical notion agrees with that of proper calculus of fractions (\cref{def:proper-clf}).
As our focus going forward will be on quasicategories, we use the terms \emph{calculus of left fractions} and \emph{CLF} (or \emph{calculus of right fractions} and \emph{CRF}, respectively) to mean the \emph{quasi-categorical} notion unless otherwise specified.

The following proposition establishes that the $k=n$ lifting condition in CLF (and the $k=0$ condition in CRF) is redundant.
\begin{proposition} \label{crf-implies-good-crf}
	Let $(X, W)$ be a marked simplicial set.
	\begin{enumerate}
		\item If $(X, W)$ has the right lifting property against $\{ \dfLJ[n-1] \to \dfLI[n-1] \mid n \geq 2 \}$ then it has the right lifting property against $\{ \dfLJ[n] \to \dfLI[n] \mid n \geq 2 \}$.
		\item If $(X, W)$ has the right lifting property against $\{ \dfRJ[1] \to \dfRI[1] \mid n \geq 2 \}$ then it has the right lifting property against $\{ \dfRJ[0] \to \dfRI[0] \mid n \geq 2 \}$.
	\end{enumerate} 
\end{proposition}
\begin{proof}
	We prove (1), as (2) is dual.
	Given $n \geq 2$, the inclusion $\LI{n}{n} \to \LI{n+1}{n}$ defined, as a poset map, by $A \mapsto A \cup \{ n+1 \}$ has a retraction $r \from \LI{n+1}{n} \to \LI{n}{n}$ defined by
	\[ r(A) := \begin{cases}
		\sigma_n(S) & \text{if $n+1 \in S$} \\
		\{ n \} & \text{otherwise}.
	\end{cases} \]
	This map restricts so that the diagram
	\[ \begin{tikzcd}
		\LJ{n}{n} \ar[r, hook] \ar[d] & \LJ{n+1}{n} \ar[r, dotted, "r \vert_{\LJ{n+1}{n}}"] \ar[d] & \LJ{n}{n} \ar[d] \\
		\LI{n}{n} \ar[r, hook] & \LI{n+1}{n} \ar[r, "r"] & \LI{n}{n}
	\end{tikzcd} \]
	commutes.
	Thus, $\LJ{n}{n} \to \LI{n}{n}$ is a retract of $\LJ{n+1}{n} \to \LI{n+1}{n}$.
\end{proof}

\section{$(\infty,1)$-calculus of fractions vs.~classical calculus of fractions} \label{sec:clf-infty-vs-classical}

We now show that the quasicategorical definition of calclulus of fractions agrees with that of \emph{proper} calculus of fractions, as stated in \cref{def:proper-clf} (cf.~\cref{1-cat-crf}).

To show this, we use the following technical lemma which encapsulates our usage of the ``co-equalising'' condition from classical CLF (i.e.~condition (3) of \cref{def:classical-clf}).

\begin{lemma} \label{crf-inductive-square-equaliser-lemma}
	Let $(\mc{C}, W)$ be a marked category satisfying CLF.
	For $m \geq 1$, let
	\[ \{ f_i, g_i \from x_i \rightrightarrows y \mid i = 1, \dots, m \} \]
	be a set of $m$ parallel pairs of morphisms with identical codomains.
	If, for each $i$, there exists a weak equivalence $w_i$ such that $f_i w_i = g_i w_i$ then there exists a weak equivalence $u$ such that $u f_i = u g_i$ for all $i = 1, \dots, m$.
\end{lemma}
\begin{proof}
	We fix $m$ and construct $u$ by induction over $i$.
	For $i = 1$, this holds by definition of CLF.

	Fix $i$ and suppose there exists a weak equivalence $u^i$ such that $u^i f_j = u^i g_j$ for all $j = 1, \dots, i$.
	As $u^i f_{i+1} w_{i+1} = u^i g_{i+1} w_{i+1}$ and $(\mc{C}, W)$ satisfies CLF, there exists a weak equivalence $u'$ such that $u' u^i f_{i+1} = u' u^i g_{i+1}$.
	Setting $u = u' u^i$ gives that $u f_{j} = u g_j $ for $j = 1, \dots, i+1$.
\end{proof}

\begin{theorem} \label{1-cat-crf} 
	Let $(\mc{C},W)$ be a category with weak equivalenes. 
	\begin{enumerate}
		\item The pair $(\mc{C}, W)$ satisfies proper CLF if and only if $(N \mc{C}, W)$ satisfies CLF.
		\item The pair $(\mc{C}, W)$ satisfies proper CRF if and only if $(N \mc{C}, W)$ satisfies CRF.
	\end{enumerate}
\end{theorem}
\begin{proof}
	We show (1), as (2) is dual.
	
	Suppose $(N \mc{C}, W)$ satisfies CLF.
	Given a diagram,
	\[ \begin{tikzcd}
		\cdot \ar[r, "(\sim)"] \ar[d, "\sim"'] & \cdot \\
		\cdot & {}
	\end{tikzcd} \]
	a filler exists as $(N \mc{C}, W)$ has the right lifting property against $\LJ{2}{1} \ito \LI{2}{1}$ and $\LJ{2}{0} \ito \LI{2}{0}$.

	For $f, g \in \mc{C}$ and $w \in W$ such that $fw = gw$,
	define a map $\LJ{3}{1} \to (N \mc{C}, W)$ by the diagram:
	\[ \begin{tikzcd}
		\cdot \ar[rr, "w"] \ar[rd, "fw"] \ar[dd, "gw"'] & {} & \cdot \ar[rd, "f"] \ar[dd, "g", near end] & {} \\
		{} & \cdot \ar[rr, equal, crossing over] & {} & \cdot \\
		\cdot \ar[rr, equal] \ar[rd, equal] & {} & \cdot & {} \\
		{} & \cdot \ar[from=uu, crossing over, equal] & {} & 
	\end{tikzcd} \]
	This map admits a lift $\LI{3}{1} \to (N \mc{C}, W)$ as in the diagram:
	
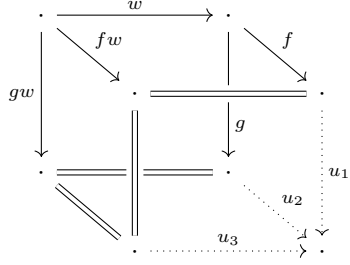
\[ \begin{tikzcd}
		\cdot \ar[rr, "w"] \ar[rd, "fw"] \ar[dd, "gw"'] & {} & \cdot \ar[rd, "f"] \ar[dd, "g", near end] & {} \\
		{} & \cdot \ar[rr, equal, crossing over] & {} & \cdot \ar[dd, dotted, "u_1"] \\
		\cdot \ar[rr, equal] \ar[rd, equal] & {} & \cdot \ar[rd, dotted, "u_2"] & {} \\
		{} & \cdot \ar[from=uu, crossing over, equal] \ar[rr, dotted, "u_3"] & {} & \cdot
	\end{tikzcd} \]
	The right face witnesses that $u_1 f = u_2 g$.
	The maps $u_2$ and $u_3$ are weak equivalences by assumption.
	The front and bottom faces witness that $u_1 = u_2 = u_3$, which suffices.

	Now, suppose $(\mc{C}, W)$ satisfies proper CLF.
	Weak equivalences are (strongly) closed under composition by assumption, so it remains to show $(N \mc{C}, W)$ has the right lifting property with respect to the set of maps
	\[ \{ \dfLJ \to \dfLI \mid n \geq 2, \ 0 < k \leq n \}. \]
	We proceed by induction on $n$.
	The base case $n = 2$ follows by definition of proper CLF.

	Fix $f \from \LJ{n+1}{k} \to (\mc{C}, W)$ and suppose $(N \mc{C}, W)$ has the right lifting property against
	\[ \{ \dfLJ \to \dfLI \mid 0 < k \leq n \}. \]
	It suffices to construct an object $x \in \mc{C}$ and, for $i \neq k$, a morphism 
	\[ \varphi_i \from f([n+1] - \{ i \}) \to x \]
	which is a weak equivalence for $i \neq n+1$ and such that the square
	\[ \begin{tikzcd}
		f([n+1] - \{i, j\}) \ar[r] \ar[d] & f([n+1] - \{ j \}) \ar[d, "\varphi_j"] \\
		f([n+1] - \{ i \}) \ar[r, "\varphi_{i}"] & x
	\end{tikzcd} \]
	commutes for all $i, j = 0, \dots, n+1$ where $i, j \neq k$.
	
	Let $d$ denote the value $\partial^{n}_{k}(n)$. That is,
	\[ d := \begin{cases}
		n & \text{if } k = n+1 \\
		n+1 & \text{otherwise.}
	\end{cases} \]
	The inclusion $\LJ{n}{n} \to \LJ{n+1}{k}$ defined by $S \mapsto \partial^n_k(S) \cup \{ k \}$ gives a lifting problem
	\[ \begin{tikzcd}
		\LJ{n}{n} \ar[r, hook] \ar[d] & \LJ{n+1}{k} \ar[r, "f"] & (N \mc{C}, W) \\
		\LI{n}{n} \ar[urr, dotted, "g_0"']
	\end{tikzcd} \]
	which admits a lift $g_0 \from \LI{n}{n} \to (\mc{C}, W)$ by the inductive hypothesis.
	Let $x_0$ denote the object $\restr{g_0}{ \{ [n] \} }$ and, for $i = 0, \dots, n+1$ such that $i \neq k, d$, let 
	\[ \varphi_{0,i} \from f([n+1] - \{ i \}) \to x_0 \] 
	denote the image under $g_0$ of the 1-simplex $[n] - \{ (\face{n}{k})^{-1}(i) \} \subseteq [n]$.
	Note this map is a weak equivalence except when $d = n$ and $i = n+1$.

	The diagram
	\[ \begin{tikzcd}
		f([n+1] - \{ 0,d \}) \ar[d, "\sim"'] \ar[r] & f([n+1] - \{ 0 \}) \ar[r, "\varphi_{0,0}", "\sim"'] & x_0 \\
		f([n+1] - \{ d \}) & {} & {}
	\end{tikzcd} \]
	admits a lift
	\[ \begin{tikzcd}
		f([n+1] - \{ 0, d \}) \ar[d, "\sim"'] \ar[r] & f([n+1] - \{ 0 \}) \ar[r, "\varphi_{0,0}", "\sim"'] & x_0 \ar[d, "w", "\sim"'] \\
		f([n+1] - \{ d \}) \ar[rr, dotted, "\varphi_{1,d}"] & {} & x_1
	\end{tikzcd} \]
	thus defining a map $\varphi_{1,d}$.
	If $k = n+1$ then $d = n$, which gives that the top map (and hence the bottom map) is a weak equivalence.
	For $i \neq k, d$, let
	\[ \varphi_{1,i} \from f([n+1] - \{ i \}) \to x_1 \]
	denote the composite $\varphi_{0,i} \circ w$.
	Note that $\varphi_{1,i}$ is a weak equivalence for $i \neq n+1$.

	Fix $i, j = 0, \dots, n+1$ such that $i, j \neq k$.
	If $i, j \neq d$ then the square
	\[ \begin{tikzcd}
		f([n+1] - \{ i, j \}) \ar[r] \ar[d] & f([n+1] - \{ i \}) \ar[d, "\varphi_{1,i}"] \\
		f([n+1] - \{ j \}) \ar[r, "\varphi_{1,j}"] & x_1
	\end{tikzcd} \]
	commutes by construction.
	Otherwise, we have that every face except the front face in the cube
	\[ \begin{tikzcd}[cramped, column sep = small]
		f([n+1] - \{0, i, d\}) \ar[rr] \ar[dd] \ar[rd, "\sim"] & {} & f([n+1] - \{ 0, i \}) \ar[dd] \ar[rd] \\
		{} & f([n+1] - \{ i, d \}) \ar[rr, crossing over] & {} & f([n+1] - \{ i \}) \ar[dd] \\
		f([n+1] - \{ 0, d \}) \ar[rr] \ar[rd] & {} & f([n+1] - \{ 0 \}) \ar[rd] & {} \\
		{} & f([n+1] - \{ d \}) \ar[rr] \ar[from=uu, crossing over] & {} & x_1
	\end{tikzcd} \]
	commutes.
	Thus, the front face commutes after pre-composition with the top-left weak equivalence.
	By \cref{crf-inductive-square-equaliser-lemma}, there exists a weak equivalence
	\[ u \from x_1 \to x \]
	such that $x$ and $\varphi_i = u \circ \varphi_{1,i}$ assemble to give a lift of $f \from \LJ{n+1}{k} \to (\mc{C}, W)$.
\end{proof}
\begin{remark}
	The proof of the reverse direction in \cref{1-cat-crf} proves a slightly stronger result: given a marked quasicategory $(\mc{C}, W)$, if
	\begin{itemize}
		\item $W$ is weakly closed under composition; and
		\item $(\mc{C}, W)$ has the right lifting property against the maps $\LJ{2}{1} \ito \LI{2}{1}$ and $\LJ{3}{1} \ito \LI{3}{1}$,
	\end{itemize}
	then its homotopy category $\Ho \mc{C}$, when marked at all morphisms which are identified with a morphism in $W$, satisfies classical CLF.
	If $(\mc{C}, W)$ additionally has the right lifting property against $\LJ{2}{2} \ito \LI{2}{2}$ then $\Ho \mc{C}$ at this marking satisfies proper CLF.
	The dual statement holds for (proper) CRF.
\end{remark}

\section{Examples of calculus of fractions} \label{sec:clf-examples}

This section is devoted to giving examples of marked quasicategories (beyond nerves of 1-categories) which satisfy calculus of fractions.
In particular, analogues of the examples given in \cite[Ch.~1]{gabriel-zisman} are covered in this section.
These include: marking at equivalences (\cref{equivs-satisfy-cf}), marking at the inverse image of equivalences under a limit-preserving functor (\cref{limits-create-crf}), and pulling back CLF/CRF from an adjunction (\cref{clf-from-adjunction}).
Note that this last example applies to all reflective localizations (\cref{reflective-localization-satisfies-clf}).

We also introduce the notion of a \emph{simple inner horn decomposition} in \cref{def:simple-inner-horn-decomp}.
This is a technical condition for showing when a marked map has the left lifting property against all marked quasicategories $(\mc{C}, W)$ where $W$ is weakly closed under composition (see \cref{anodyne-decomp-is-anodyne}).
In particular, we use this both in the proof of \cref{clf-from-adjunction} and in \cref{sec:mEx-is-qcat} to show that our analogue of the category of fractions produces a quasicategory (\cref{mEx-qcat}).
We conclude by showing that if a marked quasicategory satisfies calculus of fractions then the full subcategory of the arrow category spanned by marked arrows also satisfies calculus of fractions (\cref{arr-cat-clf}).
This requires a lemma relating right lifting properties against the maps $\dfLJ \ito \dfLI$ between a marked quasicategory and its marked arrow category (\cref{lj-mapping-space-lift-lemma}), which we use in \cref{sec:mapping-space} when considering the ``marked slice'' of a marked quasicategory (see \cref{pathloop-qcat-clf}).

To show that a quasicategory marked at equivalences satisfies CLF (and CRF), we make use of the following result.
\begin{proposition} \label{jnk-ink-weq} \leavevmode
	\begin{enumerate}
		\item \begin{enumerate}
			\item For $n \geq 1$ and $0 \leq k \leq n$, the map $\max \from \dfLI \to \minm{(\simp{\{ k, \dots, n \}})}$ is a marked homotopy equivalence.
			\item For $n \geq 2$ and $0 < k \leq n$, the inclusion $\dfLJ \to \dfLI$ is an acyclic cofibration.
		\end{enumerate}
		\item \begin{enumerate}
			\item For $n \geq 1$ and $0 \leq k \leq n$, the map $\min \from \dfRI \to \minm{(\simp{k})}$ is a marked homotopy equivalence.
			\item For $n \geq 2$ and $0 \leq k < n$, the inclusion $\dfRJ \to \dfRI$ is an acyclic cofibration.
		\end{enumerate}
	\end{enumerate}
\end{proposition}
\begin{proof}
	We prove (1), as (2) is formally dual.
	\begin{enumerate}[label=\alph*)]
		\item The map $\max \from \LI{n}{k} \to \minm{(\simp{\{k, \dots, n\}})}$ has a section defined by $i \mapsto \{ 0, \dots, i \}$.
		For $A \in \LI{n}{k}$, the 1-simplex $A \subseteq \{ 0, \dots, \max A \}$ is marked, giving the desired homotopy.
		\item The inclusion $\face{}{0} \from \LI{n-1}{k-1} \to \dfLJ$ defined by sending a subset to its image under $\face{}{0}$ admits a retraction defined by $A \mapsto \degen{}{0}(A - \{ 0 \})$.
		For $A \in \dfLJ$, the 1-simplex $\face{}{0}(\degen{}{0}(A - \{ 0 \})) \subseteq A$ is marked as $\max A \geq k > 0$ (hence $A - \{ 0 \}$ contains $\max A$).
		Thus, $\face{}{0}$ is a marked homotopy equivalence.
		The diagram
		\[ \begin{tikzcd}
			\LI{n-1}{k-1} \ar[d, "\sim", "\max"'] \ar[r, hook, "\sim"', "\face{}{0}"] & \dfLJ \ar[r, hook] & \dfLI \ar[d, "\sim"', "\max"] \\
			\simp{\{ k-1, \dots, n-1 \}} \ar[rr, "\cong"] & {} & \simp{\{ k, \dots, n \}}
		\end{tikzcd} \]
		commutes, hence the top right map is a weak equivalence by 2-out-of-3 (applied twice). \qedhere
	\end{enumerate}
\end{proof}
\begin{corollary} \label{equivs-satisfy-cf}
	Any quasicategory marked at equivalences satisfies both CLF and CRF.
\end{corollary}
\begin{proof}
	Follows from \cref{jnk-ink-weq}, as quasicategories marked at equivalences are fibrant objects in the model structure on $\msSet$.
\end{proof}

Our second example of when CLF is satisfied is in the case where weak equivalences are created by a colimit-preserving functor.
\begin{theorem} \label{limits-create-crf}
	Let $F \from \mc{C} \to \mc{D}$ be a functor between quasicategories and $W \subseteq \mc{C}_1$ be the 1-simplices which are sent to equivalences under $F$.
	\begin{enumerate}
		\item If $\mc{C}$ admits all finite colimits and $F$ preserves finite colimits then $(\mc{C}, \mc{W})$ satisfies CLF.
		\item If $\mc{C}$ admits all finite limits and $F$ preserves finite limits then $(\mc{C}, W)$ satisfies CRF.
	\end{enumerate}
\end{theorem}
Before proving \cref{limits-create-crf}, we introduce an additional construction.
This construction is only used to prove \cref{limits-create-crf} and will not be needed going forward (though a related construction is introduced in \cref{sec:limits}).

For $n \geq 2$ and $0 \leq k \leq n$, define $\hat{\dfLJ}$ to be the nerve of the preorder whose underlying set is
\[ \{ A \subseteq [n] \mid k \in A \text{ and } A \neq [n] \} \]
and where $A_0 \leq A_1$ if $\max A_0 \leq \max A_1$.
There is an evident inclusion $\dfLJ \subseteq \hat{\dfLJ}$.
The $E^1$-equivalences in $\hat{\dfLJ}$ are exactly the 1-simplices $A_0 \leq A_1$ where $\max A_0 = \max A_1$.
In particular, the map $\max \from \hat{\dfLJ} \to \simp{\{k, \dots, n\}}$ is an $E^1$-equivalence with inverse given by $i \mapsto \{ k, i \}$.
\begin{lemma} \label{limits-create-crf-lemma}
	Let $(X, W)$ be a marked simplicial set.
	For $n \geq 2$ and $0 < k \leq n$, the map
	\[ \dfLJ \join (X, W) \ito \natm{\hat{\dfLJ}} \join (X, W) \]
	is an acyclic cofibration of marked simplicial sets.
\end{lemma}
\begin{proof}
	Applying \cref{preorder-marked-htpy-join} gives that the top and left maps in
	\[ \begin{tikzcd}
		\LI{n-1}{k-1} \join (X, W) \ar[d, "\sim", "{\max \join (X, W)}"'] \ar[r, hook, "\sim"', "{\face{}{0} \join (X, W)}"] & \dfLJ \join (X, W) \ar[d, "{\max \join (X, W)}"] \\
		\simp{\{ k-1, \dots, n-1 \}} \join (X, W) \ar[r, "\cong"] & \simp{\{ k, \dots, n \}} \join (X, W)
	\end{tikzcd} \]
	are marked homotopy equivalences.
	By 2-out-of-3, the right map is a weak equivalence.
	Similarly, $\max \from \hat{\dfLJ} \to \simp{\{ k, \dots, n \}}$ is an $E^1$-homotopy equivalence, hence a weak equivalence of marked simplicial sets.
	The result follows by 2-out-of-3 on the triangle
	\[ \begin{tikzcd}
		\dfLJ \join (X, W) \ar[rr, hook] \ar[rd, "\sim"'] & {} & \natm{\hat{\dfLJ}} \join (X, W) \ar[ld, "\simeq"] \\
		{} & \simp{\{ k, \dots, n \}} \join (X, W) & {}
	\end{tikzcd} \]
\end{proof}

We now proceed with the proof of \cref{limits-create-crf}.
\begin{proof}[Proof of \cref{limits-create-crf}]
	We show (1), as (2) is analogous.
	Fix a map $D \from \dfLJ \to \mc{C}$.
	The underlying simplicial set of $\dfLI$ is the join of the underlying simplicial set of $\dfLJ$ with $\simp{0}$.
	Thus, the colimit cone $\lambda_D$ gives a lift
	\[ \begin{tikzcd}
		\dfLJ \ar[r, "D"] \ar[d, hook] & \mc{C} \\
		\dfLI \ar[ur, dotted, "\lambda_D"']
	\end{tikzcd} \]
	of underlying simplicial sets.
	It remains to show that if $A \in \dfLJ$ is such that $n \in A$ then the component map $D(A) \to \colim D$ is a weak equivalence.
	As $F$ preserves finite colimits and creates weak equivalences, it suffices to show $FD(A) \to \colim (FD)$ is an equivalence in $\mc{D}$.

	By assumption, $FD \from \dfLJ \to \mc{D}$ sends marked simplicies to equivalences.
	In the commutative triangle,
	\[ \begin{tikzcd}
		{} & \natm{\hat{\dfLJ}} \ar[rd, "\max"] & {} \\
		\dfLJ \ar[rr, "\max"] \ar[ur, hook] & {} & \minm{(\simp{\{ k, \dots, n \}})}
	\end{tikzcd} \]
	the right map is an $E^1$-homotopy equivalence and the bottom map is a weak equivalence by \cref{jnk-ink-weq}.
	Thus, the left inclusion is an acyclic cofibration of marked simplicial sets.
	Hence, there exists $\overline{FD} \from \hat{\dfLJ} \to \mc{D}$ such that the triangle
	\[ \begin{tikzcd}
		\dfLJ \ar[r, "FD"] \ar[d, hook, "\sim"'] & \natm{\mc{D}} \\
		\natm{\hat{\dfLJ}} \ar[ur, dotted, "\natm{(\overline{FD})}"']
	\end{tikzcd} \]
	commutes.
	The inclusion $\dfLJ \ito \hat{\dfLJ}$ induces a map
	\[ f \from \overline{FD} \slice \mc{D} \to FD \slice \mc{D}  \]
	by pre-composition.
	It suffices to show $f$ is a trivial fibration: this would imply that the colimit of $\overline{FD}$ (which is an initial object of $\overline{FD} \slice \mc{D}$) is equivalent to the colimit of $FD$.
	A subset $A \in \dfLJ$ such that $n \in A$ is a terminal object in $\hat{\dfLJ}$, hence the component map 
	\[ FD(A) = \overline{FD}(A) \to \colim(\overline{FD}) = \colim (FD) \] 
	is an equivalence in $\mc{D}$.

	Given a commutative square
	\[ \begin{tikzcd}
		\bd \simp{n} \ar[r, "u"] \ar[d, hook] & \overline{FD} \slice \mc{D} \ar[d, "f"] \\
		\simp{n} \ar[r, "v"] & FD \slice \mc{D}
	\end{tikzcd} \]
	the maps $u$ and $v$ may be identified as maps
	\[ u \from \hat{\dfLJ} \join \bd \simp{n} \to \mc{D}, \quad v \from \dfLJ \join \simp{n} \to \mc{D} \]
	making the triangles
	\[ \begin{tikzcd}
		{} & \hat{\dfLJ} \ar[rd, "\overline{FD}"] \ar[ld] & {}  \\
		\hat{\dfLJ} \join \bd \simp{n} \ar[rr, "u"] & {} & \mc{D} 
	\end{tikzcd} \qquad \begin{tikzcd}
		{} & {\dfLJ} \ar[rd, "{FD}"] \ar[ld] & {}  \\
		{\dfLJ} \join \simp{n} \ar[rr, "v"] & {} & \mc{D} 
	\end{tikzcd} \]
	commute.
	These maps assemble into a square
	\[ \begin{tikzcd}
		\dfLJ \join \minm{(\bd \simp{n})} \ar[r, hook] \ar[d, hook] & \dfLJ \join \minm{(\simp{n})} \ar[d, "v"] \\
		\natm{\hat{\dfLJ}} \join \minm{(\bd \simp{n})} \ar[r, "u"] & \natm{\mc{D}}
	\end{tikzcd} \]
	of marked maps, which induces a map $P \to \natm{\mc{D}}$ from the pushout $P$.
	In the square of marked maps
	\[ \begin{tikzcd}
		\dfLJ \join \minm{(\bd \simp{n})} \ar[r, hook] \ar[d, hook] & {\dfLJ} \join \minm{(\simp{n})} \ar[d, hook] \\
		\natm{\hat{\dfLJ}} \join \minm{(\bd \simp{n})} \ar[r, hook] & \natm{\hat{\dfLJ}} \join \minm{(\simp{n})}  
	\end{tikzcd} \]
	the left and right maps are acyclic cofibrations by \cref{limits-create-crf-lemma}.
	Thus, the induced map from the pushout into $\natm{\hat{\dfLJ}} \join \minm{(\simp{n})}$ is as well, since it is both a weak equivalence (by 2-out-of-3) and a monomorphism (by direct verification).
	Thus, the map $P \to \natm{\mc{D}}$ lifts to a map $\natm{\hat{\dfLJ}} \join \minm{(\simp{n})} \to \mc{D}$.
	This gives a lift of the starting square when viewed as a map $\simp{n} \to \overline{FD} \slice \mc{D}$.
\end{proof}

The following result gives a condition for when CLF (or CRF) may be ``pulled back'' from an adjunction.
\begin{theorem} \label{clf-from-adjunction}
	Let $L \from \mc{C} \to \mc{D}$ be a map between quasicategories with right adjoint $R \from \mc{D} \to \mc{C}$.
	\begin{enumerate}
		\item Suppose $W \subseteq \mc{D}_1$ is a marking on $\mc{D}$ such that $(\mc{D}, W)$ satisfies CLF.
		If every component of the unit is sent by $L$ to a marked 1-simplex $L \eta_x \from Lx \to LRLx$ then $(\mc{C}, L^{-1} W)$ satisfies CLF.
		\item Suppose $W \subseteq \mc{C}_1$ is a marking on $\mc{C}$ such that $(\mc{C}, W)$ satisfies CRF.
		If every component of the counit is sent by $R$ to a marked 1-simplex $R \varepsilon_y \from RLR y \to Ry$ then $(\mc{D}, R^{-1} W)$ satisfies CRF.
	\end{enumerate}
\end{theorem}

This result will require a more technical proof, in particular, we will construct explicit cellular decompositions in the marked model structure.
Before doing so, we recall a consequence of \cref{clf-from-adjunction} which follows from \cref{equivs-satisfy-cf}.
\begin{corollary} \label{reflective-localization-satisfies-clf}
	Let $i \from \mc{D} \to \mc{C}$ be a fully faithful functor between quasicategories.
	\begin{enumerate}
		\item If $i$ admits a left adjoint $L \from \mc{C} \to \mc{D}$ then $(\mc{C}, L^{-1}(\operatorname{Equiv}\mc{D}))$ satisfies CLF, where $L^{-1}(\operatorname{Equiv}\mc{D})$ denotes the set of morphisms in $\mc{C}$ which are sent to equivalences in $\mc{D}$.
		\item If $i$ admits a right adjoint $R \from \mc{C} \to \mc{D}$ then $(\mc{C}, R^{-1}(\operatorname{Equiv} \mc{D}))$ satisfies CRF, where $R^{-1}(\operatorname{Equiv}\mc{D})$ denotes the set of morphisms in $\mc{C}$ which are sent to equivalences in $\mc{D}$. \qed
	\end{enumerate}
\end{corollary}

As for the proof of \cref{clf-from-adjunction}, we begin by introducing the definition of a \emph{simple inner horn decomposition}, which will also be used in \cref{sec:mEx-is-qcat}.
\begin{definition} \label{def:simple-inner-horn-decomp}
	For an injective map $f \from (X, W) \ito (Y, V)$ between marked simplicial sets, a \emph{simple inner horn decomposition} for $f$ consists of
	\begin{itemize}
		\item for $n \geq 0$, a partition of the set of non-degenerate $n$-simplices in $Y$ which are not in the image of $f$ into a pair of sets $A^n \sqcup B^n$ such that $A^1$, $A^0$, and $B^0$ are empty;
		\item for $n \geq 1$, a pair of finite partitions
		\[ A^n = A^n_1 \sqcup \dots \sqcup A^n_{a(n)}, \quad B^n = B^n_1 \sqcup \dots \sqcup B^n_{b(n)} \]
		such that $b(1) = 1$ and $a(n+1) = b(n)$;
		\item for $n \geq 2$, a function $d \from \{ 1, \dots, a(n) \} \to \{ 1, \dots, n-1 \}$;
	\end{itemize}
	such that for $n \geq 2$ and $k \in \{ 1, \dots, a(n) \}$,
	\begin{enumerate}
		\item the $d(k)$-th face map restricts to a bijection
		\[ \restr{\face{}{d(k)}}{A^n_k} \from A^n_k \xrightarrow{\cong} B^{n-1}_k; \]
		\item for $u \in A^2_1$, if $u\face{}{1}$ is marked then every face of $u$ is marked;
		\item for $u \in A^n_k$ and $i \neq d(k)$, writing $u\face{}{i}$ as a degeneracy $\vec{\degen{}{}} \from \simp{n-1} \to \simp{p}$ of some non-degenerate $p$-simplex $v \from \simp{p} \to Y$, one of the following holds:
		\begin{itemize}
			\item $v$ is contained in the image of $f$;
			\item $v$ is contained in $A^p_i$ for some $i \in \{ 1, \dots, a(p) \}$;
			\item $p < n-1$ and $v$ is contained in $B^p_i$ for some $i \in \{ 1, \dots, b(p) \}$;
			\item $p = n-1$ and $v$ is contained in $B^{n-1}_i$ for some $i \in \{ 1, \dots, k-1 \}$.
		\end{itemize}
	\end{enumerate}
\end{definition}
We say $f$ \emph{admits a simple inner horn decomposition} if there exists a simple inner horn decomposition for $f$.
\begin{lemma} \label{anodyne-decomp-is-anodyne}
	Let $f \from (X, W) \to (Y, V)$ be an injective map of marked simplicial sets.
	If $f$ admits a simple inner horn decomposition then $f$ is in the saturation of the set
	\[ \{ \minm{(\horn{n}{k})} \ito \minm{(\simp{n})} \mid n \geq 2, \ 0 < k < n \} \cup \{ \maxm{(\horn{2}{1})} \ito \maxm{(\simp{2})} \}. \]
\end{lemma}
\begin{proof}
	Let $Y^1_1 \subseteq (Y, V)$ denote the image of $f$, viewed as a marked simplicial subset of $(Y, V)$.
	For $n \geq 2$ and $k \in \{ 1, \dots, a(n) \}$, define a sequence of marked simplicial subsets of $Y$ inductively by
	\[ Y^n_k := \begin{cases}
		Y^{n-1}_{a(n)} \cup B^{n-1}_{1} \cup A^n_1 & \text{if $k = 1$} \\
		Y^n_{k-1} \cup B^{n-1}_{k} \cup A^n_k & \text{if $k > 1$}. 
	\end{cases} \]
	The conditions for a simple inner horn decomposition assert that each $Y^n_k$ is a well-defined simplicial subset of $Y$, i.e.~that if $u$ is a non-degenerate simplex of $Y^n_k$ then every face of $u$ is also contained in $Y^n_k$.
	By construction, there is a sequence of inclusions
	\[ \begin{array}{c @{ \ } c @{ \ } c @{ \ } c}
		Y^1_1 &\subseteq Y^2_1 \\
		&\subseteq Y^3_1 &\subseteq \dots &\subseteq Y^3_{a(3)}  \\
		& \vdots \\
		&\subseteq Y^n_1 &\subseteq \dots &\subseteq Y^n_{a(n)} \\
		& \vdots
	\end{array} \]
	and the colimit of this sequence is $Y$.
	We show each inclusion in this sequence is a pushout of a coproduct of maps in the set 
	\[ \{ \minm{(\horn{n}{k})} \ito \minm{(\simp{n})} \mid n \geq 2, \ 0 < k < n \} \cup \{ \maxm{(\horn{2}{1})} \ito \maxm{(\simp{2})} \}. \]

	For a 2-simplex $u \in A^2_1$, the restriction $\restr{u}{\horn{2}{1}} \from \horn{2}{1} \to Y$ to the 1-horn factors through $Y^1_1$ by definition of a simple inner horn decomposition.
	Moreover, if $u\face{}{1} \in B^1_1$ is marked then every face of $u$ is marked.
	As $\restr{\face{}{1}}{B^2_1} \from A^2_1 \to B^1_1$ is a bijection, it follows that the square
	\[ \begin{tikzcd}[column sep = 4.7em]
		\coprod\limits_{\substack{u \in A^2_1 \\ u\partial_1 \text{ marked} }} \maxm{(\Lambda^2_1)} \sqcup \coprod\limits_{\substack{u \in A^2_1 \\ u\partial_1 \text{ unmarked}}} \minm{(\Lambda^2_1)} \ar[r, "u \vert_{\Lambda^2_1}"] \ar[d, hook] & Y^1_1 \ar[d, hook] \\
		\coprod\limits_{\substack{u \in A^2_1 \\ u\partial_1 \text{ marked}}} \maxm{(\simp{2})} \sqcup \coprod\limits_{\substack{u \in A^2_1 \\ u\partial_1 \text{ unmarked} }} \minm{(\simp{2})} \ar[r, "u"] & Y^2_1
	\end{tikzcd} \]
	is a pushout in $\msSet$.
	Similarly, for $n \geq 3$ and $k > 1$, the squares
	\[ \begin{tikzcd}[column sep = 4.4em]
		\coprod\limits_{u \in A^n_{1}} \minm{(\Lambda^n_{d(1)})} \ar[r, "u \vert_{\Lambda^n_{d(1)}}"] \ar[d, hook] & Y^{n-1}_{a(n-1)} \ar[d, hook] \\
		\coprod\limits_{u \in A^n_{1}} \minm{(\simp{n})} \ar[r, "u"] & Y^n_1
	\end{tikzcd} \qquad \begin{tikzcd}[column sep = 4.4em]
		\coprod\limits_{u \in A^n_{k}} \minm{(\Lambda^n_{d(k)})} \ar[r, "u \vert_{\Lambda^n_{d(k)}}"] \ar[d, hook] & Y^{n}_{k-1} \ar[d, hook] \\
		\coprod\limits_{u \in A^n_{k}} \minm{(\simp{n})} \ar[r, "u"] & Y^n_k
	\end{tikzcd} \]
	are pushouts.
\end{proof}
We highlight an important benefit of \cref{anodyne-decomp-is-anodyne}, which is that marked quasicategories which are weakly closed under composition are exactly the objects in $\msSet$ with the right lifting property against the set $\{ \minm{(\horn{n}{k})} \ito \minm{(\simp{n})} \mid n \geq 2, \ 0 < k < n \} \cup \{ \maxm{(\horn{2}{1})} \ito \maxm{(\simp{2})} \}$.

The following lemma, which makes use of simple horn decompositions, will then allow us to prove (a generalization of) \cref{clf-from-adjunction}.
\begin{lemma} \label{prod-join-inner-anodyne}
	For a marked poset $(P, W)$ and any subset $Q \subseteq P$, viewed as a discrete set, the induced map $f$ from the pushout in the diagram
	\[ \begin{tikzcd}[cramped, column sep = tiny]
		(P, W) \ar[r, "i_1"] \ar[d, hook] \ar[rd, phantom, "\ulcorner" very near end] & (P, W) \times \maxm{(\simp{1})} \ar[d] \ar[rdd, hook, bend left] & {} \\
		(P \join \simp{0}, W \cup (Q \join \simp{0})_1) \ar[r] \ar[rrd, "i_1 \join \simp{0}", bend right] & \cdot \ar[rd, dotted, "f"] & {} \\
		{} & {} & \left( (P \times \simp{1}) \join \simp{0}, (W \times \simp{1}_1) \cup ((Q \times \{ 0, 1 \}) \join \simp{0})_1 \right)
	\end{tikzcd} \]
	is in the saturation of the set
	\[ \{ \minm{(\horn{n}{k})} \ito \minm{(\simp{n})} \mid n \geq 2, 0 < k< n \} \cup \{ \maxm{(\horn{2}{1})} \ito \maxm{(\simp{2})} \}, \]
\end{lemma}
\begin{proof}
	By \cref{anodyne-decomp-is-anodyne}, it suffices to construct a simple inner horn decomposition for $f$.
	
	Note the codomain of $f$ is the poset $(P \times [1]) \join [0]$.
	We write the cone point as $\top$ and all other elements as $(x, \varepsilon)$ for some $x \in P$ and $\varepsilon = 0, 1$.
	A non-degenerate $n$-simplex in $(P \times [1]) \join [0]$ is an alternating chain.
	An $n$-chain is not contained in the image of $f$ if and only if the first element in the chain is of the form $(x, 0)$ for some $x \in P$ and the last element in the chain is the cone point.
	Let $S^n$ denote the set of such $n$-chains.

	For $n \geq 2$ and $k = 1, \dots, n-1$, define subsets $A^n_k \subseteq S^n$ and $B^{n-1}_k \subseteq S^{n-1}$ by
	\begin{align*}
		A^n_k &:= \left\{ ((x_0, 0) \leq (x_1, \varepsilon_1) \leq \dots \leq (x_{n-1}, \varepsilon_{n-1}) \leq \top) \in S^n \mid \begin{array}{l}
			\varepsilon_{k-1} = 0 \\
			\varepsilon_{k} = 1 \\
			x_{k} = x_{k-1}
		\end{array} \right\} \\
		B^{n-1}_k &:= \left\{ ((x_0, 0) \leq (x_1, \varepsilon_1) \leq \dots \leq (x_{n-2}, \varepsilon_{n-2}) \leq \top) \in S^n \mid \begin{array}{l}
			\varepsilon_{k-1} = 0 \\
			\varepsilon_{k} = 1 \\
			x_{k} \neq x_{k-1}
		\end{array} \right\}.
	\end{align*}
	(For $B^{n-1}_{n-1}$, our convention is that the second and third conditions are vacuous.)
	The sets $A^n_k$ and $B^n_k$ are all mutually disjoint, and the pair
	\[ A^n := A^n_1 \sqcup \dots \sqcup A^n_{n-1}, \quad B^n := B^n_1 \sqcup \dots \sqcup B^n_{n-1} \]
	partitions $S^n$.
	We show the tuple
	\[ \left( \{ A^n_k \}_{k=1}^{n-1}, \{ B^n_k \}_{k=1}^{n}, \id \right) \]
	is a simple inner horn decomposition for $f$.
	
	Fix $n \geq 2$ and $k \in \{ 1, \dots, n-1 \}$.
	\begin{enumerate}
		\item For a chain $((x_0, 0) \leq (x_1, \varepsilon_1) \leq \dots \leq (x_{n-1}, \varepsilon_{n-1}) \leq \top)$ in $A^n_k$, we have that $\varepsilon_{k+1} = 1$.
		As this chain is non-degenerate and $\varepsilon_{k} = \varepsilon_{k+1}$, we have that $x_{k+1} \neq x_{k} = x_{k-1}$.
		Thus, the $k$-th face of $((x_0, 0) \leq (x_1, \varepsilon_1) \leq \dots \leq (x_{n-1}, \varepsilon_{n-1}) \leq \top)$ is an element of $B^{n-1}_{k}$.
		The inverse function is given by
		\begin{align*} 
			&((x_0, 0) \leq \dots \leq (x_{k-1}, 0) \leq (x_k, 1) \leq \dots \leq (x_{n-2}, 1) \leq \top) \\
			&\quad \mapsto ((x_0, 0) \leq \dots \leq (x_{k-1}, 0) \leq (x_{k-1}, 1) \leq (x_k, 1) \leq \dots \leq (x_{n-2}, 1) \leq \top). 
		\end{align*}

		\item For $((x_0, 0) \leq (x_0, 1) \leq \top) \in A^2_1$, the $\face{}{1}$-face $(x_0, 0) \leq \top$ is marked if $x_0 \in Q$.
		This gives that $(x_0, 1) \leq \top$ is marked.
		The 1-simplex $(x_0, 0) \leq (x_0, 1)$ is marked as a product of marked 1-simplices.

		\item For an $n$-simplex $((x_0, 0) \leq (x_1, \varepsilon_1) \leq \dots \leq (x_{n-1}, \varepsilon_{n-1}) \leq \top) \in A^n_k$ and $i \neq k$, the $i$-th face of this $n$-simplex is non-degenerate.
		We proceed by case analysis on $i$:
		\begin{itemize}
			\item if $i < k-1$ then the $i$-th face is an element of $A^{n-1}_{k-1}$;
			\item if $i = k-1$ then we proceed by case analysis on $k$:
			\begin{itemize}
				\item for $k = 1$, the 0-th face does not contain any elements of the form $(x, 0)$.
				Thus, it is in the image of $f$;
				\item for $k \geq 2$, we have that $x_{k-2} \neq x_{k-1}$ as $((x_0, 0) \leq (x_1, \varepsilon_1) \leq \dots \leq (x_{n-1}, \varepsilon_{n-1}) \leq \top)$ is non-degenerate and $\varepsilon_{k-2} = \varepsilon_{k-1} = 0$.
				As $x_k = x_{k-1}$, the $(k-1)$-th face is an element of $B^{n-1}_{k-1}$;
			\end{itemize} 
			\item if $i > k$ then we proceed by case analysis on $k$:
			\begin{itemize}
				\item if $k < n-1$ then $\varepsilon_{n-1} = 1$, thus the $i$-th face is an element of $A^{n-1}_k$.
				\item if $k = n-1$ then $i = n$.
				The $n$-th face does not contain $\top$, hence is contained in the image of $f$. \qedhere
			\end{itemize} 
		\end{itemize}
	\end{enumerate}
\end{proof}

The following theorem is a generalization of \cref{clf-from-adjunction}, giving a condition for when a calculus of fractions within a quasicategory $\mc{C}$ may be pulled back from a smaller class to a larger class.
\begin{theorem} \label{endofunctor-clf}
	Let $W_0 \subseteq W$ be two markings on a quasicategory $\mc{C}$ where $W$ is weakly closed under composition and let $F \from (\mc{C}, W) \to (\mc{C}, W_0)$ be an endofunctor which takes $W$ to $W_0$.
	\begin{enumerate}
		\item If $(\mc{C}, W_0)$ has the right lifting property against the set
		$\{ \dfLJ \ito \dfLI \mid n \geq 2 , \ 0 < k \leq n \}$
		and there exists a marked homotopy
		$ \alpha \from (\mc{C}, W) \times \maxm{(\simp{1})} \to (\mc{C}, W) $
		from $\id$ to $F$ then $(\mc{C}, W)$ satisfies CLF.
		\item If $(\mc{C}, W_0)$ has the right lifting property against the set
		$\{ \dfRJ \ito \dfRI \mid n \geq 2 , \ 0 \leq k < n \}$ 
		and there exists a marked homotopy
		$\alpha \from (\mc{C}, W) \times \maxm{(\simp{1})} \to (\mc{C}, W)$
		from $F$ to $\id$ then $(\mc{C}, W)$ satisfies CRF.
	\end{enumerate} 
\end{theorem}
\begin{proof}
	We prove (1), as (2) is formally dual.
	
	It suffices to show $(\mc{C}, W)$ has the right lifting property against $\{ \dfLJ \ito \dfLI \mid n \geq 2, \ 0 < k \leq n \}$.
	Thus, we fix a map $f \from \dfLJ \to (\mc{C}, W)$.
	The composite $Ff \from \dfLJ \to (\mc{C}, W_0)$ admits a lift $\overline{Ff} \from \dfLI \to (\mc{C}, W_0)$ by assumption.
	This gives a commutative square
	\[ \begin{tikzcd}
		\dfLJ \ar[r, "i_1"] \ar[d, hook] & \dfLJ \times \maxm{(\simp{1})} \ar[d, "\alpha \circ (f \times \id)"] \\
		\dfLI \ar[r, "\overline{Ff}"] & (\mc{C}, W_0)
	\end{tikzcd} \]
	We apply \cref{prod-join-inner-anodyne} with $(P, W) = \dfLJ$ and $Q \subseteq P$ being the collection of subsets which contain $n$.
	Thus, this square induces a map from the pushout which admits a lift to a map
	\[ \overline{[\alpha, \overline{Ff}]} \from \left( (P \times \simp{1}) \join \simp{0}, (W \times \simp{1}_1) \cup (Q \times \{0, 1\}) \join \simp{0} \right) \to (\mc{C}, W). \]
	We have an inclusion from $\dfLI$ to the domain of $\overline{[\alpha, \overline{Ff}]}$ defined, as a poset map, by sending $[n]$ to the cone point and all other subsets $A$ to $(A, 0)$.
	The composite
	\[ \dfLI \ito \left( (P \times \simp{1}) \join \simp{0}, (W \times \simp{1}_1) \cup (Q \times \{0, 1\}) \join \simp{0} \right) \xrightarrow{\overline{[\alpha, \overline{Ff}]}} (\mc{C}, W) \]
	is the desired lift.
\end{proof}
From \cref{endofunctor-clf}, we deduce \cref{clf-from-adjunction}.
\begin{proof}[Proof of \cref{clf-from-adjunction}]
	For (1), we apply \cref{endofunctor-clf} to the inclusion of markings $F(W) \subseteq L^{-1}(W)$ and the endofunctor $FL \from \mc{C} \to \mc{C}$.  
	The required homotopy is given by the unit $\eta \from \id \to RL$.
	
	Statement (2) is dual.
\end{proof}

We show that if a quasicategory satisfies CLF then its arrow category does as well.
To do this, we make use of the following lemma.
\begin{lemma} \label{lj-mapping-space-lift-lemma}
	Let $(X, W)$ be a marked simplicial set.
	\begin{enumerate}
		\item For $n \geq 2$ and $0 \leq k \leq n$, the marked simplicial set $(X, W)$ has the right lifting property against $\LJ{n+1}{k+1} \ito \LI{n+1}{k+1}$ if and only if the map 
		\[ i_0^* \from (X,W)^{\maxm{(\simp{1})}} \to (X, W) \]
		has the right lifting property against $\dfLJ \ito \dfLI$.
		\item For $n \geq 2$ and $0 \leq k \leq n$, the following are equivalent, the marked simplicial set $(X, W)$ has the right lifting property against $\RJ{n+1}{k} \ito \RI{n+1}{k}$ if and only if the map 
		\[ i_1^* \from (X,W)^{\maxm{(\simp{1})}} \to (X, W) \]
		has the right lifting property against $\dfRJ[k+1] \ito \dfRI[k+1]$.
	\end{enumerate}
\end{lemma}
\begin{proof}
	We prove (1), as (2) is dual.
	
	The data of a commutative square
	\[ \begin{tikzcd}
		\dfLJ \ar[r, "u"] \ar[d, hook] & (X, W)^{\maxm{(\simp{1})}} \ar[d, "i_0^*"] \\
		\dfLI \ar[r, "v"] & (X, W)
	\end{tikzcd} \]
	is exactly the data of a map $\overline{u} \from \dfLJ \times \maxm{(\simp{1})} \to (X, W)$ and a map $v \from \dfLI \to (X, W)$ such that $v$ is a lift of the map $\overline{u} \circ i_0 \from \dfLJ \to (X, W)$.
	Such a square admits a lift if and only if the map from the pushout 
	\[ [\overline{u}, v] \from \dfLI \push_{\dfLJ} \dfLJ \times \maxm{(\simp{1})} \to (X, W) \]
	admits a lift to a map $\dfLI \times \maxm{(\simp{1})} \to (X, W)$.
	We show the inclusion 
	\[ \dfLI \push_{\dfLJ} \dfLJ \times \maxm{(\simp{1})} \ito \dfLI \times \maxm{(\simp{1})} \] 
	is naturally isomorphic to the inclusion $\LJ{n+1}{k+1} \ito \LI{n+1}{k+1}$, from which the result follows.

	Define a map $f \from \dfLI \times \maxm{(\simp{1})} \to \LI{n+1}{k+1}$ by
	\[ f(A, i) := \begin{cases}
		\face{}{0}(A) & \text{if } i=0 \\
		\face{}{0}(A) \cup \{ 0 \} & \text{if } i=1.
	\end{cases}  \]
	This map has an inverse defined by
	\[ A \mapsto \begin{cases}
		(\degen{}{0}(A), 1) & \text{if } 0 \in A \\
		(\degen{}{0}(A), 0) & \text{if } 0 \not\in A.
	\end{cases} \]
	This gives an isomorphism $\dfLI \times \maxm{(\simp{1})} \cong \LI{n+1}{k+1}$.
	Moreover, the map $f$ restricts to a map 
	\[ \restr{f}{\dfLJ \times \maxm{(\simp{1})}} \from \dfLJ \times \maxm{(\simp{1})} \to \LJ{n+1}{k+1}. \] 
	It remains to show the square
	\[ \begin{tikzcd}
		\dfLJ \ar[r, "i_0", hook] \ar[d, hook] & \dfLJ \times \maxm{(\simp{1})} \ar[d, "\restr{f}{\dfLJ \times \maxm{(\simp{1})}}"] \\
		\dfLI \ar[r, "\face{}{0}"] & \LJ{n+1}{k+1}
	\end{tikzcd}  \]
	is a pushout.
	A non-degenerate $n$-simplex $(A_0 \subsetneq \dots \subsetneq A_n)$ in $\LJ{n+1}{k+1}$ is in the image of $\face{}{0} \from \dfLI \to \LJ{n+1}{k+1}$ if and only if $A_n$ does not contain 0.
	Moreover, if $A_n \neq [n+1] - \{ 0 \}$ then this $n$-simplex is also in the image of $f i_0$.
	This shows the intersection of the images of the bottom and right maps is the image of the composite map.
	If $A_n$ contains 0 then $(A_0 \subsetneq \dots \subsetneq A_n)$ in $\LJ{n+1}{k+1}$ is in the image under $f$ of $\dfLJ \times \maxm{(\simp{1})}$,
	which shows the union of the images of the bottom and right maps is all of $\LJ{n+1}{k+1}$, which suffices.
\end{proof}
\begin{corollary} \label{arr-cat-clf} 
	Let $(\mc{C}, W)$ be a marked quasicategory such that $W$ is strongly closed under composition.
	\begin{enumerate}
		\item If $(\mc{C}, W)$ satisfies CLF then the marked arrow category $(\mc{C}, W)^{\maxm{(\simp{1})}}$ satisfies CLF.
		\item If $(\mc{C}, W)$ satisfies CRF then the marked arrow category $(\mc{C}, W)^{\maxm{(\simp{1})}}$ satisfies CRF.
	\end{enumerate}
\end{corollary}
\begin{proof}
	Marked 1-simplices are weakly closed under composition by assumption.
	For (1), \cref{lj-mapping-space-lift-lemma} shows the marked arrow category has the right lifting property with respect to the set $\{ \dfLJ \ito \dfLI \mid n \geq 2, \ 0 < k \leq n \}$.
	Statement (2) is dual.
\end{proof}

\chapter*{Model for the localization}
\cftaddtitleline{toc}{chapter}{Model for the localization}{}

\section{The marked Ex functor} \label{sec:mEx-is-qcat}

We now describe the analogue of the category of fractions $C W^{-1}$ (or $W^{-1} C$) for quasicategories. 
Given a marked simplicial set $(X,W)$, we form a simplicial set whose 0-simplices are those of $X$ and whose 1-simplices are (co)spans of the form $x \to y' \overset{\sim}{\leftarrow} y$ (or $x \overset{\sim}{\leftarrow} x' \to y$). 
This is done using a marked variant of the $\Ex$ functor.
The main theorem of this section is \cref{mEx-qcat}, which generalizes the analogous result for classical calculus of fractions: if $(\mc{C}, W)$ is a marked quasicategory where $W$ is weakly closed under composition then calculus of fractions is a sufficient \emph{and necessary} condition for the ``category of fractions'' to be a quasicategory.

For $n \geq 0$, let $\msd [n]$ be the poset $\mathcal{P}_{\neq \varnothing}[n]$ of non-empty subsets of $[n]$ (regarded as a simplicial set).
We view a 1-simplex $A_0 \subseteq A_1$ as marked if $\max A_0 = \max A_1$.
This gives the data of a cosimplicial object $\Delta \to \msSet$, with simplicial operators defined using the image.
\begin{example}
	We give example depictions of $\msd [0]$, $\msd [1]$, and $\msd [2]$, respectively.
	
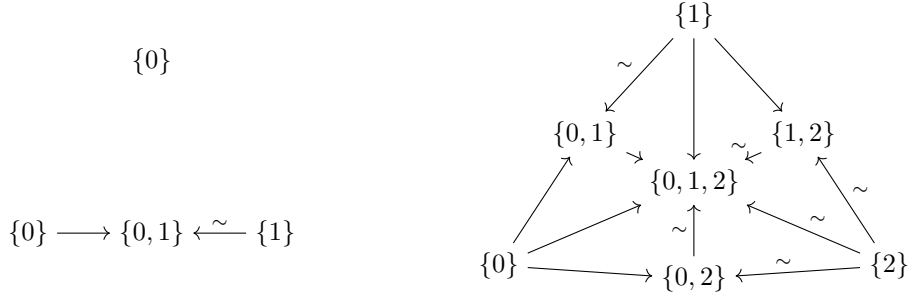
\begin{figure}[H]
		\centering
		\begin{minipage}{0.38\textwidth}
			\begin{tikzpicture}[commutative diagrams/every diagram, node distance=2em]
				\node (0) {$\{ 0 \}$};
				\node (01) [right=of 0] {$\{ 0, 1 \}$};
				\node (1) [right=of 01] {$\{ 1 \}$};

				\node (0b) [above=4.8em of 01] {$\{ 0 \}$};

				\path[commutative diagrams/.cd, every arrow, every label]
					(0) edge (01)
					(1) edge node[swap] {$\sim$} (01);
			\end{tikzpicture}
		\end{minipage} \begin{minipage}{0.38\textwidth}
			\begin{tikzpicture}[commutative diagrams/every diagram]
				\node (0) {$\{ 0 \}$};
				\path (0) ++(52:12em) node (1) {$\{ 1 \}$};
				\path (1) ++(-52:12em) node (2) {$\{ 2 \}$};
				\node (01) at (barycentric cs:0=1,1=1) {};
				\node (02) at (barycentric cs:0=1,2=1) {};
				\node (12) at (barycentric cs:1=1,2=1) {};
				\path (01) ++(150:0.5em) node (01b) {$\{ 0, 1 \}$};
				\path (12) ++(30:0.5em) node (12b) {$\{ 1, 2 \}$};
				\path (02) ++(270:0.5em) node (02b) {$\{ 0, 2 \}$};
				\node (012) at (barycentric cs:0=1,1=1,2=1) {$\{ 0, 1, 2 \}$};

				\path[commutative diagrams/.cd, every arrow, every label]
					(0) edge (01b)
						edge (02b)
						edge (012)
					(1) edge node[swap] {$\sim$} (01b)
						edge (12b)
						edge (012)
					(2) edge node[swap] {$\sim$} (02b)
						edge node[swap] {$\sim$} (12b)
						edge node[swap] {$\sim$} (012)
					(01b) edge (012)
					(02b) edge node {$\sim$} (012)
					(12b) edge node[swap] {$\sim$} (012);
			\end{tikzpicture}
		\end{minipage}
		\caption*{The marked simplicial sets $\msd [0]$, $\msd [1]$, and $\msd [2]$.}
	\end{figure}
\end{example}
\begin{definition}
	The \emph{marked subdivision} functor $\mSd \from \sSet \to \msSet$ is the extension by colimits of the cosimplicial object $\msd \from \Delta \to \msSet$.
	\[ \begin{tikzcd}
		\Delta \ar[d, hook] \ar[r, "\msd"] & \msSet \\
		\sSet \ar[ur, dotted, "\mSd"']
	\end{tikzcd} \]
	The \emph{marked $\mathrm{Ex}$-functor} is its right adjoint $\mEx \from \msSet \to \sSet$, defined by
	\[ \mEx(X, W)_n := \msSet(\msd [n], (X, W)). \]
\end{definition}

\begin{remark} \label{Ex-v-mEx}
  The composite
  \[ \sSet \xrightarrow{\maxm{(-)}} \msSet \xrightarrow{\mEx} \sSet \]
  recovers Kan's original $\Ex$-functor \cite{kan:css}.
\end{remark}

We write $\msdop [n]$ for the opposite poset of $\msd [n]$.
We view a 1-simplex $A_0 \supseteq A_1$ of $\msdop [n]$ as marked if $\min A_0 = \min A_1$.
\begin{example}
	Below are example depictions of $\msdop [0]$, $\msdop [1]$, and $\msdop [2]$, respectively.
	\begin{figure}[H]
		\centering
		\begin{minipage}{0.38\textwidth}
			\begin{tikzpicture}[commutative diagrams/every diagram, node distance=2em]
				\node (0) {$\{ 0 \}$};
				\node (01) [right=of 0] {$\{ 0, 1 \}$};
				\node (1) [right=of 01] {$\{ 1 \}$};

				\node (0b) [above=4.8em of 01] {$\{ 0 \}$};

				\path[commutative diagrams/.cd, every arrow, every label]
					(01) edge node[swap] {$\sim$} (0)
						 edge 				(1);
			\end{tikzpicture}
		\end{minipage} \begin{minipage}{0.38\textwidth}
			\begin{tikzpicture}[commutative diagrams/every diagram]
				\node (0) {$\{ 0 \}$};
				\path (0) ++(52:12em) node (1) {$\{ 1 \}$};
				\path (1) ++(-52:12em) node (2) {$\{ 2 \}$};
				\node (01) at (barycentric cs:0=1,1=1) {};
				\node (02) at (barycentric cs:0=1,2=1) {};
				\node (12) at (barycentric cs:1=1,2=1) {};
				\path (01) ++(150:0.5em) node (01b) {$\{ 0, 1 \}$};
				\path (12) ++(30:0.5em) node (12b) {$\{ 1, 2 \}$};
				\path (02) ++(270:0.5em) node (02b) {$\{ 0, 2 \}$};
				\node (012) at (barycentric cs:0=1,1=1,2=1) {$\{ 0, 1, 2 \}$};

				\path[commutative diagrams/.cd, every arrow, every label]
					(012) edge node[swap] {$\sim$} (01b)
						  edge node {$\sim$} (02b)
						  edge 				 (12b)
						  edge node[swap] {$\sim$} (0)
						  edge 				 (1)
						  edge 				 (2)
					(01b) edge node[swap] {$\sim$} (0)
						 edge 				(1)
					(02b) edge node[swap] {$\sim$} (0)
						 edge				(2)
					(12b) edge node[swap] {$\sim$} (1)
						 edge				(2);
			\end{tikzpicture}
		\end{minipage}
		\caption*{The marked simplicial sets $\msdop [0]$, $\msdop [1]$, and $\msdop [2]$.}
	\end{figure}
\end{example}
The respective extension by colimits is denoted $\mSdop \from \sSet \to \msSet$, and its right adjoint $\mExop \from \msSet \to \sSet$ is defined by
\[ \mExop(X, W)_n := \msSet(\msdop [n], (X, W)). \]
By construction, we have a natural isomorphism $\mSdop \simp{n} \cong (\mSd (\simp{n})^\op)^\op$ of cosimplicial objects.
This induces isomorphisms
\[ \mSdop X \cong (\mSd X^\op)^\op , \quad \mExop(X, W) \cong \left( \mEx(X^\op, W) \right)^\op \]
natural in $X \in \sSet$ and $(X, W) \in \msSet$, respectively.

The goal of this section is to prove the following result.
\begin{theorem} \label{mEx-qcat}
	Let $(X, W)$ be a marked simplicial set.
	\begin{enumerate}
		\item \begin{enumerate}
			\item If $\mEx (X, W)$ is a quasicategory then $(X,W)$ has the right lifting property with respect to the set $\{ \dfLJ \ito \dfLI \mid n \geq 2, \ 0 < k \leq n \}$.
			\item If $X$ is a quasicategory and $(X, W)$ satisfies CLF then $\mEx (X, W)$ is a quasicategory.
		\end{enumerate}
		\item \begin{enumerate}
			\item If $\mExop (X, W)$ is a quasicategory then $(X,W)$ has the right lifting property with respect to the set $\{ \dfRJ \ito \dfRI \mid n \geq 2, \ 0 \leq k < n \}$.
			\item If $X$ is a quasicategory and $(X, W)$ satisfies CRF then $\mExop (X, W)$ is a quasicategory.
		\end{enumerate}
	\end{enumerate}
\end{theorem}
To prove this, we establish some auxiliary definitions and results.
As \cref{mEx-qcat} is our primary focus in this section, we work with calculus of left fractions and $\mSd \from \sSet \to \msSet$, omitting the dual statements.

We identify $\mSd \Lambda^n_k$ with the maximal simplicial subset of $\mSd \simp{n}$ which omits the 0-simplices $[n]$ and $[n] - \{ k \}$.
Note that there are inclusions $\dfLI \subseteq \mSd \simp{n}$ and $\dfLJ \subseteq \mSd \Lambda^n_i$.
\begin{proposition} \label{i-j-retract}
	For $n \geq 1$ and $0 \leq k \leq n$, the inclusion $\dfLJ \ito \dfLI$ is a retract of the inclusion $\mSd \horn{n}{k} \ito \mSd \simp{n}$. 
\end{proposition}
\begin{proof}
	The inclusion $\dfLI \subseteq \mSd \simp{n}$ admits a retraction defined, as a poset map, by $A \mapsto A \cup \{ k \}$.
	The restriction of this retraction to $\mSd \horn{n}{k}$ factors through $\dfLJ$, as desired.
	\[ \begin{tikzcd}
		\dfLJ \ar[r, hook] \ar[d, hook] & \mSd \horn{n}{k} \ar[d, hook] \ar[r, "\cup \{ k \}"] & \dfLJ \ar[d, hook] \\
		\dfLI \ar[r, hook] & \mSd \simp{n} \ar[r, "\cup \{ k \}"] & \dfLI
	\end{tikzcd} \]
\end{proof}

\begin{definition}
	For $n \geq 0$ and $0 < k < n$, 
	\begin{enumerate}
		\item let $J^n_k \subseteq \mSd \simp{n}$ be the marked simplicial subset which contains $\mSd \Lambda^n_i$ and, for $k \geq 0$, all $m$-simplices $(A_0 \subseteq \dots \subseteq A_m)$ such that $k \in A_0$; and
		\item let $K^n_k \subseteq \mSd \simp{n}$ be the maximal marked simplicial subset which omits the 0-simplex $[n] - \{ k \}$. 
	\end{enumerate}
\end{definition}
By definition, there is a sequence of inclusions
\[ \mSd \Lambda^n_k \subseteq J^n_k \subseteq K^n_k \subseteq \mSd \simp{n}. \]
Note there is an inclusion $\dfLI \subseteq J^n_k$ which restricts to an inclusion $\dfLJ \subseteq \mSd \Lambda^n_i$.
\begin{proposition}
	For $n \geq 2$ and $0 \leq k \leq n$, the square
	\[ \begin{tikzcd}
		\dfLJ \ar[r, hook] \ar[d, hook] & \mSd \horn{n}{k} \ar[d, hook] \\
		\dfLI \ar[r, hook] & J^n_k
	\end{tikzcd} \]
	is a pushout.
\end{proposition}
\begin{proof}
	Follows since the simplices in $J^n_k$ which are not contained in $\mSd \horn{n}{i}$ are exactly the chains $A_0 \subseteq \dots \subseteq A_m$ where $A_m = [n]$ and all subsets contain $k$ (i.e.~$k \in A_0$).
\end{proof}
\begin{corollary} \label{mEx-qcat-CRF-lift}
	Let $(X, W)$ be a marked simplicial set.
	If $(X, W)$ has the right lifting property against $\{ \dfLJ \ito \dfLI \mid n \geq 2, \ 0 < k \leq n \}$
	then it has the right lifting property against
	\[ \{ \mSd \horn{n}{k} \ito J^n_k \mid n \geq 2, \ 0 < k \leq n \}. \] 
	\qed
\end{corollary}
\begin{proposition} \label{mEx-qcat-retraction}
	For $n \geq 2$ and $0 \leq k < n$, the inclusion $K^n_k \subseteq \mSd \simp{n}$ admits a retraction.
\end{proposition}
\begin{proof}
	The retraction is defined by
	\[ S \mapsto \begin{cases}
		[n] & \text{if $S = [n] - \{ i \}$} \\
		S & \text{otherwise}.
	\end{cases} \]
	This map preserves marked 1-simplices as $i \neq n$ by assumption.
\end{proof}

\begin{lemma} \label{knk-lnk-inner-anodyne}
	The inclusion $J^n_k \subseteq K^n_k$ is in the saturation of the set
	\[ \{ \minm{(\horn{n}{k})} \ito \minm{(\simp{n})} \mid n \geq 2, \ 0 < k < n \} \cup \{ \maxm{(\horn{2}{1})} \ito \maxm{(\simp{2})} \}. \]
\end{lemma}
\begin{proof}
	We construct a simple inner horn decomposition and apply \cref{anodyne-decomp-is-anodyne}.

	The simplices of $K^n_k$ which are not contained in $J^n_k$ are exactly chains $(A_0 \subseteq \dots \subseteq A_m)$ such that $A_m = [n]$ and $k \not\in A_0$.
	For $m \geq 2$ and $j = 1, \dots, m-1$, define subsets $S^m_j \subseteq (K^n_k)_m$ of non-degenerate $m$-simplices by
	\begin{align*}
		S^m_j &:= \left\{ (A_0 \subsetneq \dots \subsetneq A_m) \in (K^n_k)_m \mid \begin{array}{l}
			A_m = [n] \\
			k \not\in A_{j-1} \\
			A_{j} = A_{j-1} \cup \{ k \}
		\end{array} \right\}.
	\end{align*}
	For $m \geq 1$ and $j = 1, \dots, m$, define subsets $T^m_j \subseteq (K^n_k)_m$ of non-degenerate $m$-simplices by
	\begin{align*}
		T^m_j &= \left\{ (A_0 \subsetneq \dots \subsetneq A_m) \in (K^n_k)_m \mid \begin{array}{l}
			A_m = [n] \\
			k \not\in A_{j-1} \\
			A_{j} \supsetneq A_{j-1} \cup \{ k \}
		\end{array} \right \}.
	\end{align*}
	We show that 
	\[ (\{ S^m_j \}_{j=1}^{m-1}, \{ T^m_j \}_{j=1}^m, \id[\{ 1, \dots, m-1 \}]) \] 
	is a simple inner horn decomposition for the inclusion $J^n_k \subseteq K^n_k$.
	
	Fix $m \geq 2$ and $j \in \{ 1, \dots, m-1 \}$.
	\begin{enumerate}
		\item Every simplex $(A_0 \subsetneq \dots \subsetneq A_m)$ in $S^m_j$ is non-degenerate, thus $A_{j+1}$ contains an element which is not contained in $A_j = A_{j-1} \cup \{ k \}$.
		This shows the image of $\partial_j \vert_{S^m_j}$ is contained in $T^{m-1}_{j}$.
		The inverse function $T^{m-1}_{j} \to S^m_j$ is defined by
		\[ (A_0 \subsetneq \dots \subsetneq A_{m-1}) \mapsto (A_0 \subsetneq \dots \subsetneq A_{j-1} \subsetneq A_{j-1} \cup \{ i \} \subsetneq A_j \subsetneq \dots \subsetneq A_{m-1} ). \]
		\item If $(A_0 \subsetneq A_1 \subsetneq A_2)$ is such that $\max A_0 = \max A_2$, then $\max A_0 \leq \max A_1$ and $\max A_1 \leq \max A_2$ implies $(A_0 \subsetneq A_1 \subsetneq A_2)$ is maximally-marked.
		\item For $u \in S^m_j$ and $l \neq j$, we proceed by case analysis on $u\face{}{l}$:
		\begin{itemize}
			\item if $l < j$ then $u\face{}{l} \in S^{m-1}_{j-1}$;
			\item if $l > j$ then $f\partial_l \in S^{m-1}_j$;
			\item if $l = m$ then $u\face{}{l} \in \mSd \Lambda^n_i$ since, writing $u = (A_0 \subseteq \dots \subseteq A_m)$, we have that $A_{m-1} \neq [n]$ as $u$ is non-degenerate. \qedhere
		\end{itemize}
	\end{enumerate}
\end{proof}

With this, we may prove \cref{mEx-qcat}.
\begin{proof}[Proof of \cref{mEx-qcat}]
	Statement (1) follows from \cref{i-j-retract}.

	For (2), fix a map $\Lambda^n_k \to \mEx(X, W)$ for $n \geq 2$ and $0 < i < n$.
	By adjointness, it suffices to show the tranpose $\mSd \Lambda^n_k \to (X, W)$ admits a lift to $\mSd \simp{n} \to (X, W)$.
	A lift from $\mSd \Lambda^n_i$ to $J^n_k$ exists by \cref{mEx-qcat-CRF-lift}.
	A lift from $J^n_k$ to $K^n_k$ exists by \cref{knk-lnk-inner-anodyne}.
	As $\mSd \simp{n}$ retracts onto $K^n_k$ by \cref{mEx-qcat-retraction}, this gives the desired lift.
	\[ \begin{tikzcd}
		\mSd \Lambda^n_i \ar[r] \ar[d, hook] & (X, W) \\
		J^n_k \ar[ur, dotted, "\text{by \ref{mEx-qcat-CRF-lift}}" very near start, swap] \ar[d, hook] \\
		K^n_k \ar[d, hook] \ar[uur, dotted, "\text{by \ref{knk-lnk-inner-anodyne}}", swap, bend right] \\
		\mSd \simp{n} \ar[u, bend right, dotted, "\text{by \ref{mEx-qcat-retraction}}", swap]
	\end{tikzcd} \]
\end{proof}

\section{Localization at equivalences via marked Ex} \label{sec:localization-at-equivs}

In this section, we move towards showing that if $(\mc{C}, W)$ satisfies CLF then $\mEx(\mc{C},W)$ is a model for the localization of $\mc{C}$ at $W$.
We construct the canonical map $\mc{C} \to \mEx(\mc{C}, W)$ (or $\mc{C} \to \mExop(\mc{C, W})$) and show that it sends weak equivalences to equivalences (\cref{min-ascends}).
We then consider the case where $W$ is exactly the set of equivalences in $\mc{C}$.
In this case, one has that the localization is a categorical equivalence (\cref{ex-msup-weq}), which we prove by adapting the argument given in \cite[Thm.~4.1]{latch-thomason-wilson} to show that $X \to \Ex X$ is a weak homotopy equivalence.

The map of posets $\max \from \msd[n] \to [n]$ is natural in $n$; that is, it induces a natural transformation between functors $\Delta \to \msSet$.
For a simplicial set $X$, this induces a map 
\[ \msub \from \mSd X \to \minm{X}\] 
natural in $X$ via extension by colimits.
For a marked simplicial set $(X, W)$, pre-composition by $\max$ induces a map 
\[ \msup \from X \to \mEx(X, W) \] 
which sends an $n$-simplex $f \from \simp{n} \to X$ of $X$ to the $n$-simplex $f \circ \max \from \mSd \simp{n} \to (X,W)$ of $\mEx(X, W)$.
Dually, the map $\min \from \msdop [n] \to [n]$ induces a map 
\[ \minsub \from \mSdop X \to \minm{X} \] 
and a map 
\[ \minsup \from X \to \mExop(X, W). \]

The map $\msup \from \mc{C} \to \mEx(\mc{C}, W)$ (and its dual) will be our candidate for the localization map.
Note that for a weak equivalence $w \in W$, the diagram
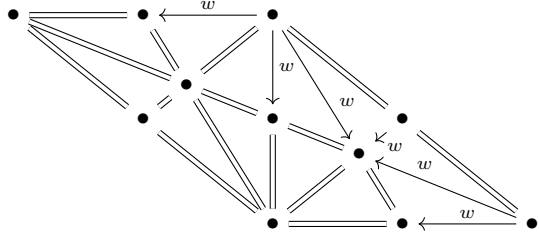
\begin{center}
	\begin{tikzpicture}[commutative diagrams/every diagram,node distance=3.7em]
		\node (MM) {$\bullet$};
		\node (MT) [above=2.8em of MM] {$\bullet$};
		\node (MB) [below=2.8em of MM] {$\bullet$};
		\node (LT) [left=of MT] {$\bullet$};
		\node (LLT) [left=of LT] {$\bullet$};
		\node (RB) [right=of MB] {$\bullet$};
		\node (RRB) [right=of RB] {$\bullet$};
		\node (LM) at (barycentric cs:LLT=1,MB=1) {$\bullet$};
		\node (RM) at (barycentric cs:MT=1,RRB=1) {$\bullet$};
		\node (CL) at (barycentric cs:LLT=1,MT=1,MB=1) {$\bullet$};
		\node (CR) at (barycentric cs:MT=1,MB=1,RRB=1)
		{$\bullet$};

		\path[commutative diagrams/.cd, every arrow, every label]
			(LLT) edge[commutative diagrams/equal] (LT)
				  edge[commutative diagrams/equal] (LM)
				  edge[commutative diagrams/equal] (CL)
			(MT) edge node[swap] {$w$} (LT)
				 edge node {$w$} (MM)
				 edge[commutative diagrams/equal] (RM)
				 edge[commutative diagrams/equal] (CL)
				 edge node[near end] {$w$} (CR)
			(MB) edge[commutative diagrams/equal] (MM)
				 edge[commutative diagrams/equal] (RB)
				 edge[commutative diagrams/equal] (LM)
				 edge[commutative diagrams/equal] (CL)
				 edge[commutative diagrams/equal] (CR)
			(RRB) edge node[swap] {$w$} (RB)
				  edge[commutative diagrams/equal] (RM)
				  edge node[swap, near end] {$w$} (CR)
			(LT) edge[commutative diagrams/equal] (CL)
			(LM) edge[commutative diagrams/equal] (CL)
			(MM) edge[commutative diagrams/equal] (CL)
				 edge[commutative diagrams/equal] (CR)
			(RM) edge node {$w$} (CR)
			(RB) edge[commutative diagrams/equal] (CR);
	\end{tikzpicture}
\end{center}
defines a map $J \to \mEx(X, W)$ which witnesses that $(\msup)_1(w) \in {\mEx(X, W)}_1$ is invertible.
Dualizing this diagram gives that
\begin{proposition} \label{min-ascends}
	For a marked simplicial set $(X, W)$, the maps 
	\[ \msup \from X \to \mEx(X, W), \quad \minsup \from X \to \mExop(X, W) \] 
	send weak equivalences to $J$-equivalences. \qed
\end{proposition}

Via \cref{min-ascends}, the maps 
\[ \msup \from X \to \mEx(X, W), \quad \minsup \from X \to \mExop(X, W) \] 
ascend to marked maps
\[ \msupb \from (X, W) \to \natm{\mEx(X, W)}, \quad \minsupb \from (X, W) \to \natm{\mExop(X, W)} \] 
A further corollary of \cref{min-ascends} is that the functors $\mEx, \mExop \from \msSet \to \sSet$ send marked homotopy equivalances to $J$-homotopy equivalences.
\begin{proposition} \label{mEx-preserves-htpy-equivs}
	The functors $\mEx, \mExop \from \msSet \to \sSet$ send marked homotopy equivalences to $J$-homotopy equivalences.
\end{proposition}
\begin{proof}
	For $\mEx$, the map
	\[ J \to \mEx \maxm{(\simp{1})} \]
	which witnesses the invertibility of the 1-simplex $(\msup)_1(\id[\maxm{(\simp{1})}]) \in \mEx(\maxm{(\simp{1})})_1$ preserves endpoints.
	Thus, a marked homotopy 
	\[ \alpha \from (X, W) \times \maxm{(\simp{1})} \to (X', W')\] 
	induces a $J$-homotopy
	\[ \mEx(X, W) \times J \to \mEx(X, W) \times \mEx \maxm{(\simp{1})} \to \mEx(X', W'). \]
	The argument for $\mExop$ is analogous.
\end{proof}

In order to show $\mEx(\mc{C}, W)$ computes the localization, we first consider the case where $W$ is exactly the equivalences in $\mc{C}$.
In this case, we wish to show $\mEx \natm{\mc{C}}$ is weakly equivalent to $\mc{C}$.
To show this, we prove the following lemma.
\begin{lemma} \label{max-weq}
	For $n \geq 0$, the map 
	\[ \max \from \mSd \simp{n} \to \minm{(\simp{n})} \] 
	is a marked homotopy equivalence.
\end{lemma}
\begin{proof}
	The map of posets $\max \from \msd[n] \to [n]$ has a section $f \from [n] \to \msd [n]$ defined by
	\[ f(i) := \{ 0, \dots, i \}. \]
	For $A \in \msd [n]$, the inclusion $A \subseteq \{ 0, \dots, \max A \}$ induces a marked homotopy $\mSd \simp{n} \times \maxm{(\simp{1})} \to \mSd \simp{n}$ from the composite $f \circ \max$ to the identity map on $\mSd \simp{n}$.
\end{proof}

Our proof that $\msup \from \mc{C} \to \mEx \natm{\mc{C}}$ is a categorical equivalence uses an analogous argument to \cite[Thm.~4.1]{latch-thomason-wilson}, which shows $X \to \operatorname{Ex} X$ is a Kan--Quillen weak equivalence.
In particular, we make use of the generalized diagonal lemma of \cite[Thm.~2.5]{carranza-kapulkin-wong:diagonal} and specifically its instantiation to the Joyal model structure \cite[Ex.~3.6]{carranza-kapulkin-wong:diagonal}.

\begin{theorem} \label{ex-nat-weq}
	For a quasicategory $\mc{C}$, the maps
	\[ \msup \from \mc{C} \to \mEx \natm{\mc{C}}, \quad \minsup \from \mc{C} \to \mExop \natm{\mc{C}} \]
	are equivalences of quasicategories.
\end{theorem}
\begin{proof}
	We show the map $\msup$ is an equivalence.
	The argument for $\minsup$ is dual.

	We regard the square
	\[ \begin{tikzcd}
		\msSet\left( \minm{(\simp{m})} \times \simp{0}, \natm{\mc{C}} \right) \ar[r] \ar[d] & \msSet \left( \mSd \simp{m} \times \simp{0}, \natm{\mc{C}} \right) \ar[d] \\
		\msSet\left( \minm{(\simp{m})} \times \maxm{(\simp{n})}, \natm{\mc{C}} \right) \ar[r] & \msSet \left( \mSd \simp{m} \times \maxm{(\simp{n})}, \natm{\mc{C}} \right)
	\end{tikzcd} \]
	as a diagram of bisimplicial sets by varying $m$ and $n$.
	Fixing $n$, the left map in the square
	\[ \begin{tikzcd}
		\mc{C} \ar[r] \ar[d] & \mEx \natm{C} \ar[d] \\
		U ((\natm{\mc{C}})^{\maxm{(\simp{n})}}) \ar[r] & \mEx((\natm{\mc{C}})^{\maxm{(\simp{n})}})
	\end{tikzcd} \]
	is obtained by applying the forgetful functor $U \from \msSet \to \sSet$ to the map
	\[ !^* \from (\natm{\mc{C}})^{\simp{0}} \to (\natm{\mc{C}})^{\maxm{(\simp{n})}}. \]
	This map is a marked homotopy equivalence as $(\natm{\mc{C}})^{\relbar} \from (\msSet)^\op \to \msSet$ preserves marked homotopies.
	The codomain is a quasicategory since $\mc{C}$ is.
	A marked 1-simplex of the codomain $H \from \maxm{(\simp{n})} \times \maxm{(\simp{1})} \to \natm{\mc{C}}$ is a marked homotopy.
	As $\mc{C}$ is a quasicategory marked at equivalences, every such $H$ ascends to an $E^1$-homotopy, thus the codomain is marked at exactly equivalences.
	This gives that $!^*$ is a weak equivalance between fibrant-cofibrant objects in $\msSet$, hence the left map is an equivalence of quasicategories.
	The right map is obtained by applying $\mEx \from \msSet \to \sSet$ to $!^*$.
	By \cref{mEx-preserves-htpy-equivs}, the right map is a weak equivalence.

	Fixing $m$, the bottom map
	\[ \begin{tikzcd}
		\cdot \ar[r] \ar[d] & \cdot \ar[d] \\
		\mcore \left( (\natm{\mc{C}})^{\minm{(\simp{m})}} \right) \ar[r] & \mcore \left( (\natm{\mc{C}})^{\mSd \simp{m}} \right)
	\end{tikzcd} \]
	is obtained by applying $\mcore \from \msSet \to \sSet$ to the map
	\[ (\natm{\mc{C}})^{\max } \from (\natm{\mc{C}})^{\minm{(\simp{m})}} \to (\natm{\mc{C}})^{\mSd \simp{m}}. \]
	This map is a marked homotopy equivalence by \cref{max-weq}.
	It is, moreover, a map between quasicategories marked at equivalences, hence this map is an equivalence of quasicategories.
	Thus, it is an equivalence of categories after applying $\mcore$.

	Applying the diagonal functor to this square,
	\[ \begin{tikzcd}
		\mc{C} \ar[r] \ar[d, "\sim"'] & \mEx \natm{C} \ar[d, "\sim"] \\
		\cdot \ar[r, "\sim"] & \cdot
	\end{tikzcd} \]
	the left, right, and bottom maps are weak equivalences by \cite[Ex.~3.6]{carranza-kapulkin-wong:diagonal}.
	Thus, the top map is a weak equivalence by 2-out-of-3.
\end{proof}

\section{Localization via marked Ex} \label{sec:localization-general}

In this section, we show that $\mEx(\mc{C}, W)$ and $\mExop(\mc{C}, W)$ are models for the localization for an arbitrary $W$, using the ``base case'' where $W$ is the set of equivalences.
Our main theorem is \cref{mEx-computes-localization}, which is the quasicategorical analogue of \cref{classical-fractions-compute-localization}.
In particular, this result requires multiple auxiliary lemmas, most important of which is \cref{min-ex-min-J-htpy}.
As these statements are more technical in nature and are only used to prove the main theorem of this section (\cref{mEx-computes-localization}), we state \cref{msup-msub-transpose,min-ex-min-diagram,min-ex-min-htpy,ex-msup-weq} only for $\mEx$, omitting the dual versions.

\begin{lemma} \label{msup-msub-transpose}
	For a marked simplicial set $(X, W)$, the composite
	\[ \mSd X \xrightarrow{\msub} \minm{X} \hookrightarrow (X, W) \]
	is adjoint transpose to the map
	\[ X \xrightarrow{\msup} \mEx(X, W) \]
	under the $(\mSd \dashv \mEx)$-adjunction.
\end{lemma}
\begin{proof}
	We wish to show the diagram
	\[ \begin{tikzcd}
		\mSd X \ar[r, "\msub"] \ar[d, "\mSd \msup"'] & \minm{X} \ar[d, hook] \\
		\mSd \mEx(X, W) \ar[r, "\varepsilon_{(X, W)}"] & (X, W)
	\end{tikzcd} \]
	commutes.
	It suffices to show it commutes after applying $U$.
	That is, it suffices to show the triangle
	\[ \begin{tikzcd}
		U \mSd X \ar[rd, "\mSd \msup"'] \ar[rr, "\msub"] & {} & X \\
		{} & U \mSd \mEx (X, W) \ar[ur, "\varepsilon_{(X, W)}"']
	\end{tikzcd} \]
	commutes.

	Recall the $n$-simplices of $\mSd X$ are pairs $(u, v)$ where
	\begin{itemize}
		\item $u \from \simp{m} \to X$ is a non-degenerate $m$-simplex of $X$ for some $m \geq 0$; and
		\item $v \from \simp{n} \to \mSd \simp{m}$ is an $n$-simplex of $\mSd \simp{m}$
	\end{itemize}
	subject to the identification $(u \vec{\partial}, v) = (u, \mSd (\vec{\partial}) \circ v)$ for any injective map $\vec{\partial} \from [m] \to [m']$ in the simplex category.
	The top map is given by the formula
	\[ (\msub)_n(u, v) = u \circ \max \circ v. \]
	We calculate the bottom composite as
	\begin{align*}
		(\varepsilon_{(X, W)})_m ((\mSd \msup)_m (u, v)) &= (\varepsilon_{(X, W)})_m (u \circ \max, v) \\
		&= (u \circ \max) \circ v. \qedhere
	\end{align*}
\end{proof}
\begin{lemma} \label{min-ex-min-diagram}
	For a marked simplicial set $(X, W)$, the squares
	\[ \begin{tikzcd}[column sep = 5em, row sep = 4em]
		\sSet(\simp{n}, \mEx(X, W)) \ar[r, "(\msup)_*"] \ar[d, "\cong"'] & \sSet(\simp{n}, \mEx ( \maxm{\mEx(X, W)} ) ) \ar[d, "\cong"] \\
		\msSet(\mSd \simp{n}, (X, W)) \ar[r, "\mSd (U \msup)"] & \msSet(\mSd (U \mSd \simp{n}), (X, W))
	\end{tikzcd} \]
	and
	\[ \begin{tikzcd}[cramped, sep = 4em]
		\sSet(\simp{n}, \mEx(X, W)) \ar[r, "\mEx(\eta_{(X,W)})_*"] \ar[d, "\cong"'] & \sSet(\simp{n}, \mEx \maxm{X}) \ar[r, "\mEx(\maxm{(\msup)})_*"] & \sSet(\simp{n}, \mEx (\maxm{\mEx(X, W)}) ) \ar[d, "\cong"] \\
		\msSet(\mSd \simp{n}, (X, W)) \ar[r, "(\varepsilon_{\mSd \simp{n}})^*"] & \msSet(\minm{(U \mSd \simp{n})}, (X, W)) \ar[r, "(\msub)^*"] & \msSet(\mSd(U \mSd \simp{n}), (X, W))
	\end{tikzcd} \] 
	commute for all $n \geq 0$.
\end{lemma}
\begin{proof}
	For the first square, the top square in
	\[ \begin{tikzcd}[column sep = large]
		\sSet(\simp{n}, \mEx(X, W)) \ar[r, "(\msup)_*"] \ar[d, "\cong"'] & \sSet(\simp{n}, \mEx ( \maxm{\mEx(X, W)} ) ) \ar[d, "\cong"] \\
		\msSet(\minm{(\simp{n})}, \maxm{\mEx(X, W)}) \ar[r, "\msup"] \ar[d, "\cong"'] & \msSet(\mSd \simp{n}, \maxm{\mEx(X, W)}) \ar[d, "\cong"] \\
		\msSet(\mSd \simp{n}, (X, W)) \ar[r, "\mSd (U \msup)"] & \msSet(\mSd (U \mSd \simp{n}), (X, W))
	\end{tikzcd} \]
	commutes by definition and the bottom square commutes by naturality of adjunction bijections.
	The desired square is identical to the outer square in this diagram.
	
	For the second square, we show that each sub-diagram of
	\[ \begin{tikzcd}[cramped, column sep = 4em]
		\sSet(\simp{n}, \mEx(X, W)) \ar[r, "\mEx(\eta_{(X,W)})_*"] \ar[d, "\cong"'] & \sSet(\simp{n}, \mEx \maxm{X}) \ar[r, "\mEx(\maxm{(\msup)})_*"] \ar[d, "\cong"'] & \sSet(\simp{n}, \mEx (\maxm{\mEx(X, W)}) ) \ar[d, "\cong"] \\
		\msSet(\mSd \simp{n}, (X, W)) \ar[r, "(\eta_{(X, W)})_*"] \ar[ddr, "(\varepsilon_{\mSd \simp{n}})^*"'] & \msSet(\mSd \simp{n}, \maxm{X}) \ar[r, "(\maxm{(\msup)})_*"] \ar[d, "\cong"', "\varphi"] & \msSet(\mSd \simp{n}, \maxm{\mEx(X, W)}) \ar[d, "\cong"] \\
		{} & \sSet(U(\mSd \simp{n}), X) \ar[d, "\cong"', "\psi"] \ar[r, "{(\msup)_*}"] & \sSet(U(\mSd \simp{n}), \mEx(X, W)) \ar[d, "\cong", "\theta"'] \\
		{} & \msSet(\minm{(U \mSd \simp{n})}, (X, W)) \ar[r, "(\msub)^*"] & \msSet(\mSd(U \mSd \simp{n}), (X, W))
	\end{tikzcd} \]
	commutes.
	The desired equality is exactly the outer composite.

	The top left and right squares commute by naturality of adjunction bijections.
	The middle right square commutes by definition.
	We show the bottom left triangle and bottom right square commute by explicit computation.
	
	For the bottom left triangle, fix $f \from \mSd \simp{n} \to (X, W)$. 
	We apply the forgetful functor $U \from \msSet \to \sSet$ and calculate
	\[ \begin{array}{r@{ \ }l l}
		U (\psi \varphi(\eta_{(X, W)} \circ f)) &= U (\psi (\varphi(\eta_{(X, W)}) \circ Uf)) \\
		&= U (\psi (Uf)) \\
		&= Uf & (\ast) \\
		&= Uf \circ \id \\
		&= Uf \circ U \varepsilon_{\mSd \simp{n}} \\
		&= U ( f \circ \varepsilon_{\mSd \simp{n}}),
	\end{array} \]
	where $(\ast)$ holds since applying $U$ to any map $g$ obtained by the adjunction bijection $\psi$ is exactly $g$.

	For the bottom right square, fix $f \from U \mSd \simp{n} \to X$.
	We calculate
	\[ \begin{array}{r@{ \ }l l}
		\theta (\msup \circ f) &= \theta(\msup) \circ \mSd f \\
		&= \varepsilon_{X} \circ \msub \circ \mSd f & \text{by \cref{msup-msub-transpose}} \\
		&= \varepsilon_X \circ f \circ \msub & \text{by naturality} \\
		&= \theta(f) \circ \msub ,
	\end{array} \]
	where $\varepsilon_X$ denotes the counit $\minm{X} \to (X, W)$ of the minimal marking adjunction.
\end{proof}

\begin{lemma} \label{min-ex-min-J-htpy}
	For a marked simplicial set $(X, W)$, 
	\begin{enumerate}
		\item the maps
		\[ \begin{tikzcd}[column sep = 3em]
			\mEx (X, W) \ar[r, yshift=-0.75ex, "\msup", swap] \ar[r, yshift=0.75ex, "\mEx \left( \msupb \right)"] & \mEx \left( \natm{\mEx (X, W)} \right)
		\end{tikzcd} \]
		are $J$-homotopic; and
		\item the maps
		\[ \begin{tikzcd}[column sep = 3em]
			\mExop (X, W) \ar[r, yshift=-0.75ex, "\minsup", swap] \ar[r, yshift=0.75ex, "\mEx \left( \minsupb \right)"] & \mExop \left( \natm{\mExop (X, W)} \right)
		\end{tikzcd} \]
		are $J$-homotopic.
	\end{enumerate}
\end{lemma}
\begin{proof}
	We show (1), as (2) is dual.

	For $n \geq 0$, consider the (non-commuting) triangle
	\[ \label[diagram]{diagm:sd-triangle}
	\begin{tikzcd}[column sep = 3em]
		\mSd(U \mSd \simp{n}) \ar[rd, "\msub", swap] \ar[rr, "\mSd \left( U \max \right)"] & {} & \mSd \simp{n}\\
		{} & \minm{(U \mSd \simp{n})} \ar[ur, hook]
	\end{tikzcd} \tag{$\ast$} \]
	The underlying simplicial set of the right and middle terms is the poset of subsets of $[n]$ ordered by containment.
	The underlying simplicial set of the left term is the poset of chains $A_0 \subseteq \dots \subseteq A_k$ of subsets of $[n]$ ordered by containment.
	We identify the (underlying) maps above as monotone functions 
	\[ \begin{tikzcd}[column sep = 3em]
		\sd (\sd [n]) \ar[rd, "\msub", swap] \ar[rr, "\mSd \left( U \max \right)"] & {} & \sd [n] \\
		{} & \sd [n] \ar[ur, "{\mathrm{id}}", swap]
	\end{tikzcd} \]
	defined by
	\begin{align*}
		\mSd (U \max) (A_0 \subseteq \dots \subseteq A_k) &= \{ \max A_0, \dots, \max A_k \} \\
		\msub (A_0 \subseteq \dots \subseteq A_k) &= A_k.
	\end{align*}
	The natural containment $\{ \max A_0, \dots, \max A_k  \} \subseteq A_k$ induces a marked homotopy 
	\[ \mSd(U \mSd \simp{n}) \times \maxm{(\simp{1})} \to \mSd \simp{n} \] 
	between the composites in Diagram (\ref{diagm:sd-triangle}). 
	The map
	\[ \mSd J \to \maxm{(\simp{1})} \]
	which witnesses invertibility of $(\msup)_1(\id[\maxm{(\simp{1})}]) \in (\mEx \maxm{(\simp{1})})_1$ induces a map
	\[ \begin{tikzcd}[column sep = 2.7em]
		\mSd(U \mSd J) \ar[r, "\msub"] & \mSd J \ar[r] & \maxm{(\simp{1})}.
	\end{tikzcd} \]
	Let $\alpha$ denote the composite
	\[ \begin{tikzcd}[column sep = small]
		\mSd(U \mSd (\simp{n} \times J)) \ar[r] & \mSd(U \mSd \simp{n}) \times \mSd(U \mSd J) \ar[r] & \mSd(U \mSd \simp{n}) \times \maxm{(\simp{1})} \ar[r] & \mSd \simp{n}
	\end{tikzcd} \]
	(where the first map is induced by the universal property of the product).
	Pre-composition by $\alpha$ induces a map
	\[ \alpha^* \from \mEx(X, W) \to  \mEx \left( \maxm{\mEx(X, W)} \right)^{J}. \]
	By \cref{min-ex-min-diagram}, this map witnesses that the diagram
	\[ \begin{tikzcd}
		\mEx (X, W) \ar[rr, "\msup"] \ar[rd, hook, "\mEx \eta_{(X, W)}", swap] & {} & \mEx \left( \maxm{\mEx (K, W)} \right) \\
		{} & \mEx \maxm{X} \ar[ur, "\mEx (\maxm{(\msup)})"'] & {}
	\end{tikzcd} \]
	commutes up to $J$-homotopy.

	Note that, for $f \from \mSd \simp{n} \to (X, W)$, the map 
	\[ f \alpha \from \mSd(U \mSd \simp{n}) \to (X, W) \]
	sends weak equivalences to equivalences as $f$ does by \cref{min-ascends}.
	Thus, $\alpha^*$ factors as
	\[ \begin{tikzcd}
		\mEx(X, W) \times J \ar[rd, dotted, "\beta"'] \ar[rr, "{\alpha^*}"] & {} & \mEx(\maxm{\mEx(X, W)}) \\
		{} & \mEx(\natm{\mEx(X, W)}) \ar[ur, hook, "\mEx \eta"'] & {}
	\end{tikzcd} \]
	We show that $\beta$ is the desired homotopy by verifying its endpoints evaluate to $\msup$ and $\mEx(\msupb)$, respectively.
	The diagram
	\[ \begin{tikzcd}[column sep = 5em]
		\mEx(X, W) \ar[r, "\msup_{\natm{\mEx(X, W)}}"] \ar[rd, "\msup_{\maxm{\mEx(X, W)}}"'] & \mEx(\natm{\mEx(X, W)}) \ar[d, hook, "\mEx \eta"] \\
		{} & \mEx(\maxm{\mEx(X, W)})
	\end{tikzcd} \]
	commutes by naturality of $\msup \from U \Rightarrow \mEx$ applied to the inclusion $\natm{\mEx(X, W)} \hookrightarrow \maxm{\mEx(X, W)}$.
	From this, we have that
	\[ \mEx \eta \circ \msup_{\natm{\mEx(X, W)}} = \msup_{\maxm{\mEx(X, W)}} = \alpha^* \vert_{\mEx(X, W) \times \{ 0 \}} =  \mEx \eta \circ \beta \vert_{\mEx(X, W) \times \{ 0 \}}, \]
	which implies $\msup_{\natm{\mEx(X, W)}} = \beta \vert_{\mEx(X, W) \times \{ 0 \}}$ as $\mEx \eta$ is monic.
	For the other endpoint, the diagram
	\[ \begin{tikzcd}
		(X, W) \ar[r, "\msupb"] \ar[d, hook, "\eta'"'] & \natm{\mEx(X, W)} \ar[d, hook, "\eta"] \\
		\maxm{X} \ar[r, "\maxm{(\msup)}"] & \maxm{\mEx(X, W)}
	\end{tikzcd} \]
	commutes as it commutes after applying $U \from \msSet \to \sSet$.
	Thus, this diagram commutes after applying $\mEx \from \msSet \to \sSet$.
	Hence,
	\[ \mEx \eta \circ \mEx \msupb = \mEx \maxm{(\msup)} \circ \mEx \eta' = \alpha^* \vert_{\mEx(X, W) \times \{ 1 \}} = \mEx \eta \circ \beta \vert_{\mEx(X, W) \times \{ 1 \} }, \]
	which implies $\mEx \msupb = \beta \vert_{\mEx(X, W) \times \{ 1 \}}$ as $\mEx \eta$ is monic.
\end{proof}

\begin{corollary} \label{min-ex-min-htpy}
	For a marked quasicategory $(\mc{C}, W)$ satisfying CLF, the (underlying) maps in the diagram
	\[ \begin{tikzcd}[column sep = 3.7em]
		\natm{\mEx (\mc{C}, W)} \ar[r, yshift=-0.75ex, "\msupb", swap] \ar[r, yshift=0.75ex, "\natm{\mEx \left( \msupb \right)}"] & \natm{\mEx \left( \natm{\mEx (\mc{C}, W)} \right)}
	\end{tikzcd} \]
	are $E^1$-homotopic.
\end{corollary}
\begin{proof}
	The underlying maps are $J$-homotopic by \cref{min-ex-min-J-htpy}.
	The result follows since the codomain is a quasicategory.
\end{proof}
\begin{corollary} \label{ex-msup-weq}
	For a marked quasicategory $(\mc{C}, W)$ satisfying CLF, the map
	\[ \mEx(\msupb) \from \mEx(\mc{C}, W) \to \mEx \left( \natm{\mEx(\mc{C}, W)} \right) \]
	is an equivalence of categories.
\end{corollary}
\begin{proof}
	This map is $E^1$-homotopic to $\msup$ (\cref{min-ex-min-htpy}), which is an equivalence (\cref{ex-nat-weq}).
\end{proof}

With this, we may prove the main theorem of this section.
\begin{theorem} \label{mEx-computes-localization}
	Let $(\mc{C}, W)$ be a marked quasicategory
	\begin{enumerate}
		\item If $(\mc{C}, W)$ satisfies CLF then $\msup \from \mc{C} \to \mEx(\mc{C}, W)$ is the localization of $\mc{C}$ at $W$. 
		\item If $(\mc{C}, W)$ satisfies CRF then $\minsup \from \mc{C} \to \mExop(\mc{C}, W)$ is the localization of $\mc{C}$ at $W$. 
	\end{enumerate}
\end{theorem}
\begin{proof}
	We prove (1), as (2) is formally dual.

	Let $\gamma \from \mc{C} \to \mc{C}[W^{-1}]$ be the localization of $\mc{C}$ at $W$.
	By \cref{min-ascends}, we have the following diagram of maps
	\[ \begin{tikzcd}
		\mc{C} \ar[r, "\msup"] \ar[d, "\gamma", swap] & \mEx(\mc{C}, W) \\
		\mc{C}[W^{-1}] \ar[ur, dotted, "F", swap]
	\end{tikzcd} \]
	which commutes up to natural equivalence.
	This implies the map 
	\[ \msup_{\mc{C}[W^{-1}]} \from \mc{C}[W^{-1}] \to \mEx(\natm{\mc{C}[W^{-1}]}) \] 
	is $E^1$-homotopic to the composite $\mEx(\gamma) \circ F$ since
	\[ (\mEx(\gamma) \circ F \circ \gamma) \sim (\mEx(\gamma) \circ \msup) = (\msup \circ \gamma), \]
	where the equality follows from naturality.
	Moreover, this diagram ascends to a diagram of marked simplicial sets
	\[ \begin{tikzcd}
		(\mc{C}, W) \ar[r, "\msupb"] \ar[d, "\overline{\gamma}", swap] & \natm{\mEx(\mc{C},W)} \\
		\natm{\mc{C}[W^{-1}]} \ar[ur, "\overline{F}", swap]
	\end{tikzcd} \]
	whose underlying maps commute up to $E^1$-homotopy.
	By \cref{mEx-preserves-htpy-equivs}, the bottom right triangle in
	\[ \begin{tikzcd}
		\mc{C}[W^{-1}] \ar[r, "F"] \ar[d, "\msup", swap] & \mEx(\mc{C}, W) \ar[d, "\mEx(\msupb)"', swap] \ar[ld, "\mEx(\overline{\gamma})" description] \\
		\mEx (\natm{\mc{C}[W^{-1}]}) \ar[r, "\mEx \overline{F}"'] & \mEx(\natm{\mEx(\mc{C}, W)})
	\end{tikzcd} \]
	commutes up to $E^1$-homotopy.
	As the vertical maps are equivalences (by \cref{ex-nat-weq} and \cref{ex-msup-weq}, respectively), the map $F$ is an equivalence by 2-out-of-6.
\end{proof}

\begin{corollary} \label{ExC_Kan_iff_clf}
  Let $\mc{C}$ be a quasicategory.
  Then $\Ex \mc{C}$ is a Kan complex if and only if $\maxm{\mc{C}}$ admits the calculus of left fractions. 
\end{corollary}

\begin{proof}
  If $\Ex \mc{C}$ is a Kan complex, then by the first part of \cref{mEx-qcat}, we get that $\maxm{\mc{C}}$ admits the calculus of left fractions.
  The converse follows from \cref{mEx-computes-localization,Ex-v-mEx}.
\end{proof}

\begin{remark}
  A special case of \cref{ExC_Kan_iff_clf} for the nerve of a category was observed in \cite[Rmk.~5.8]{latch-thomason-wilson} (where it is attributed to Fritsch and Latch) and proven in \cite[Thm.~2.1]{meier-ozornova:partial-model}.
  Note that a fully marked category $\maxm{\mc{C}}$ satisfies CLF if and only if it satisfies proper CLF, and hence the result above indeed generalizes those of \cite{latch-thomason-wilson,meier-ozornova:partial-model}.
\end{remark}

\begin{example}
  Recall that a localization $\gamma \from \mc{C} \to \mc{C}[W^{-1}]$ is \emph{reflective} is $\gamma$ admits a full and faithful right adjoint.
  These are occasionally referred to simply as \emph{localizations}, e.g., \cite[Def.~5.2.7.2]{lurie:htt}.
  In view of \cref{reflective-localization-satisfies-clf}, these are therefore necessarily of the form $\msup \from \mc{C} \to \mEx(\mc{C}, W)$ (as described in \cref{mEx-computes-localization}) where $W$ is the collection of maps in $\mc{C}$ inverted by $\gamma$.
  Dually, a coreflective localization is of the form $\minsup \from \mc{C} \to \mExop(\mc{C}, W)$.
\end{example}

In the setting of (co)fibration categories, we have that the class of weak equivalences satisfies the stronger 2-out-of-6 property if and only if it is saturated, i.e.~weak equivalences are exactly the maps inverted by the localization functor (cf.~\cite[Thm.~7.2.7]{radulescu-banu}).
An analogous result was proven in the context of categories satisfying calculus of fractions by Kashiwara and Schapira \cite[Prop.~7.1.20.(2)]{kashiwara-schapira}.
To introduce our generalization, we first define when a collection of morphisms in a quasicategory is saturated.

\begin{definition}
    A collection of morphisms $W \subseteq \operatorname{Mor} \mc{C}$ is \emph{saturated} if the localization map
    \[ \gamma \from (\mc{C}, W) \to \natm{\mc{C}[W^{-1}]} \] 
    reflects markings. That is, if $\gamma(f)$ is an equivalence in $\mc{C}[W^{-1}]$ then $f \in W$.
\end{definition}
\begin{proposition}[cf.~{\cite[Prop.~7.1.20.(2)]{kashiwara-schapira}}] \label{saturated_iff_2-outta-6}
    If $(\mc{C}, W)$ satisfies CLF then $W$ is closed under 2-out-of-6 if and only if $W$ is saturated.
\end{proposition}
\begin{proof}
    The reverse direction is immediate. 
    For the forward direction, suppose $f \from a \to b$ becomes an equivalence in $\mEx(\mc{C}, W)$.
    This gives a map $\mSd J \to (\mc{C}, W)$ of the form
    \[ \begin{tikzpicture}[commutative diagrams/every diagram,node distance=3.7em]
        \node (MM) {$b$};
        \node (MT) [above=2.8em of MM] {$a$};
        \node (MB) [below=2.8em of MM] {$b$};
        \node (LT) [left=of MT] {$a'$};
        \node (LLT) [left=of LT] {$b$};
        \node (RB) [right=of MB] {$a'$};
        \node (RRB) [right=of RB] {$a$};
        \coordinate (LMcoord) at (barycentric cs:LLT=1,MB=1);
        \coordinate (RMcoord) at (barycentric cs:MT=1,RRB=1);
        \path (LMcoord) -- +(-0.25, -0.25) node (LM) {$b$};
        \path (RMcoord) -- +(0.25, 0.25) node (RM) {$a$};
        \node (CL) at (barycentric cs:LLT=1,MT=1,MB=1) {$b''$};
        \node (CR) at (barycentric cs:MT=1,MB=1,RRB=1)
        {$a''$};

        \path[commutative diagrams/.cd, every arrow, every label]
            (LLT) edge node {$g$} (LT)
                    edge[commutative diagrams/equal] (LM)
            (MT) edge node[swap] {$w$} (LT)
                    edge node {$f$} (MM)
                    edge[commutative diagrams/equal] (RM)
            (MB) edge[commutative diagrams/equal] (MM)
                    edge node {$g$} (RB)
                    edge[commutative diagrams/equal] (LM)
            (RRB) edge node[swap] {$w$} (RB)
                    edge[commutative diagrams/equal] (RM)
            (LT) edge node {$h$} (CL)
            (LM) edge node {$\sim$} (CL)
            (MM) edge (CL)
                    edge (CR)
            (RM) edge node[swap] {$\sim$} (CR)
            (RB) edge node[swap, near start] {$\sim$} (CR);
    \end{tikzpicture} \]
    from which we extract commutative diagrams
    \[ \begin{tikzcd}
        a \ar[r, "\sim"] \ar[d, "f"'] & a'' \\
        b \ar[r, "g"] & a' \ar[u, "\sim"']
    \end{tikzcd} \qquad \begin{tikzcd}
        {} & a' \ar[rd, "h"] & {} \\
        b \ar[rr, "\sim"] \ar[ur, "g"] & {} & b''
    \end{tikzcd} \]
    in $\mc{C}$.
    Specifically, the left diagram is a composite of two squares in the right subdivided triangle; the right diagram is obtained from the left-most square in the left subdivided triangle after composing.

    Recall that closure under 2-out-of-6 implies closure under 2-out-of-3, which further implies $W$ is strongly closed under composition.
    In the left diagram, since $a \to a''$ and $a' \to a''$ are in $W$, 2-out-of-3 gives that any composite of $g$ with $f$ is in $W$.
    In the right diagram, since the bottom map is in $W$, any composite of $h$ with $g$ is in $W$.
    By 2-out-of-6 on the triple
    \[ \begin{tikzcd} 
        a \ar[r, "f"] & b \ar[r, "g"] & a' \ar[r, "h"] & b'' ,
    \end{tikzcd} \]
    the morphism $f$ is in $W$.
\end{proof}

\chapter*{Properties of the localization}
\cftaddtitleline{toc}{chapter}{Properties of the localization}{}

\section{Mapping spaces in the localization} \label{sec:mapping-space}

In this section, we give explicit descriptions for the mapping spaces in $\mEx(\mc{C}, W)$ and $\mExop(\mc{C}, W)$.
To do so, we give two constructions of the ``marked slice'' category.
Mirroring the usual slice construction, both are defined by mapping out of ``marked cones'' over an $n$-simplex. 
We form one version of the marked cone using the product; the other is formed using the join.
These allow us to define the \emph{simplicial set of fractions} between two objects in a quasicategory (\cref{def:sset-of-fractions}), whic yields our first description of the mapping space (\cref{mapping-space-is-mEx-LF}).
Our second description of the mapping spaces in $\mEx(\mc{C}, W)$ is as a colimit of mapping spaces in $\mc{C}$ (\cref{mapping-space-is-colim}), indexed by the marked slice category.

Define a cosimplicial object $\lcone \from \Delta \to (\msSet)_{\ast}$ by the pushout
\[ \begin{tikzcd}
    \minm{(\simp{n})} \ar[r] \ar[d, "i_0"'] \ar[rd, phantom, "\ulcorner" very near end] & \simp{0} \ar[d, "{[i_0]}"] \\
    \minm{(\simp{n})} \times \maxm{(\simp{1})} \ar[r] & \lcone \simp{n}
\end{tikzcd} \]
We regard $\lcone \simp{n}$ as pointed with distinguished basepoint $[i_0]$.

Geometrically, one thinks of $\lcone \simp{n}$ as a prism $\simp{n} \times \simp{1}$ which has been quotiented to a cone at the 0-endpoint.
The 1-simplices which extend from the cone point to the original $n$-simplex are exactly the marked 1-simplices of $\lcone \simp{n}$.

This induces a functor $\lcone \from \sSet \to (\msSet)_{\ast}$ via extension by colimits which sends a simplicial set $X$ to the pushout
\[ \begin{tikzcd}
    X \ar[r] \ar[d, "i_0"'] \ar[rd, phantom, "\ulcorner" very near end] & \simp{0} \ar[d, "{[i_0]}"] \\
    X \times \maxm{(\simp{1})} \ar[r] & \lcone X
\end{tikzcd} \]
with distinguished basepoint $[i_0]$.
The additional assignment 
\[ \lcone (\maxm{[1]}) := \colim \left( \begin{tikzcd}
    \maxm{(\simp{1})} \ar[r] \ar[d, "i_0"'] & \simp{0} \\
    \maxm{(\simp{1})} \times \maxm{(\simp{1})}
\end{tikzcd} \right) \]
extends $\lcone$ to a functor $\msSet \to (\msSet)_{\ast}$.
Pre-composition gives a functor $\lpathloop[\uvar]{\uvar} \from (\msSet)_{\ast} \to \msSet$ defined, on an object $(X, W, x)$, by
\[ (\lpathloop[W]{x})_n := (\msSet)_{\ast}(\lcone \simp{n}, (X, W, x)). \]
We refer to $\lpathloop{x}$ as the \emph{marked slice under $x$}, as it is the simplicial subset of the ``fat'' slice under $x$ (cf.~\cite[\S4.2.1]{lurie:htt}) spanned by 0-simplices $(y, x \to y)$ such that $x \to y$ is in $W$. 
Explicitly, an $n$-simplex of $\lpathloop[W]{x}$ is a map $u \from \simp{n} \times \simp{1} \to X$ such that
\begin{itemize}
    \item the restriction $\restr{u}{\simp{n} \times \{ 0 \}}$ is constant at the 0-simplex $x$; and 
    \item for $i = 0, \dots, n$, the 1-simplex $\restr{u}{(i, 0) \leq (i, 1)} \from \simp{1} \to X$ is marked.
\end{itemize}
A 1-simplex $u \from (\simp{1})^2 \to X$ is marked if the restriction $\restr{u}{(0, 1) \leq (1, 1)} \from \simp{1} \to X$ is a marked 1-simplex.
Note the map
\[ \minm{(\simp{n})} \times \maxm{(\simp{1})} \to \lcone \simp{n} \] 
induces a natural inclusion
\[ \lpathloop{x} \to (X, W)^{\maxm{(\simp{1})}} \]
by pre-composition.
We write $\Pi_x$ for the composite
\[ \lpathloop{x} \ito (X, W)^{\maxm{(\simp{1})}} \xrightarrow{i_1^*} (X, W) \]
which we refer to as the \emph{projection map}.

We similarly define a cosimplicial object $\rcone \simp{n}$ by the pushout,
\[ \begin{tikzcd}
    \minm{(\simp{n})} \ar[r] \ar[d, "i_1"'] \ar[rd, phantom, "\ulcorner" very near end] & \simp{0} \ar[d, "{[i_1]}"] \\
    \minm{(\simp{n})} \times \maxm{(\simp{1})} \ar[r] & \rcone \simp{n}
\end{tikzcd} \]
which one thinks of as a prism $\minm{(\simp{n})} \times \maxm{(\simp{1})}$ which has been quotiented to a cone at the 1-endpoint.
This defines a functor $\rcone \from \msSet \to (\msSet)_{\ast}$ via extension by colimits (with an analogous definition on $\maxm{\simp{1}}$).
Its right adjoint $\rpathloop[\uvar]{\uvar} \from (\msSet)_{\ast} \to \msSet$ is defined by
\[ (\rpathloop[W]{x})_n := (\msSet)_{\ast}(\rcone \simp{n}, (X, W, x)). \]
We refer to $\rpathloop{x}$ as the \emph{marked slice over $x$}.
This functor comes with a natural projection map
\[ \rpathloop{x} \ito (X, W)^{\maxm{(\simp{1})}} \xrightarrow{i_0^*} (X, W) \]
Note that for $(X, W, x) \in (\msSet)_{\ast}$, there is an isomorphism
\[ (\lpathloop{x})^\op \cong \rpathloop[W^\op]{x} \]
natural in $(X, W, x)$.

The following proposition follows by construction.
\begin{proposition} \label{pb-pathloop-is-qcat}
    For a marked simplicial set $(X, W)$ and $x \in X_0$, the squares
    \[ \begin{tikzcd}
        \lpathloop{x} \ar[r, hook] \ar[d] \ar[rd, phantom, "\lrcorner" very near start] & (X, W)^{\maxm{(\simp{1})}} \ar[d, "i_0^*"] \\
        \simp{0} \ar[r, "x"] & (X, W)
    \end{tikzcd} \qquad \begin{tikzcd}
        \rpathloop{x} \ar[r, hook] \ar[d] \ar[rd, phantom, "\lrcorner" very near start] & (X, W)^{\maxm{(\simp{1})}} \ar[d, "i_1^*"] \\
        \simp{0} \ar[r, "x"] & (X, W)
    \end{tikzcd} \]
    are pullbacks. \qed
\end{proposition}
\begin{proposition} \label{pathloop-qcat-clf}
    Let $x \in X$ be a 0-simplex in a marked simplicial set $(X, W)$ such that $W$ is strongly closed under composition. 
    \begin{enumerate}
        \item If $X$ is a quasicategory then $\lpathloop{x}$ and $\rpathloop{x}$ are quasicategories.
        \item If $(X, W)$ satisfies CLF then $\lpathloop{x}$ satisfies CLF.
        \item If $(X, W)$ satisfies CRF then $\rpathloop{x}$ satisfies CRF.
    \end{enumerate}
\end{proposition}
\begin{proof}
    For (1), applying \cite[Cor.~2.4.7.12]{lurie:htt} to the map $\simp{0} \xrightarrow{x} X$ gives that $\lpathloop{x} \to \simp{0}$ is a fibration.
    Dually, $\rpathloop{x} \to \simp{0}$ is a fibration.

    For (2), the result for $\lpathloop{x}$ follows from \cref{lj-mapping-space-lift-lemma}, as $\lpathloop{x} \to \simp{0}$ is a pullback of a map which has the right lifting property against $\{ \dfLJ \ito \dfLI \mid n \geq 2, \ 0 < k \leq n \}$ by \cref{pb-pathloop-is-qcat}.

    
    The proof of (3) is dual.
\end{proof}

\newcommand{\ljcone}{\lcone_{\join}} 
\newcommand{\rjcone}{\rcone_{\join}} 
We now describe an alternate construction of the marked cone and slice using the join rather than the product.
Let $\ljcone \from \sSet \to (\msSet)_*$ denote the functor which sends a simplicial set $X$ to the simplicial set $\simp{0} \join X$ (whose distinguished point is the cone point), where every 1-simplex connected to the cone point is marked.
This functor has a right adjoint $(\msSet)_* \to \sSet$ which sends $(X, W, x)$ to the simplicial set
\[ (x \slice^{\join} W)_n := (\msSet)_*(\ljcone \simp{n}, (X, W, x)). \]
Explicitly, $x \slice^{\join} W$ is the simplicial subset of the ordinary slice category under $x$ spanned by 0-simplices $(y, x \to y)$ such that $x \to y$ is in $W$.
The map $\lcone \simp{n} \to \ljcone \simp{n}$ induced by the square
\[ \begin{tikzcd}[column sep = 3.5em]
    \minm{(\simp{n})} \ar[r] \ar[d, "i_0"'] & \simp{0} \ar[d, "\bot"] \\
    \minm{(\simp{n})} \times \maxm{(\simp{1})} \ar[r, "{(i, 0) \mapsto i}", "{(i, 1) \mapsto \bot}"'] & \ljcone \simp{n}
\end{tikzcd} \]
commutes with faces and degeneracies.
Thus, pre-composition induces a map
\[ x \slice^{\join} W \to \lpathloop{x} \]
which is natural in $(X, W, x)$.
Analogously, one defines a functor $\rjcone \from \sSet \to (\msSet)_*$ and a map $\rcone \simp{n} \to \rjcone \simp{n}$, which induces a map
\[ W \slice^{\join} x \to \rpathloop{x}. \]
\begin{proposition} \label{lpathloop-weq-slice-equiv}
    Let $x$ be an object in a marked quasicategory $(\mc{C}, W)$.
    The maps
    \[ x \slice^{\join} W \to \lpathloop{x}, \quad W \slice^{\join} x \to \rpathloop{x} \]
    are categorical equivalences.
\end{proposition}
\begin{proof}
    These maps are restrictions onto full subcategories of the canonical equivalences from the ordinary slice to the ``fat'' slice (\cite[Prop.~4.2.1.5]{lurie:htt}).
\end{proof}
For the remainder of this section, we work with the marked ``fat'' slice $\lpathloop{x}$ rather than the marked ordinary slice $x \slice^{\join} W$.

With this, we introduce the following auxiliary construction.
\begin{definition} \label{def:sset-of-fractions}
    Let $(X, W)$ be a marked simplicial set and $x, y \in X$ be 0-simplices.
    \begin{enumerate}
        \item The simplicial set $\LF{(X, W)}(x, y)$ of \emph{left fractions} from $x$ to $y$ is the pullback
        \[ \begin{tikzcd}
            \LF{(X, W)}(x, y) \ar[r] \ar[d] \ar[rd, phantom, "\lrcorner" very near start] & (x \slice X, W) \ar[d, "\Pi_x"] \\
            \lpathloop{y} \ar[r, "\Pi_y"] & (X, W)
        \end{tikzcd} \]
        in $\msSet$.
        \item The simplicial set $\LF{(X, W)}(x, y)$ of \emph{right fractions} from $x$ to $y$ is the pullback
        \[ \begin{tikzcd}
            \RF{(X, W)}(x, y) \ar[r] \ar[d] \ar[rd, phantom, "\lrcorner" very near start] & (X \slice y , W) \ar[d, "\Pi_y"] \\
            \rpathloop{x} \ar[r, "\Pi_x"] & (X, W)
        \end{tikzcd} \]
        in $\msSet$.
    \end{enumerate}
\end{definition}

Our interest in the simplicial set of left (and right) fractions is in modelling the mapping space of the localization.
\begin{theorem} \label{mapping-space-is-mEx-LF}
    Let $(X, W)$ be a marked simplicial set and $x, y$ be 0-simplices.
    There are isomorphisms
    \begin{align*}
        \map_L(\mEx(X, W), x, y) &\cong \mEx(\LF{(X, W)}(x, y)), \\ 
        \map_R(\mExop(X, W), x, y) &\cong \mExop(\RF{(X, W)}(x, y))
    \end{align*} 
    natural in $(X, W)$ and $x, y$.
\end{theorem}

The proof of \cref{mapping-space-is-mEx-LF} will require some additional definitions and lemmas.
We present these only for $\mSd$ and $\mEx$, omitting their dual statements, as their primary purpose is in proving \cref{mapping-space-is-mEx-LF}.

Observe that the left mapping space of a quasicategory arises as pre-composition by a cosimplicial object.
For $n \geq 0$, let $\simp{n+1} / \face{}{0}$ denote the simplicial set obtained by the pushout
\[ \begin{tikzcd}
    \simp{n} \ar[d, "\face{}{0}"'] \ar[r] \ar[rd, phantom, "\ulcorner" very near end] & \simp{0} \ar[d, "{[\face{}{0}]}"] \\
    \simp{n+1} \ar[r] & \simp{n+1} / \face{}{0}
\end{tikzcd} \]
Let $q$ denote the quotient map $\simp{n+1} \to \simp{n+1} / \face{}{0}$.
We regard $\simp{n+1}/\face{}{0}$ as bi-pointed at $0$ and $[\face{}{0}]$.
The left mapping space $\map_L \from \sSet_{\ast \ast} \to \sSet$ is exactly the induced functor given by pre-composition, i.e.
\[ \map_L (X, x, y)_n := \sSet_{\ast \ast}((\simp{n+1} / \face{}{0}, 0, [\face{}{0}]), (X, x, y)). \]


Define a map $f \from \simp{0} \join \mSd \simp{n} \to \mSd \simp{n+1}$ by
\[ f(A) := \begin{cases}
    \{ 0 \} & A = \bot \\
    \face{}{0}(A) \cup \{ 0 \} & \text{otherwise.}
\end{cases} \]
This formula defines a map on the underlying simplicial sets as the domain and codomain are nerves of posets.
One verifies this map sends marked 1-simplices to marked 1-simplices.
We also define a map $g \from \mSd \simp{n} \times \maxm{(\simp{1})} \to \mSd \simp{n+1}$ by
\[ g(A, t) := \begin{cases}
    \face{}{0}(A) & \text{if $t = 0$} \\
    \face{}{0}(A) \cup \{ 0 \} & \text{if $t = 1$}.
\end{cases} \]
\begin{lemma} \label{sd-simp-po}
    The square
    \[ \begin{tikzcd}
        \mSd \simp{n} \ar[r, "\id \times \{ 1 \}"] \ar[d, "i_1"'] & \mSd \simp{n} \times \maxm{(\simp{1})} \ar[d, "g"] \\
        \simp{0} \join \mSd \simp{n} \ar[r, "f"] & \mSd \simp{n+1}
    \end{tikzcd} \]
    is a pushout.
\end{lemma}
\begin{proof}
    The square commutes by definition.
    Given a non-degenerate simplex $(A_0 \subsetneq \dots \subsetneq A_n)$ of $\mSd \simp{n+1}$, we show this simplex is in the image of $f$ or $g$ by case analysis on $A_0$.
    Moreover, we show the intersection of the images is the image of the composite map in the square.
    This suffices as each map in the square is injective.

    If $A_0 = \{ 0 \}$ then this simplex is in the image of $f$ as
    \[ (A_0 \subsetneq A_1 \subsetneq \dots \subsetneq A_n) = f(\bot \subsetneq \degen{}{0}(A_1 - \{ 0 \}) \subsetneq \dots \subsetneq \degen{}{0}(A_n - \{ 0 \})). \]
    Note this simplex is not contained in the image of $g$ as $\{ 0 \}$ is not contained in the image of $g$.

    If $A_0$ does not contain 0 then let $i = 0, \dots, n$ be maximal such that $0 \not\in A_i$.
    Applying $g$ to the $n$-simplex 
    \[ (\degen{}{0}(A_0), 0) \subsetneq \dots \subsetneq (\degen{}{0}(A_i), 0) \subsetneq (\degen{}{0}(A_{i+1}), 1) \subsetneq \dots \subsetneq (\degen{}{0}(A_n), 1)  \]
    gives exactly $(A_0 \subsetneq \dots \subsetneq A_n)$.
    Note this simplex is not contained in the image of $f$ as the 0-simplices in the image of $f$ are subsets that contain 0.

    The simplices that remain in our case analysis are exactly those in the set of non-degenerate simplices defined by
    \[ \{ (A_0 \subsetneq \dots \subsetneq A_n) \mid \{ 0 \} \subsetneq A_0 \}. \]
    Note the previous argumens show that the intersection of the images of $f$ and $g$ must be contained in this set.
    However, this set is contained in the image of the composite map since, for an $n$-simplex $(A_0 \subsetneq \dots \subsetneq A_n)$ in this set, we have equalities
    \begin{align*} 
        (f \circ i_1)(\degen{}{0}(A_0 - \{ 0 \}) \subsetneq \dots \subsetneq \degen{}{0}(A_n - \{ 0 \})) &= \face{}{0} \degen{}{0}(A_0 - \{ 0 \}) \cup \{ 0 \} \subsetneq \dots \subsetneq  \face{}{0}\degen{}{0}(A_n - \{ 0 \}) \cup \{ 0 \} ) \\
        &= (A_0 \subsetneq \dots \subsetneq A_n), 
    \end{align*}
    and
    \begin{align*}
        &(g \circ (\id \times \{ 1 \}))(\degen{}{0}(A_0 - \{ 0 \}) \subsetneq \dots \subsetneq \degen{}{0}(A_n - \{ 0 \})) \\
        &\quad = \face{}{0} \degen{}{0}(A_0 - \{ 0 \}) \cup \{ 0 \} \subsetneq \dots \subsetneq  \face{}{0}\degen{}{0}(A_n - \{ 0 \}) \cup \{ 0 \} ) \\
        &\quad = (A_0 \subsetneq \dots \subsetneq A_n). 
    \end{align*}
    This gives both that the intersection of the images is the image of the composite and that every non-degenerate $n$-simplex of $\mSd \simp{n+1}$ is contained in either the image of $f$ or the image of $g$.
\end{proof}

The square
\[ \begin{tikzcd}
    \mSd \simp{n} \ar[rr] \ar[d, "\id \times \{ 0 \}"'] & {} & \simp{0} \ar[d, "{[\face{}{0}]}"] \\
    \mSd \simp{n} \times \maxm{(\simp{1})} \ar[r, "g"] & \mSd \simp{n+1} \ar[r, "q"] & \mSd (\simp{n+1} / \face{}{0})
\end{tikzcd} \]
induces a map $[g] \from \lcone(\mSd \simp{n}) \to \mSd (\simp{n+1} / \face{}{0})$ by universal property of the pushout.
This map gives a commutative square
\[ \begin{tikzcd}
    \left( \varnothing \to \mSd \simp{n} \right) \ar[r, "i_1"] \ar[d] & \left( \{ [i_0] \} \to \lcone ( \mSd \simp{n}) \right) \ar[d, "{[g]}"] \\
    \left( \{ \bot \} \to \simp{0} \join \mSd \simp{n} \right) \ar[r, "{qf}"] & \left( \{ 0, [\face{}{0}] \} \to \mSd (\simp{n+1} / \face{}{0}) \right)
\end{tikzcd} \]
in the arrow category $\fcat{[1]}{\msSet}$.
One verifies these maps are natural in $n$, so that for a marked simplicial set $(X, W)$ and $x, y \in X$, this induces a commutative square
\[ \begin{tikzcd}
    \map_L(\mEx(X, W), x, y) \ar[r, "(qf)^*"] \ar[d, "{[g]^*}"'] & \mEx(x \slice (X, W)) \ar[d, "\mEx \Pi_x"] \\
    \mEx(\lpathloop{y}) \ar[r, "\mEx \Pi_y"] & \mEx(X, W)
\end{tikzcd} \]
by pre-composition, where one verifies that $\mEx \Pi_x$ is identical to the pre-composition map with $(\varnothing \ito \mSd \simp{n}) \ito (\{ \bot \} \to \simp{0} \join \mSd \simp{n})$ and $\mEx \Pi_y$ is identical to the pre-composition map with $(\varnothing \ito \mSd \simp{n}) \ito (\{ [i_0] \} \to \lcone (\mSd \simp{n}) )$.
\begin{lemma} \label{rel-sd-simp-po}
    \begin{enumerate}
        \item The square
        \[ \begin{tikzcd}
            \left( \varnothing \to \mSd \simp{n} \right) \ar[r, "i_1"] \ar[d] & \left( \{ [i_0] \} \to \lcone (\mSd \simp{n}) \right) \ar[d, "{[g]}"] \\
            \left( \{ \bot \} \to \simp{0} \join \mSd \simp{n} \right) \ar[r, "{qf}"] & \left( \{ 0, [\face{}{0}] \} \to \mSd (\simp{n+1} / \face{}{0}) \right)
        \end{tikzcd} \]
        is a pushout in the arrow category $\fcat{[1]}{\msSet}$.
        \item For a marked simplicial set $(X, W)$ and $x, y \in X$, the square
        \[ \begin{tikzcd}
            \map_L(\mEx(X, W), x, y) \ar[r, "(qf)^*"] \ar[d, "{[g]^*}"'] & \mEx(x \slice X, W) \ar[d, "\mEx \Pi_x"] \\
            \mEx(\lpathloop{y}) \ar[r, "\mEx \Pi_y"] & \mEx(X, W)
        \end{tikzcd} \]
        is a pullback in $\sSet$.
    \end{enumerate}
\end{lemma}
\begin{proof}
    For (1), it is clear that the domains form a pushout square.
    For the codomains, we consider the diagram
    \[ \begin{tikzcd}
        {} & \mSd \simp{n} \ar[r] \ar[d, "i_0"'] \ar[rd, phantom, "\ulcorner" very near end] & \simp{0} \ar[d, "{[i_0]}"] \\
        \mSd \simp{n} \ar[r, "i_1"] \ar[d, "i_1"'] \ar[rd, phantom, "\ulcorner" very near end] & \mSd \simp{n} \times \maxm{(\simp{1})} \ar[r] \ar[d, "g"'] & \lcone (\mSd \simp{n}) \ar[d, "{[g]}"] \\
        \simp{0} \join \mSd \simp{n} \ar[r, "f"] & \mSd \simp{n+1} \ar[r] & \mSd(\simp{n+1} / \face{}{0} )
    \end{tikzcd} \]
    The bottom left square is a pushout by \cref{sd-simp-po}.
    The top right square is a pushout by definition of $\lcone \from \sSet \to (\msSet)_{\ast}$.
    The right composite square is a pushout as $\mSd$ preserves colimits, hence the bottom right square is a pushout.
    This gives that the bottom composite square is a pushout.

    For statement (2), fix $n \geq 0$ and consider the product of co-spans
    \[ \begin{tikzcd}
        1 \ar[r, "\bot \mapsto x"] \ar[d, "\id"'] 
            & \msSet( \{ \bot \} , \{ x, y \} ) 
                \ar[d, "{(\varnothing \to \{ \bot \})^*}"] 
            & \fcat{[1]}{\msSet}\left( (\{ \bot \} \to \simp{0} \join \mSd \simp{n}), (\{ x, y \} \to (X, W)) \right) 
                \ar[l] \ar[d, "i_1^*"] \\
        1 \ar[r, "!"]
            & \msSet( \varnothing, \{ x, y \} ) & \fcat{[1]}{\msSet}\left( (\varnothing \to \mSd \simp{n}), (\{ x, y \} \to (X, W)) \right)
                \ar[l] \\
        1 \ar[r, "{[i_0] \mapsto y}"] \ar[u, "\id"]
            & \msSet( \{ [i_0] \}, \{ x, y \} ) 
                \ar[u, "{(\varnothing \to \{ [i_0] \})^*}"']
            & \fcat{[1]}{\msSet}((\{ [i_0] \} \to \lcone (\mSd \simp{n}) ), (\{ x, y \} \to (X, W)))
                \ar[l] \ar[u, "i_1^*"']
    \end{tikzcd} \]
    where the right horizontal arrows are given by applying the domain projection $\fcat{[1]}{\msSet} \to \msSet$.
    The pullback of each horizontal co-span gives the co-span
    \[ \begin{tikzcd}
        {} & \mEx(x \slice X, W)_n \ar[d, "(\mEx \Pi_x)_n"] \\
        \mEx(\lpathloop[(X, W)]{y})_n \ar[r, "(\mEx \Pi_y)_n"] & \mEx(X, W)_n
    \end{tikzcd} \]
    As limits commute with limits, this pullback is computed by the pullback of the pullbacks of the vertical co-spans.
    The co-span
    \[ \begin{tikzcd}
        {} & \fcat{[1]}{\msSet}((\{ 0, [\face{}{0}] \} \to \mSd(\simp{n+1} / \face{}{0}) ), (\{ x, y \} \to (X, W))) \ar[d] \\
        1 \ar[r, "\bot \mapsto x", "{[i_0] \mapsto y}"'] & \msSet(\{ \bot, [i_0] \} \to \{ x, y \})
    \end{tikzcd} \]
    is the pullback of the vertical co-spans as the contravariant Yoneda embedding takes colimits to limits.
    In particular, our computation of the top right corner follows from statement (1).
    This pullback is exactly the $n$-simplices of the left mapping space of $\mEx(X, W)$ from $x$ to $y$.
\end{proof}

With this, we may prove \cref{mapping-space-is-mEx-LF}, giving our first description of the mapping spaces in the localization.
\begin{proof}[Proof of \cref{mapping-space-is-mEx-LF}]
    The first isomorphism follows from \cref{pb-pathloop-is-qcat} and item (2) of \cref{rel-sd-simp-po}, as $\mEx \from \msSet \to \sSet$ preserves pullbacks.
    The second is dual.
\end{proof}

For our second description of the mapping spaces in the localization, which will mirror the classical description of mapping spaces (cf.~\cref{classical-fractions-colims}), we make use of the following colimit computation.
\begin{theorem} \label{pb-colim-for-mapping-space}
    Let $\mc{C}$ be a quasicategory with objects $x, y$. 
    Then,
    \begin{enumerate}
        \item The Kan fibrant replacement of the (underlying) simplicial set of left fractions $\LF{(\mc{C}, W)}(x, y)$ from $x$ to $y$ models the $\infty$-colimit of the functor
        \[ U(\lpathloop{y}) \xrightarrow{U \Pi} \mc{C} \xrightarrow{\mc{C}(x, -)} \Sp. \]
        That is, there is an equivalence of Kan complexes 
        \[ \colim \left( \mc{C}(x, -) \circ U \Pi \circ i \right) \simeq U (\LF{(\mc{C}, W)}(x, y))[\LF{(\mc{C}, W)}(x, y)_1^{-1}]. \]
        \item The Kan fibrant replacement of the (underlying) simplicial set of right fractions $\RF{(\mc{C}, W)}(x, y)$ from $x$ to $y$ models the $\infty$-colimit of the functor
        \[ U(\rpathloop{x})^\op \xrightarrow{(U \Pi)^\op} \mc{C}^\op \xrightarrow{\mc{C}(-, y)} \Sp. \]
        That is, there is an equivalence of Kan complexes 
        \[ \colim \left( \mc{C}(-, y) \circ (U \Pi)^\op \circ i \right) \simeq U (\RF{(\mc{C}, W)}(x, y))[\RF{(\mc{C}, W)}(x, y)_1^{-1}]. \]
    \end{enumerate}
\end{theorem}
Before proceeding to the proof of \cref{pb-colim-for-mapping-space}, we make the following remark to clarify any potential size issues.
\begin{remark}
    For the colimit formulas given in \cref{pb-colim-for-mapping-space}, we caution that the indexing category may fail to be $\mathcal{U}$-small, even if $\mc{C}$ is locally $\mathcal{U}$-small.
    One way of guaranteeing that the colimit is $\mathcal{U}$-small is to assume that $y \slice W$ admits a final functor from a $\mathcal{U}$-small $\infty$-category (and analogously, that $W \slice x$ admits an initial functor from a $\mathcal{U}$-small $\infty$-category).

    Going forward, whenever we speak of a marked category $(\mc{C}, W)$, we will always assume the existence of a large enough universe $\mathcal{U}$ so that colimits indexed by $y \slice W$ or $W \slice x$ exist in the $\infty$-category of $\mathcal{U}$-small $\infty$-groupoids.
    For this reason, we will supress any smallness conditions from our statements.
\end{remark}
\begin{proof}[Proof of \cref{pb-colim-for-mapping-space}]
    We prove (1), as (2) is formally dual.

    By \cite[Cor.~3.3.4.6]{lurie:htt}, the colimit of any functor $F \from \mc{C} \to \Sp$ is computed by the Kan fibrant replacement of the domain of the left fibration classified by $F$.
    Thus, it suffices to show the functor $\mc{C}(x, -) \circ \Pi$ classifies the left fibration $\LF{(\mc{C}, W)}(x, y) \to \lpathloop{y}$.
    
    Applying \cite[Cor.~5.3.21]{cisinski:higher-categories}, the right square in
    \[ \begin{tikzcd}
        \LF{(\mc{C}, W)}(x, y) \ar[r] \ar[d, two heads] \ar[rd, phantom, "\lrcorner" very near start] & x \slice \mc{C} \ar[d, two heads] \ar[r] & \ULF \ar[d, two heads] \\
        \lpathloop{y} \ar[r, "\Pi"] & \mc{C} \ar[r, "{\mc{C}(x, -)}"] & \Sp
    \end{tikzcd} \]
    is a homotopy pullback (in the Joyal model structure) as $\mc{C}(x, -)$ classifies the left fibration $x \slice \mc{C} \to \mc{C}$.
    The left square is a homotopy pullback as a pullback along a left fibration.
    Thus, the composite square is a homotopy pullback by \cite[Cor.~5.3.21]{cisinski:higher-categories}, which suffices.
\end{proof}
\Cref{pb-colim-for-mapping-space} gives a description of mapping spaces in the localization as a colimit of mapping spaces in $\mc{C}$ indexed by the marked slice category.
\begin{corollary} \label{mapping-space-is-colim}
    Let $x, y$ be objects in a marked quasicategory $(\mc{C}, W)$.
    \begin{enumerate}
        \item If $(\mc{C}, W)$ satisfies CLF then there is an equivalence
        \[ \map_L(\mEx(\mc{C}, W), x, y) \simeq \colim(\mc{C}(x, \uvar) \circ U\Pi ) \]
        which is natural in $x, y$ and $(\mc{C}, W)$.
        \item If $(\mc{C}, W)$ satisfies CRF then there is an equivalence
        \[ \map_R(\mExop(\mc{C}, W), x, y) \simeq \colim(\mc{C}(\uvar, y) \circ U\Pi^\op ) \]
        which is natural in $x, y$ and $(\mc{C}, W)$.
    \end{enumerate} 
\end{corollary}
\begin{proof}
    We prove (1), as (2) is analogous.
    
    By \cref{mapping-space-is-mEx-LF}, we have an isomorphism
    \[ \map_L(\mEx(\mc{C}, W), x, y) \cong \mEx(\LF{(\mc{C}, W)}(x, y)), \]
    natural in $(\mc{C}, W)$, $x$, and $y$.
    As the left simplicial set is Kan, the right one is.
    The marked 1-simplices in $\LF{(\mc{C}, W)}(x, y)$ are closed under composition as $W$ is strongly closed under composition.
    Thus, \cref{mEx-qcat} gives that $\LF{(\mc{C}, W)}(x, y)$ satisfies CLF.
    By \cref{mEx-computes-localization}, the map
    \[ \msup \from U\LF{(\mc{C}, W)}(x, y) \to \mEx (\LF{(\mc{C}, W)}(x, y)) \]
    is the localization of $\LF{(\mc{C}, W)}(x, y)$.
    As the codomain is Kan, this map is also the localization of $U \LF{(\mc{C}, W)}(x, y)$ at all maps, i.e.~the Kan fibrant replacement of $U \LF{(\mc{C}, W)}(x, y)$.
    The result then follows by \cref{pb-colim-for-mapping-space}.
\end{proof}

\section{(Co)limits in the localization} \label{sec:limits}

This section fully establishes the quasicategorical analogue of \cref{classical-fractions-colims}, giving conditions for the existence of (co)limits in the localization (\cref{localization-has-limits}).
Following the argument given in \cite[Ch.~1]{gabriel-zisman}, we prove this by showing the categories $\lpathloop{x}$ and $\rpathloop{x}$ are filtered if $W$ is closed under 2-out-of-3 (\cref{lpathloop-filtered}).
In fact, we show more generally that, in the presence of either CLF (or CRF) on a maximally-marked quasicategory $\maxm{\mc{C}}$, being (co)filtered is \emph{equivalent} to being weakly contractible (\cref{filtered-equiv-fibr-replacement-contr}).
From this, one concludes that the mapping spaces in $\mEx(\mc{C}, W)$ are \emph{filtered} colimits of the mapping spaces in $\mc{C}$.

Define a cosimplicial object $\eqsd \from \Delta \to \msSet$ which sends $[n]$ to the preorder on the set
\[ \{ A \subseteq [n] \mid A \neq \varnothing \} \]
where
\[ A_0 \leq A_1 := \max A_0 \leq \max A_1. \]
We view this preorder as a simplicial set marked at $E^1$-equivalences (i.e.~a 1-simplex $A_0 \leq A_1$ is marked if $\max A_0 = \max A_1$).
Let $\eqSd \from \sSet \to \msSet$ denote its extension by colimits.
Note the commutative triangle of cosimplicial objects
\[ \begin{tikzcd}
    \msd {[n]} \ar[rr, hook] \ar[rd, "\max"'] & {} & \eqsd {[n]} \ar[ld, "\max"] \\
    {} & {[n]} & {}
\end{tikzcd} \]
induces a commutative triangle
\[ \begin{tikzcd}
    \mSd X \ar[rr, hook] \ar[rd, "\msub"'] & {} & \eqSd X \ar[ld, "\msub"] \\
    {} & \minm{X} & {}
\end{tikzcd} \]
where each map is natural in $X$.
\begin{proposition} \label{weq-sd-triangle}
    For a simplicial set $X$, 
    \begin{enumerate}
        \item the map $\msub \from \eqSd X \to \minm{X}$ is an $E^1$-homotopy equivalence;
        \item each map in the triangle
        \[ \begin{tikzcd}
            \mSd X \ar[rr, hook] \ar[rd, "\msub"'] & {} & \eqSd X \ar[ld, "\msub"] \\
            {} & \minm{X} & {}
        \end{tikzcd} \]
        is a weak equivalence.
    \end{enumerate}
\end{proposition}
\begin{proof}
    \begin{enumerate}
        \item The map $f \from [n] \to \eqsd [n]$ defined by $i \mapsto \{ i \}$ commutes with faces and degeneracies.
        Moreover, for $A \in \eqsd [n]$, we have that $A \leq \{ \max A \}$ is an equivalence.
        This induces a homotopy $\eqsd [n] \times E[1] \to \eqsd [n]$ from the identity to $f \circ \max$ which also commutes with faces and degeneracies.
        Thus, this map induces a section $\minm{X} \to \eqSd X$ to the map $\msub \from \eqSd X \to \minm{X}$, along with a homotopy $\eqSd X \times E^1 \to \eqSd X$ from the identity to $f_! \circ \msub$.
        
        \item \cref{max-weq} shows the map $\max \from \msd [n] \to [n]$ is a homotopy equivalence.
        By \cite[Prop.~3.1.14]{cisinski:higher-categories}, the left map in the triangle is a weak equivalence.
        The right map is a weak equivalence by (1).
        Thus, the top map is by 2-out-of-3.
    \end{enumerate}
\end{proof}

\begin{corollary} \label{weq-sd-triangle-join-simp0}
    For a simplicial set $X$, each map in the triangle
    \[ \begin{tikzcd}
        (\mSd X) \join \simp{0} \ar[rr, hook] \ar[rd, "\msub \join \simp{0}"'] & {} & (\eqSd X) \join \simp{0} \ar[ld, "\msub \join \simp{0}"] \\
        {} & \minm{X} \join \simp{0} & {}
    \end{tikzcd} \]
    is a weak equivalence.
\end{corollary}
\begin{proof}
    As in the proof of \cref{weq-sd-triangle}, when $X = \minm{(\simp{n})}$, these maps are homotopy equivalences by \cref{preorder-marked-htpy-join}.
    Applying \cite[Prop.~3.1.14]{cisinski:higher-categories} in the slice model category $\simp{0} \slice \msSet$ gives the desired result.
\end{proof}

Recall (cf.~\cite[Def.~5.3.1.7]{lurie:htt}) that a simplicial set $X$ is \emph{filtered} if every map $K \to X$ from a finite simplicial set $K$ (i.e.~$K$ has finitely many non-degenerate simplices) admits a lift to a map $K \join \simp{0} \to X$.
Dually, $X$ is \emph{cofiltered} if every such map $K \to X$ admits a lift to a map $\simp{0} \join K \to X$.
\begin{lemma} \label{filtered-lemma}
    If a quasicategory has the right lifting property against the set of maps
    \[ \{ \Sd \bd \simp{n} \ito \Sd \simp{n} \mid n \geq 0 \} \]
    then it is filtered.
\end{lemma}
\begin{proof}
    Let $\mc{C}$ be a quasi-category with the right lifting property against $\{ \Sd \bd \simp{n} \ito \Sd \simp{n} \mid n \geq 0 \}$.
    By \cite[Lem.~5.3.1.12]{lurie:htt}, it suffices to check that every map $\bd \simp{n} \to \mc{C}$ admits a lift to a map $\bd \simp{n} \join \simp{0} \to \mc{C}$.
    That is, it suffices to show $\mc{C}$ has the right lifting property against the set of maps
    \[ \{ \bd \simp{n} \ito (\bd \simp{n}) \join \simp{0} \mid n \geq 0 \}. \]
    
    The square
    \[ \begin{tikzcd}
        \mSd \bd \simp{n} \ar[r, hook] \ar[d, hook] & \eqSd \bd \simp{n} \ar[d, hook] \\
        (\mSd \bd \simp{n}) \join \simp{0} \ar[r, hook] & (\eqSd \bd \simp{n}) \join \simp{0}
    \end{tikzcd} \]
    induces a map from the pushout, which we denote by $f \from P \to (\eqSd \bd \simp{n}) \join \simp{0}$.
    The horizontal maps are weak equivalences by \cref{weq-sd-triangle} and \cref{weq-sd-triangle-join-simp0}, respectively.
    Thus, $f$ is a weak equivalence by 2-out-of-3.
    One verifies $f$ is a monomorphism, thus it is an acyclic cofibration.

    As $\mc{C}$ has the right lifting property against the set $\{ \Sd \bd \simp{n} \ito \Sd \simp{n} \mid n \geq 0 \}$, its maximal marking $\maxm{\mc{C}}$ has the right lifting property against the set of maps
    \[ \{ \mSd \bd \simp{n} \ito (\mSd \bd \simp{n}) \join \simp{0} \mid n \geq 0 \} \]
    by adjointness (as $\Sd \simp{n} \cong U (\mSd \bd \simp{n} \join \simp{0})$).
    Thus, it has the right lifting property against the map $\eqSd \bd \simp{n} \ito P$.

    Given a map $u \from \bd \simp{n} \to \mc{C}$, this gives that the composite
    \[ \eqSd \bd \simp{n} \xrightarrow{\msub} \minm{(\bd \simp{n})} \xrightarrow{\minm{u}} \minm{\mc{C}} \ito \maxm{\mc{C}} \]
    lifts to a map
    \[ v \from P \to \maxm{\mc{C}}. \]
    As this composite sends marked 1-simplices to degeneracies and $\eqSd \bd \simp{n} \ito P$ is surjective on 1-simplices, we may view $v$ as a map $P \to \natm{\mc{C}}$.

    The map $P \to (\eqSd \bd \simp{n}) \join \simp{0}$ is an acyclic cofibration and $\natm{\mc{C}}$ is fibrant, hence $v$ lifts to a map 
    \[ \overline{v} \from (\eqSd \bd \simp{n}) \join \simp{0} \to \natm{\mc{C}}. \]
    Applying naturality to the section $i \from \minm{(\bd \simp{n})} \to \eqSd \bd \simp{n}$ defined in \cref{weq-sd-triangle} gives a commutative square:
    \[ \begin{tikzcd}
        \minm{(\bd \simp{n})} \ar[r, "i"] \ar[d, hook] & \eqSd \bd \simp{n} \ar[d, hook] \\
        \minm{(\bd \simp{n})} \join \simp{0} \ar[r, "i \join \simp{0}"] & (\eqSd \bd \simp{n}) \join \simp{0}
    \end{tikzcd} \]
    The (underlying) map $\overline{v} \circ (i \join \simp{0})$ is the desired lift.
    \[ \begin{tikzcd}
        \minm{(\bd \simp{n})} \ar[r, "i"] \ar[d, hook] & \eqSd \bd \simp{n} \ar[r, "\msub"] \ar[d, hook] & \minm{(\bd \simp{n})} \ar[r, "\minm{u}"] & \minm{\mc{C}} \ar[r, hook] & \natm{\mc{C}} \\
        \minm{(\bd \simp{n})} \join \simp{0} \ar[r, "i \join \simp{0}"] & (\eqSd \bd \simp{n}) \join \simp{0} \ar[urrr, "\overline{v}"']
    \end{tikzcd} \]
\end{proof}
\begin{corollary} \label{filtered-equiv-fibr-replacement-contr}
    Let $\mc{C}$ be a quasicategory.
    \begin{enumerate}
        \item If $\maxm{\mc{C}}$ satisfies CLF then $\mc{C}$ is filtered if and only if the Kan fibrant replacement of $\mc{C}$ is contractible.
        \item If $\maxm{\mc{C}}$ satisfies CRF then $\mc{C}$ is cofiltered if and only if the Kan fibrant replacemnt of $\mc{C}$ is contractible.
    \end{enumerate}
\end{corollary}
\begin{proof}
    We prove (1), as (2) is formally dual.

    For the forward direction, apply \cite[Lem.~5.3.1.18]{lurie:htt}.
    For the converse, \cref{mEx-computes-localization} shows that $\mEx \maxm{\mc{C}} \cong \Ex \mc{C}$ is a Kan fibrant replacement of $\mc{C}$.
    By adjointness, $\Ex \mc{C}$ is contractible if and only if $\mc{C}$ has the right lifting property with respect to the set of maps $\{ \Sd \bd \simp{n} \ito \Sd \simp{n} \mid n \geq 0 \}$.
    This then follows from \cref{filtered-lemma}.
\end{proof}
We can now show that if $W$ is closed under 2-out-of-3 then the marked slice category is filtered.
\begin{corollary} \label{lpathloop-filtered}
    Let $x$ be an object in a marked quasicategory $(\mc{C}, W)$ such that $W$ is closed under 2-out-of-3.
    \begin{enumerate}
        \item The underlying simplicial set $U(\lpathloop{x})$ of the marked slice under $x$ is filtered.
        \item The underlying simplicial set $U(\rpathloop{x})$ of the marked slice over $x$ is cofiltered.
    \end{enumerate}
\end{corollary}
\begin{proof}
    We prove (1), as (2) is formally dual.

    As $W$ is closed under 2-out-of-3, every 1-simplex of $\lpathloop{x}$ is marked.
    That is, $\lpathloop{x} = \maxm{U (\lpathloop{x})}$.
    As $\lpathloop{x}$ satisfies CLF by \cref{pathloop-qcat-clf,filtered-equiv-fibr-replacement-contr}, it suffices to show the Kan fibrant replacement of $U (\lpathloop{x})$ is contractible.
    This follows as $\lpathloop{x}$ is a quasicategory with an initial object $(x, x \xrightarrow{x\degen{}{1}} x)$.
\end{proof}

From this, it follows that the localization functor $\mc{C} \to \mEx(\mc{C}, W)$ preserves colimits.
The proofs presented here follow the arguments given in \cite[Ch.~1]{gabriel-zisman}.
\begin{lemma} \label{mEx-coeq-equiv}
    Let $(\mc{C}, W)$ be a marked quasicategory.
    \begin{enumerate}
        \item If $(\mc{C}, W)$ satisfies CLF then every co-equalizer diagram $\simp{1} \push_{\bd \simp{1}} \simp{1} \to \mEx(\mc{C}, W)$ is $E^1$-homotopic to one which factors through $\msup \from \mc{C} \to \mEx(\mc{C}, W)$.
        \item If $(\mc{C}, W)$ satisfies CRF then every equalizer diagram $\simp{1} \push_{\bd \simp{1}} \simp{1} \to \mExop(\mc{C}, W)$ is $E^1$-homotopic to one which factors through $\minsup \from \mc{C} \to \mExop(\mc{C}, W)$.
    \end{enumerate}
\end{lemma}
\begin{proof}
    We prove (1), as (2) is formally dual.

    Fix a co-equaliser diagram $F$, which we view as a map $\mSd \simp{1} \push_{\{ 0, 1 \}} \mSd \simp{1} \to (\mc{C}, W)$
    as in the diagram:
    
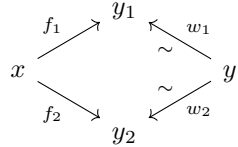
\[ \begin{tikzcd}[row sep = 1em]
        {} & y_1 & {} \\
        x \ar[ur, "f_1"] \ar[rd, "f_2"'] & {} & y \ar[ul, "w_1"', "\sim"] \ar[dl, "w_2", "\sim"'] \\
        {} & y_2 & {}
    \end{tikzcd} \]
    As $(\mc{C}, W)$ satisfies CLF, the span $(w_1, w_2)$ may be completed to a commutative square as in the diagram:
    
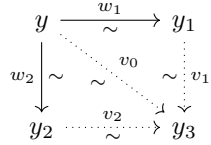
\[ \begin{tikzcd}[sep = large]
        y \ar[r, "w_1", "\sim"'] \ar[d, "w_2"', "\sim"] \ar[rd, "v_0", dotted, "\sim"'] & y_1 \ar[d, "v_1", dotted, "\sim"'] \\
        y_2 \ar[r, "v_2", dotted, "\sim"'] & y_3
    \end{tikzcd} \]
    Let $v_1 f_1 \from x \to y_3$ be a composite of $f_1$ and $v_1$ and let $v_2 f_2 \from x \to y_3$ be a composite of $f_2$ and $v_2$.
    The diagram
    
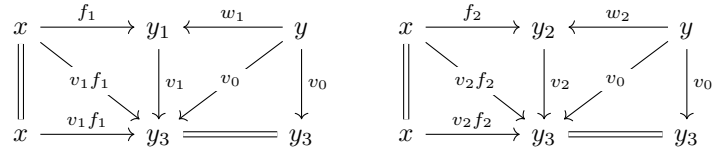
\[ \begin{tikzcd}[sep = large]
        x \ar[r, "f_1"] \ar[d, equal] \ar[rd, "v_1 f_1" description] & y_1 \ar[d, "v_1"] & y \ar[l, "w_1"'] \ar[d, "v_0"] \ar[ld, "v_0"' description] \\
        x \ar[r, "v_1 f_1"] & y_3 & y_3 \ar[l, equal]
    \end{tikzcd} \qquad \begin{tikzcd}[sep = large]
        x \ar[r, "f_2"] \ar[d, equal] \ar[rd, "v_2 f_2" description] & y_2 \ar[d, "v_2"] & y \ar[l, "w_2"'] \ar[d, "v_0"] \ar[ld, "v_0"' description] \\
        x \ar[r, "v_2 f_2"] & y_3 & y_3 \ar[l, equal]
    \end{tikzcd} \]
    depicts a marked homotopy $(\mSd \simp{1} \push_{\{ 0, 1 \}} \mSd \simp{1}) \times \maxm{(\simp{1})} \to (\mc{C}, W)$ from $F$ to $[v_1 f_1, v_2 f_2] \circ \msub$ (note the outer vertical 1-simplices in each component of the diagram are identified). 
    This induces an $E^1$-homotopy in $\mEx(\mc{C}, W)$ from $F$ to $\msup \circ [v_1 f_1, v_2 f_2]$.
\end{proof}

With this, we may state a complete quasicategorical analogue of \cref{classical-fractions-colims}, giving a criterion for the existence of colimits in the localization.
\begin{theorem} \label{localization-has-limits}
    Let $(\mc{C}, W)$ be a marked quasicategory where $W$ is closed under 2-out-of-3.
    \begin{enumerate}
        \item If $(\mc{C}, W)$ satisfies CLF then:
        \begin{enumerate}
            \item for any finite colimit cone $\lambda \from K \join \simp{0} \to \mc{C}$, the composite 
            \[ \msup \circ \lambda \from K \join \simp{0} \to \mEx(\mc{C}, W) \] 
            is a colimit cone; and
            \item if $\mc{C}$ admits all finite colimits then $\mEx(\mc{C}, W)$ admits all finite colimits.
        \end{enumerate} 
        \item If $(\mc{C}, W)$ satisfies CRF then:
        \begin{enumerate}
            \item for any finite limit cone $\lambda \from K \join \simp{0} \to \mc{C}$, the composite 
            \[ \minsup \circ \lambda \from K \join \simp{0} \to \mExop(\mc{C}, W) \] 
            is a limit cone; and
            \item if $\mc{C}$ admits all finite limits then $\mExop(\mc{C}, W)$ admits all finite limits.
        \end{enumerate}
    \end{enumerate}
\end{theorem}
\begin{proof}
    We prove (1), as (2) is formally dual.

    For (a), let $F$ denote the restriction $\restr{\lambda}{K} \from K \to \mc{C}$ of $\lambda$ to $K$ and let $\colim F$ denote the value of $\lambda$ at the cone point.
    It suffices to show $\colim F$, as an object in $\mEx(\mc{C}, W)$, represents the functor of cones under $\msup \circ F$.
    That is, for every object $x \in \mEx(\mc{C}, W)$, we show that the canonical map
    \[ \varphi \from \mEx(\mc{C}, W)(\colim F, x) \to \lim \left( K^\op \xrightarrow{\msup \circ F} \mEx(\mc{C}, W)^\op \xrightarrow{\mEx(\mc{C}, W)(\uvar, x)} \Sp \right) \]
    is an equivalence of Kan complexes.

    Let $G \from K^\op \times U(\lpathloop{x}) \to \Sp$ denote the composite map
    \[ K^\op \times U(\lpathloop{x}) \xrightarrow{F^\op \times \Pi_x} \mc{C}^\op \times \mc{C} \xrightarrow{\mc{C}(-,-)} \Sp. \]
    We show there is a natural $E^1$-homotopy from $\varphi$ to the canonical map
    \[ \colim_{U(\lpathloop{x})} (\lim_{K^\op} G) \to \lim_{K^\op} ( \colim_{U(\lpathloop{x})} G ). \]
    As $U(\lpathloop{x})$ is filtered (\cref{lpathloop-filtered}), the result would then follow since filtered colimits commute with finite limits (cf.~(\cite[Prop.~5.3.3.3]{lurie:htt})).
    
    Comparing the codomains of $\varphi$ and the canonical map, the diagram $\colim_{U(\lpathloop{x})} G \from K^\op \to \Sp$ is naturally $E^1$-homotopic to the diagram
    \[ K^\op \xrightarrow{\msup \circ F} \mEx(\mc{C}, W) \xrightarrow{\mEx(\mc{C}, W)(-, x)} \Sp \]
    by \cref{mapping-space-is-colim}.
    Comparing the domains, the diagram $\lim_{K^\op} G \from U(\lpathloop{x}) \to \Sp$ is naturally $E^1$-homotopic to the diagram 
    \[ U(\lpathloop{x}) \xrightarrow{\Pi_x} \mc{C} \xrightarrow{\mc{C}(\colim F, -)} \Sp \] 
    as the Yoneda embedding $Y \from \mc{C}^\op \to \fcat{\mc{C}}{\Sp}$ takes $\colim F = \lim (F^\op)$ to $\lim(Y \circ F^\op)$ (cf.~\cite[Prop.~5.1.3.2]{lurie:htt}).
    Thus, the colimit of $\lim_{K^\op} G$ is naturally equivalent to the mapping space $\mEx(\mc{C}, W)(\colim F, x)$ by \cref{mapping-space-is-colim}.

    For (b), this follows from (a) by \cref{mEx-coeq-equiv}.
\end{proof}

The formula presented in \cref{mapping-space-is-colim} in particular implies that if a 1-category satisfies CLF then its localization as an $\infty$-category is already a 1-category. 
\begin{corollary} \label{1-cat-clf-localization-is-1-cat}
    Let $(\mc{C}, W)$ be a marked 1-category.
    \begin{enumerate}
        \item If $(\mc{C}, W)$ satisfies CLF then the map
        \[ \mEx(\mc{C}, W) \to \Ho \mEx(\mc{C}, W) \]
        is a categorical equivalence.
        \item If $(\mc{C}, W)$ satisfies CRF then the map
        \[ \mExop(\mc{C}, W) \to \Ho \mExop(\mc{C}, W) \]
        is a categorical equivalence.
    \end{enumerate} 
\end{corollary}
\begin{proof}
    These maps are equivalences on the homotopy category.
    By \cref{mapping-space-is-colim} and \cref{lpathloop-filtered}, mapping spaces in the localization are filtered colimits of discrete $\infty$-groupoids, hence discrete.
    This shows the map is an equivalence (in fact, an isomorphism) on mapping spaces.
\end{proof}

We also use \cref{localization-has-limits} to deduce a characterization of CLF in terms of properties of the localization.
Recall a collection of weak equivalences $W$ is \emph{saturated} if every morphism that becomes an equivalence under $\mc{C} \to \mc{C}[W^{-1}]$ is already in $W$.
\begin{corollary} \label{clf-iff-fin-lim-preserve}
    Let $(\mc{C}, W)$ be a marked quasicategory which is finitely cocomplete. 
    Suppose $W$ is saturated.
    Then,
    \begin{enumerate}
        \item the pair $(\mc{C}, W)$ satisfies CLF if and only if $\mc{C}[W^{-1}]$ admits finite colimits and the localization functor $\mc{C} \to \mc{C}[W^{-1}]$ preserves them.
        \item the pair $(\mc{C}, W)$ satisfies CRF if and only if $\mc{C}[W^{-1}]$ admits finite limits and the localization functor $\mc{C} \to \mc{C}[W^{-1}]$ preserves them.
    \end{enumerate}
\end{corollary}
\begin{proof}
    Note that $W$ is closed under 2-out-of-3 since it saturated.
    For both (1) and (2), the forward direction is \cref{localization-has-limits}, and the reverse direction is \cref{limits-create-crf}.
\end{proof}

We conclude the paper by proving a generalization of the fact that if an abelian category satisfies both CLF and CRF, then its localization is again abelian \cite{gabriel-zisman}.
The generalization to the context of quasicategory theory is straightforward; indeed, the only change necessary is to replace abelian categories with stable quasicategories, which we now define.

\begin{definition}
    An $\infty$-category $\mc{C}$ is \emph{stable} if
    \begin{enumerate}
        \item there is an object $0 \in \mc{C}$ which is both initial and terminal;
        \item $\mc{C}$ has all finite limits and colimits;
        \item a square of the form
        \[ \begin{tikzcd}
            x \ar[r] \ar[d] & y \ar[d] \\
            0 \ar[r] & z
        \end{tikzcd} \]
        in $\mc{C}$ is a pullback if and only if it is a pushout.
    \end{enumerate}
\end{definition}
\begin{theorem}[cf.~{\cite[p.~18--19]{gabriel-zisman}}] \label{stable_clf_crf}
    Let $(\mc{C}, W)$ be a marked quasicategory where $\mc{C}$ is stable and $W$ is closed under 2-out-of-3.
    If $(\mc{C}, W)$ satisfies CLF and CRF then the localization $\mEx(\mc{C}, W)$ is stable.
\end{theorem}
\begin{proof}
    Axioms (1) and (2) follow from \cref{localization-has-limits}.
    It remains to show axiom 3.

    Suppose
    \[ \begin{tikzcd}[column sep = 1.9em]
        x \ar[r] & y' & y \ar[l, "\sim"'] \ar[r] & z' & z \ar[l, "\sim"']
    \end{tikzcd} \]
    is a fiber sequence in $\mEx(\mc{C}, W)$.
    Using CLF, we construct a diagram
    \[ \begin{tikzcd}
        y \ar[r] \ar[d, "\sim"'] & z' \ar[d, "\sim", dotted] \\
        y' \ar[r, dotted] & z''
    \end{tikzcd} \]
    in $\mc{C}$.
    The diagram
    \[ \begin{tikzcd}[column sep = 1.9em]
        x \ar[r] \ar[d, equal] & y' \ar[d, equal] & y \ar[l, "\sim"'] \ar[r] \ar[d, "\sim"] & z' \ar[d, "\sim"] & z \ar[l, "\sim"'] \ar[d, "\sim"] \\
        x \ar[r] & y' & y' \ar[r] \ar[l, equal] & z'' & z'' \ar[l, equal]
    \end{tikzcd} \]
    exhibits the bottom row as a fiber sequence in $\mEx(\mc{C}, W)$, since the top sequence is a fiber sequence and the vertical maps are inverted in $\mEx(\mc{C}, W)$ (note the map $z \to z''$ is given by composition of weak equivalences).
    The bottom row is the image of the diagram
    \[ x \to y' \to z'' \]
    under $\msup \from \mc{C} \to \mEx(\mc{C}, W)$.
    This map preserves limits since it is a model for the localization and $(\mc{C}, W)$ satisfies CRF (\cref{localization-has-limits}).
    Denoting the fiber of $y' \to z''$ in $\mc{C}$ by $k \in \mc{C}$, there are equivalences of sequences
    \[ \begin{tikzcd}
        x \ar[r] \ar[d, phantom, "\simeq" marking] & y' \ar[d, phantom, "\simeq" marking] & y  \ar[d, phantom, "\simeq" marking] \ar[l] \ar[r] & z'  \ar[d, phantom, "\simeq" marking] & z \ar[d, phantom, "\simeq" marking] \ar[l] \\
        x \ar[d, phantom, "\simeq" marking] \ar[r] & y' \ar[r, equal] \ar[d, phantom, "\simeq" marking] & y' \ar[r] \ar[d, phantom, "\simeq" marking] & z'' \ar[r, equal] \ar[d, phantom, "\simeq" marking] & z'' \ar[d, phantom, "\simeq" marking] \\
        k \ar[r] & y' \ar[r, equal] & y' \ar[r] & z'' \ar[r, equal] & z'' 
    \end{tikzcd} \]
    in $\mEx(\mc{C}, W)$.
    Since $k \to y' \to z''$ is a fiber sequence in $\mc{C}$, it is also a cofiber sequence.
    As $(\mc{C}, W)$ satisfies CLF, the map $\msup$ also preserves colimits, hence the bottom row is a cofiber sequence.
    Thus, the top row is a cofiber sequence.
    
    This shows every fiber sequence in $\mEx(\mc{C}, W)$ is a cofiber sequence.
    The reverse implication follows from the forward implication instantiated at $\mExop(\mc{C}^\op, W^\op)$.
\end{proof}





\bibliographystyle{amsalphaurlmod}
\bibliography{all-refs.bib}

\end{document}